\documentclass[11pt]{amsart}
\usepackage{amsmath,amsthm,amssymb,bm}
\usepackage{xcolor}
\usepackage{indentfirst}
\usepackage{graphicx}
\usepackage{tikz}
\usetikzlibrary{decorations.pathreplacing, decorations.pathmorphing, decorations.shapes}
\usetikzlibrary{arrows,calc}
\usepackage{mathrsfs,esint}
\usepackage{comment}
\usepackage{enumerate}
\usepackage{hyperref}
\hypersetup{
    colorlinks = true,
    linkcolor = red,
    anchorcolor = green!50!black,
    citecolor = green!50!black,
    filecolor = green!50!black,
    urlcolor = blue
}
\usepackage{cleveref}

\usepackage[top=1in, bottom=1in, left=1.25in, right=1.25in]{geometry}
\usepackage{setspace}

\usepackage{float}
\usepackage[ruled, vlined]{algorithm2e}

\usepackage{subcaption}
\usepackage{pgfplots}
\pgfplotsset{compat=newest, compat/show suggested version=false}
\usepackage{tikz}
\usetikzlibrary{intersections}
\usepackage{longtable}

\usepackage[titletoc,toc,title]{appendix}
\renewcommand\appendixname{Appendix}

\newtheorem{thm}{Theorem}[section]

\newtheorem{coro}[thm]{Corollary}
\newtheorem{lem}[thm]{Lemma}
\newtheorem{remark}[thm]{\textit{Remark}}
\newtheorem{defn}[thm]{Definition}

\newcommand{\beq}{\begin{equation}}
\newcommand{\eeq}{\end{equation}}

\newcommand{\dist}{\text{dist}}
\newcommand{\argmin}{\text{argmin}}

\def\cC{\mathcal{C}}

\def\cE{\mathcal{E}}

\def\cH{\mathcal{H}}
\def\cI{\mathcal{I}}

\def\cL{\mathcal{L}}

\def\cO{\mathcal{O}}
\def\cP{\mathcal{P}}

\def\L{\mathbb{L}}

\def\N{\mathbb{N}}
\def\O{\mathbb{O}}

\def\R{\mathbb{R}}

\def\Z{\mathbb{Z}}

\def\eb{\bm e}

\def\vb{\bm v}

\def\xb{\bm x}
\def\yb{\bm y}
\def\zb{\bm z}

\def\xib{\bm\xi}

\newcommand{\al}{\alpha}

\newcommand{\ga}{\gamma}
\newcommand{\del}{\delta}

\newcommand{\la}{\lambda}

\newcommand{\sig}{\sigma}

\newcommand{\om}{\omega}

\newcommand{\La}{\Lambda}

\newcommand{\Om}{\Omega}

\title[Monotone meshfree methods for linear elliptic equations]{Monotone meshfree methods for linear elliptic equations in non-divergence form via nonlocal relaxation}

\thanks{}

\author{Qihao Ye}
\address{Department of Mathematics, University of California, San Diego, CA 92093, United States} 
\email{q8ye@ucsd.edu}

\author{Xiaochuan Tian}
\address{Department of Mathematics, University of California, San Diego, CA 92093, United States} 
\email{xctian@ucsd.edu}
               
\date{}

\renewcommand \theequation {\arabic{section}.\arabic{equation}}
\begin{document}

\begin{abstract}
We design a monotone meshfree finite difference method for linear elliptic equations in the non-divergence form on point clouds via a nonlocal relaxation method.
The key idea is a novel combination of a nonlocal integral relaxation of the PDE problem with a robust meshfree discretization on point clouds. 
Minimal positive stencils are obtained through a local $l_1$-type optimization procedure that automatically guarantees the stability and, therefore, the convergence of the meshfree discretization for linear elliptic equations.
A major theoretical contribution is the existence of consistent and positive stencils for a given point cloud geometry. We provide sufficient conditions for the existence of positive stencils by finding neighbors within an ellipse (2d) or ellipsoid (3d) surrounding each interior point, generalizing the study for Poisson's equation by Seibold (Comput Methods Appl Mech Eng 198(3-4):592-601, 2008). It is well-known that wide stencils are in general needed for constructing consistent and monotone finite difference schemes for linear elliptic equations. 
Our result represents a significant improvement in the stencil width estimate for positive-type finite difference methods for linear elliptic equations in the near-degenerate regime (when the ellipticity constant becomes small), compared to previously known works in this area.
Numerical algorithms and practical guidance are provided with an eye on the case of small ellipticity constant. 
At the end, we present numerical results for the performance of our method in both 2d and 3d, examining a range of ellipticity constants including the near-degenerate regime.
\end{abstract}

\subjclass[2020]{Primary 65N06, 65N12,65N35, 35B50, 35J15; Secondary 65R20, 35J70, 45A05}

\keywords{Meshfree method, Minimal positive stencil, Monotone scheme, Discrete maximum principle, Linear elliptic PDEs, Near-degenerate, Non-divergence form, Nonlocal relaxation, Linear minimization problem, Point cloud, Super-convergence}

\maketitle
\begin{center}
{\small \color{purple} This work has been published in \emph{Journal of Scientific Computing}. Please refer to \href{https://doi.org/10.1007/s10915-023-02294-3}{the official publication} for citation.}
\end{center}
\onehalfspacing



\section{Introduction}

In this work, we consider numerical approximations to the second-order elliptic equations in non-divergence form
\beq
\label{eq:elliptic}
\left\{
\begin{aligned}
-L u(\xb) := - \sum_{i,j=1}^d a^{ij}(\xb) \partial_{ij} u (\xb) &= f(\xb) \qquad    \xb \in \Om \\
u(\xb)& = g(\xb) \qquad    \xb \in \partial\Om,  
\end{aligned}
\right. 
\eeq
for a bounded domain $\Om \subset \R^d$. $A(\xb) = ( a^{ij}(\xb) )_{i,j=1}^d$ is a bounded and measurable matrix-valued function and is assumed to be symmetric and positive definite satisfying the uniform
ellipticity condition
\beq
\la |\xib|^2 \leq \xib^{T} A(\xb)\xib \leq \La |\xib|^2 \qquad \forall \xib\in \R^d, \forall\xb\in\Om
\eeq
for positive constants $\la$ and $\La$. Notice that by dividing both sides of the first equation in \eqref{eq:elliptic} by $\La$, we can assume without loss of generality that
$\La =1$ and the ratio $\varrho: =\la/\La =\lambda \leq 1$. 
Notice that if the coefficient matrix $A$ is differentiable, then the non-divergence form equation \eqref{eq:elliptic} can be recast into a divergence form convection-diffusion equation. 
However, when $A$ is not differentiable, such reformulation no longer exists.
In this work, we only assume $A$ is bounded and measurable so that a variational formulation is not available.

Linear elliptic equations in non-divergence form arise in various applied domains 
including probability and stochastic processes \cite{fleming2006controlled}. 
They are also recognized as the linearization of fully nonlinear PDEs such as the Monge-Amp\`ere equation, which arises in applications to the optimal transportation problem and geometry  \cite{caffarelli1997properties,le2017dynamical,trudinger2008monge},   
and Hamilton-Jacobi-Bellman (HJB) equations with applications to stochastic optimal control and finance \cite{fleming2006controlled}. 
In many practical scenarios, the coefficient matrix $A(\boldsymbol{x})$ lacks smoothness or continuity, and therefore cannot be recast into a divergence form \cite{cabre2008elliptic,caffarelli1997properties,feng2021narrow,fleming2006controlled}.  
PDE theories for non-divergence form elliptic equations are well-established in the literature.
Existence, uniqueness and regularity theories are established for different notions of solutions including
classical solutions, strong solutions, and viscosity solutions \cite{crandall1992user,gilbarg1977elliptic,maugeri2000elliptic,safonov1999nonuniqueness}.

In terms of numerical methods, non-divergence form elliptic PDEs are much less discussed than the divergence form PDEs because of a lack of variational formulation aforementioned. 
Discussions on finite element methods can be found in \cite{feng2017finite,feng2018interior,lakkis2011finite,neilan2017numerical,nochetto2018discrete,smears2013discontinuous} and the references therein.
We pursue the direction of the positive-type finite difference method for elliptic PDEs which has guaranteed convergence as a result of consistency and monotonicity/stability.     
Such pursuit dates back to Motzkin and Wasow \cite{motzkin1952approximation}, Kuo and Trudinger \cite{kuo1990linear}, and Kocan \cite{kocan1995approximation}.   
What was found in these works was that consistent and positive-type finite difference schemes exist for a given elliptic operator, but the 
stencil size grows with $\varrho \to 0$, and therefore the so-called wide stencil is a necessary feature of monotone finite difference methods
even in the case of linear elliptic equations. 
Wide-stencil methods are later developed also for fully nonlinear elliptic PDEs \cite{froese2011convergent,oberman2008wide,nochetto2019pointwise,nochetto2019two}, and a few recent developments on monotone finite difference methods for nonlinear elliptic PDEs are found in \cite{finlay2019improved,froese2018meshfree,hamfeldt2022convergent}.
For the linear elliptic PDEs, Kocan \cite{kocan1995approximation} gives an estimate of the stencil width for the existence of positive-type finite difference method, and it grows linearly with $\varrho^{-1}$ in 2d and superlinearly with $\varrho^{-1}$ in 3d, which severely impacts the practical use of the scheme for small ratio $\varrho$. In a more recent work \cite{mirebeau2016minimal}, a positive-type finite difference scheme with anisotropic stencils is discussed but their stencil width estimate in 2d is also at the scale of $\varrho^{-1}$.

Our work is inspired by the recent development of meshfree methods for nonlocal models \cite{DGLZ12,fan2022meshfree,LTTF21,trask2019asymptotically}.
The fundamental new idea of numerical approximation to \eqref{eq:elliptic} is a continuous nonlocal relaxation of the PDE problem followed by asymptotically compatible discretizations. More specifically, we first define a family of nonlocal operators $\{ \cL_\del \}_{\del>0}$ as continuous approximations to the elliptic operator $L$, and then seek for robust discretizations $ \{ \cL_\del^h \}_{\del>0, h>0}$ by which the discretization parameter $h$ is allowed to be proportional to the nonlocal length parameter $\del$, i.e., $h = \Theta(\delta)$, for the convergence to the elliptic problem \eqref{eq:elliptic}. 
In our work, $\cL_\del u(\xb)$ is defined by an integral over an elliptical region depending on $A(\xb)$ where $\del>0$ denotes the bound of the semi-major axis of it. 
We note that a similar integral approximation to the linear elliptic operator is used in \cite{nochetto2018discrete}, while a non-robust discretization is employed so that a stronger relation, $h=o(\delta)$, is needed for convergence. In contrast, the idea of asymptotically compatible schemes \cite{TiDu14,trask2019asymptotically} could significantly improve the efficiency of numerical methods via nonlocal relaxation. 
We note that nonlocal integral relaxation to PDEs is a natural idea linked with meshfree discretization and has been utilized in numerical schemes such as the smoothed particle hydrodynamics (SPH) \cite{du2015integral,gingold1977smoothed,liu2010smoothed},  vortex methods \cite{beale1982vortex,cottet2000vortex} and other particle methods \cite{belytschko2000unified,craig2016blob,tornberg2003regularization}.
The basic idea of the meshfree discretization is an $l_1$-type local optimization method for obtaining minimal positive stencils on point clouds. 
The major theoretical contribution (cf.~\Cref{thm:nbh_criterion}) is an estimate of the elliptical searching region, parametrized by $\del>0$, for guaranteeing the existence of consistent and positive stencils. 
There are two main takeaways from \Cref{thm:nbh_criterion}. First, for a fixed ratio $\varrho>0$, $\delta$ can be chosen as a constant multiple of $h$, leading to discretizations of asymptotically compatible type. Second, for small $\varrho>0$, our theorem guarantees the existence of positive-type finite difference method within a stencil width of $\cO(\varrho^{-1/2})$, which is a substantial improvement of the known theoretical results in  \cite{kocan1995approximation,mirebeau2016minimal}. 
In terms of numerical algorithms, the size of the searching region determines the computational complexity of the local minimizations problems,
and the near-optimal elliptic searching area estimate allows a practical assembly process in both 2d and 3d.
In addition, the $l_1$-type minimization leads to minimal stencils with at most six points in 2d and ten points in 3d, thus the resulting linear system is sparse and can be solved efficiently by iterative methods.
One notable feature of our approach is that it is the first numerical method, as far as we are aware, that has been successfully implemented in both 2d and 3d for the solution of linear elliptic equations with a wide range of $\varrho\in (0,1]$, including the near-degenerate regime when $\varrho\ll 1$. 
We emphasize that our approach takes an innovative path that could have a potential impact on how meshfree and particle methods can be designed to enhance their mathematical properties as well as practical performances.
Indeed, many traditional meshfree methods and data analysis algorithms unanimously suffer the restrictive condition $h=o(\delta)$ on the two length-scales $h$ and $\delta$, see e.g., \cite{beale1982vortex,craig2016blob,garcia2020error,li2016convergent}.  Our work exemplifies the possible approach to improving the performance of meshfree and particle methods via nonlocal relaxation and robust discretizations. Extension of the approach to other interesting problems will be considered in the future.

The rest of the paper is organized as follows. We first define a family of $\del>0$ parameterized nonlocal integral approximations to \cref{eq:elliptic} in \Cref{sec:nonlocalrelaxation}. 
The consistency errors of the continuum nonlocal model to the elliptic equation with respect to $\delta$ are discussed.  
In \Cref{sec:meshfree}, we present the meshfree discretization based on the nonlocal regularized problem. An $l_1$-type local optimization
method is proposed in search of positive stencils in an elliptical neighborhood surrounding each node. 
A major theoretical result concerning the neighborhood criteria for guaranteeing the existence of consistent and positive stencils is presented in \Cref{thm:nbh_criterion}.
Convergence of numerical solutions in terms of the discretization parameter $h$ is then followed by monotonicity and the relation of $\delta$ and $h$ in \Cref{thm:nbh_criterion}. 
In \Cref{sec:algorithm}, we discuss point cloud generation and management, the assembling process, and
provide complexity analysis and practical guidance for the implementation of the numerical method. 
Since the feasibility condition in \Cref{sec:algorithm} contains an implicit constant $c=c(d)>0$, we estimate the constant numerically in this section and provide the searching neighborhood estimate that can be used in practice for $d=2$ and $d=3$.     
Finally, we present the 2d and 3d numerical results in \Cref{sec:num}, and make the conclusion and further discussions in \Cref{sec:conclusion}.


\section{Nonlocal relaxation to elliptic equations}
\label{sec:nonlocalrelaxation}
\renewcommand\tablename{Table}
\begin{center}
    \begin{longtable}{c | p{9.8cm}}
    \caption{Table for major notations.}\label{table:major_notations}\\
    \hline
    Symbol & Definition\\
    \hline
    $L$ & second-order elliptic operator\\
    $\Omega$ & a bounded domain\\
    $A(\boldsymbol{x})$ & coefficient matrix in $L$ such that $\lambda |\boldsymbol{\xi}|^{2} \leq \boldsymbol{\xi}^{T} A(\boldsymbol{x}) \boldsymbol{\xi} \leq \Lambda |\boldsymbol{\xi}|^{2}$\\
    $a^{i j}(\boldsymbol{x})$ & component of coefficient matrix $A(\boldsymbol{x})$\\
    $M(\boldsymbol{x})$ & $(A(\boldsymbol{x}))^{1 / 2}$\\
    $\varrho$ & ellipticity ratio $\lambda / \Lambda$\\
    $\delta$ & nonlocal regularization parameter\\
    $h$ & fill distance (discretization parameter)\\
    $\gamma$ & nonnegative kernel with $\displaystyle \int_{B_{1}(\boldsymbol{0})} |\boldsymbol{y}|^{2} \gamma(|\boldsymbol{y}|) \, d \boldsymbol{y} = 2 d$\\
    $\mathcal{E}_{\delta}^{\boldsymbol{x}}(\boldsymbol{z})$ & $\displaystyle \{ \boldsymbol{y} \in \mathbb{R}^{d} : M(\boldsymbol{x})^{- 1} (\boldsymbol{y} - \boldsymbol{z}) \in B_{\delta}(\boldsymbol{0}) \}$ (ellipse centered at $\boldsymbol{z}$)\\
    $\mathcal{C}_{\delta}^{\boldsymbol{v}}(\boldsymbol{x}_{i})$ & $\displaystyle \left \{ \boldsymbol{x} \in B_{\delta}(\boldsymbol{x}_{i}) : \boldsymbol{x}^{T} \boldsymbol{v} \geq \frac{1}{\sqrt{1 + \sigma_{d}}} |\boldsymbol{x}|^{2} \right \}$ (cone in $B_{\delta}(\boldsymbol{x}_{i})$)\\
    $\rho_{\delta}(\boldsymbol{x}, \boldsymbol{y})$ & $\frac{1}{\delta^{d + 2}} \gamma \left ( \frac{|M(\boldsymbol{x})^{- 1} \boldsymbol{y}|}{\delta} \right ) \det(M(\boldsymbol{x}))^{- 1}$\\
    $\widetilde{\mathcal{L}}_{\delta}$ & $\displaystyle \widetilde{\mathcal{L}}_{\delta} u(\boldsymbol{x}) = \int_{B_{\delta}(\boldsymbol{0})} \frac{1}{\delta^{d + 2}} \gamma \left ( \frac{|\boldsymbol{y}|}{\delta} \right ) (u(\boldsymbol{x} + \boldsymbol{y}) - u(\boldsymbol{x})) \, d \boldsymbol{y}$ \newline (nonlocal Laplace operator)\\
    $\widetilde{\mathcal{L}}_{\delta}^{h}$ & nonlocal approximate Laplace operator\\
    $\mathcal{L}_{\delta}$ & $\displaystyle \mathcal{L}_{\delta} u(\boldsymbol{x}) = \int_{\mathcal{E}_{\delta}^{\boldsymbol{x}}(\boldsymbol{0})} \rho_{\delta}(\boldsymbol{x}, \boldsymbol{y}) (u(\boldsymbol{x} + \boldsymbol{y}) - u(\boldsymbol{x})) \, d \boldsymbol{y}$ \newline (nonlocal elliptic operator)\\
    $\mathcal{L}_{\delta}^{h}$ & nonlocal approximate elliptic operator\\
    $\mathcal{L}_{\delta, \Omega}^{h}$ & nonlocal approximate elliptic operator in extended space\\
    $S_{\delta, h, p}$ & $\displaystyle \Big \{ \{ \omega_{j, i} \} : \omega_{j, i} \geq 0 \text{ and } \mathcal{L}_{\delta}^{h} u(\boldsymbol{x}_{i}) = \mathcal{L}_{\delta} u(\boldsymbol{x}_{i}) \forall u \in \mathcal{P}_{p}(\mathbb{R}^{d}) \Big \}$ \newline (feasible set)\\
    $\overline{S}_{\delta, h, p}$ & $\displaystyle \Big \{ \{ \omega_{j, i} \} : \omega_{j, i} \geq 0 \text{ and } \mathcal{L}_{\delta, \Omega}^{h} u(\boldsymbol{x}_{i}) = \mathcal{L}_{\delta} u(\boldsymbol{x}_{i}) \forall u \in \mathcal{P}_{p}(\mathbb{R}^{d}) \Big \}$\\[6pt]
    \hline
    \end{longtable}
\end{center}

In this section, we discuss the nonlocal integral approximation to \cref{eq:elliptic} on which our numerical methods are based. Major notations used in this paper are listed in \Cref{table:major_notations}.
\subsection{Nonlocal elliptic operators in non-divergence form}
Nonlocal models have gained much interest in recent years \cite{andreu2010nonlocal,bucur2016nonlocal,Du2019}.
In \cite{DDGG20,DGLZ12}, a nonlocal Laplace operator with a parameter dependence on $\delta >0$ is used as an approximation to the classical Laplace operator $\Delta$ (when $A(\xb) = I$). The nonlocal Laplace operator is given by 
\[
\widetilde{\cL}_\del u(\xb) = \int_{B_\del(\bm 0)} \frac{1}{\del^{d+2}}\gamma\left(\frac{|\yb|}{\del}\right) (u(\xb+\yb) - u(\xb)) d\yb
\]
where $\gamma$ is a nonnegative kernel with
\beq
\label{eq:kernel_secondmoment}
\int_{B_1(\bm 0)} |\yb|^2 \ga(|\yb|) d\yb = 2d. 
\eeq
It can be shown that as $\del\to0$, $\widetilde{\cL}_\del u (\xb) \to \Delta u(\xb)$ for a sufficiently smooth function $u$. 
Here we consider a more general nonlocal elliptic operator that approximates the classical elliptic operator $L$ in \cref{eq:elliptic} in the $\del\to0$ limit. 
Following \cite{nochetto2018discrete}, we define the nonlocal elliptic operator parameterized by $\del$ as
\beq
\label{NonlocalEllipticOp}
\cL_\del u(\xb) = \int_{\cE_\del^{\xb}(\bm{0})} \frac{1}{\del^{d+2}}\gamma\left(\frac{|M(\xb)^{-1}\yb|}{\del}\right) \det(M(\xb))^{-1} (u(\xb+\yb) - u(\xb)) d\yb 
\eeq 
where $M(\xb) := (A(\xb))^{1/2}$ is a positive definite matrix and $\cE_\del^{\xb}(\zb)$ denotes an ellipse with definition
\beq
\cE_\del^{\xb}(\zb) := \{ \yb \in \R^d: M(\xb)^{-1} (\yb-\zb) \in B_\del(\bm 0)\}. 
\eeq
\begin{figure}[htp]
    \centering
    \begin{tikzpicture}[scale=1]
    \coordinate (c1) at (-3, 0);
    \coordinate (c2) at (3, 0);

    \draw (c1) circle (1);
    \fill (c1) circle (1pt);
    \node at ($(c1) + (0.4, 0.45)$) {$\scriptstyle B_{\delta}(\boldsymbol{0})$};
    \node at ($(c1) - (0.15, 0.15)$) {$\scriptstyle \boldsymbol{0}$};

    \draw[rotate around={-60:(c2)}] (c2) ellipse (0.8018 and 1.5);
    \fill (c2) circle (1pt);
    \node at ($(c2) + (-0.5, -0.45)$) {$\scriptstyle \mathcal{E}_{\delta}^{\boldsymbol{x}}(\boldsymbol{z})$};
    \node at ($(c2) - (0.2, 0)$) {$\scriptstyle \boldsymbol{z}$};
    \draw[dashed] (c2) -- ($(c2) + (1.5*0.866, 1.5*0.5)$) node [below=2pt, right] {$\scriptstyle a\leq \del$}; 
    \draw[dashed] (c2) -- ($(c2) + (0.8018*0.5, -0.8018*0.866)$) node [below=3pt, right=-1pt] {$\scriptstyle  b\geq \del \sqrt{\varrho}$} ; 

    \draw[->, line width=0.3mm] ($(c1) + (1.5, 0)$) -- ($(c2) + (-1.5, 0)$);

    \end{tikzpicture}
    \caption{An illustration of $\mathcal{E}_{\delta}^{\boldsymbol{x}}(\boldsymbol{z})$ in 2d.}
    \label{fig:ellipse_definition}
\end{figure}
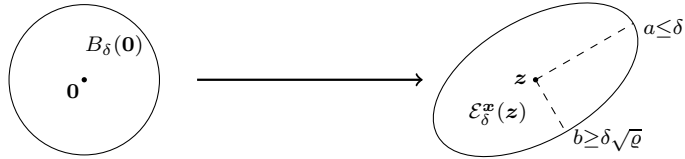
Notice that by our assumption on $A(\xb)$, $\cE_\del^{\xb}(\zb)$ is an elliptical region centered at $\zb$ with
semi-axes being $\{\del \sqrt{\lambda_i(\xb)} \}_{i=1}^d$ where $\lambda_i(\xb)$ denotes the $i$-th smallest eigenvalue of $A(\xb)$.
By our assumption, $\varrho\leq \lambda_1(\xb) \leq \cdots \leq\lambda_d(\xb)\leq 1$.
\Cref{fig:ellipse_definition} shows a 2d example where the  semi-major axis of the ellipse is $\del \lambda_2^{1/2}(\xb)\leq \del$ and the semi-minor axis is $\del\lambda_1^{1/2}(\xb)\geq \del\sqrt{\varrho}$. For convenience, we define
\beq
\rho_\del(\xb,\yb) = \frac{1}{\del^{d+2}}\gamma\left(\frac{|M(\xb)^{-1}\yb|}{\del}\right) \det(M(\xb))^{-1} 
\eeq 
and simply write $\cL_\del u(\xb) =\int_{\cE_\del^{\xb}(\bm{0})} \rho_\del(\xb,\yb)  (u(\xb+\yb) - u(\xb)) d\yb  $. 

In the next, we show the consistency between $\cL_\del$ and $L$ on sufficiently smooth functions. 
The first result asserts that $\cL_\del u$ agrees with $L u$ for all polynomials up to the third order,
and the second result gives pointwise truncation error for sufficiently smooth functions. 
Similar calculations can be found in \cite{nochetto2018discrete}, here we present them for completeness and for 
the convenience of analyzing our numerical method in \Cref{sec:meshfree}. 

\begin{lem}
\label{lem:polynomialconsistency}
Let $\cP_p(\R^d)$ denote the space of all polynomials up to order $p$ and  $\bm\al= (\al_1, \al_2,\cdots, \al_d) \in (\Z^+\cup\{0\})^d$ with $|\bm\al | =\sum_{i}\al_i$. Then
for any $\xb\in\R^d$,
\beq
\int_{\cE_\del^{\xb}(\bm{0})} \rho_\del(\xb,\yb) \yb^{\bm\al}  d\yb = \bm 0 
\eeq
if $|\bm\al | $ is an odd number
and 
\beq
\label{eq:NonlocalLocaloperator}
\cL_\del u(\xb)  = L u(\xb) \quad \forall u\in \cP_3(\R^d). 
\eeq
\end{lem}
\begin{proof}
Consider $\xb\in \R^d$ fixed. 
If $|\bm\al |$ is an odd number we have 
\[
\begin{split}
\int_{\cE_\del^{\xb}(\bm{0})} \rho_\del(\xb,\yb) \yb^{\bm\al}  d\yb&=  \int_{\cE_\del^{\xb}(\bm{0})}\frac{1}{\del^{d+2}}\gamma\left(\frac{|M(\xb)^{-1}\yb|}{\del}\right) \det(M(\xb))^{-1}  \yb^{\bm\al} d\yb   \\
&= \int_{B_\del(\bm{0})}\frac{1}{\del^{d+2}}\gamma\left(\frac{|\yb|}{\del}\right)  (M(\xb)\yb)^{\bm\al} d\yb \\
&  =0 .
\end{split}
\]
This last equality above is due to the symmetry of the integration domain and anti-symmetry of the integrand.  

Now it is easy to see that $\cL_\del u(\xb)  = L u(\xb)=0$ when $u$ is a constant or $u(\zb) = (\zb-\xb)^{\bm\al}$ with $|\bm\al|$ being an odd number.
On the other hand, notice that 
\[
\int_{B_\del(\bm{0})} \frac{1}{\del^{d+2}}  \gamma\left(\frac{|\yb|}{\del}\right)  y_i y_j  d\yb = 0 \quad i\neq j
\]
and
\[
\int_{B_\del(\bm{0})} \frac{1}{\del^{d+2}}  \gamma\left(\frac{|\yb|}{\del}\right) y_i^2 d\yb = \frac{1}{d} \int_{B_\del(\bm{0})}  \frac{1}{\del^{d+2}}  \gamma\left(\frac{|\yb|}{\del}\right) |\yb|^2 d\yb =2  \quad \forall i. 
\]
We have 
\[
\begin{split}
\int_{\cE_\del^{\xb}(\bm{0})} \rho_\del(\xb,\yb) (\yb\otimes\yb)  d\yb &=  \int_{\cE_\del^{\xb}(\bm{0})}\frac{1}{\del^{d+2}}\gamma\left(\frac{|M(\xb)^{-1}\yb|}{\del}\right) \det(M(\xb))^{-1}  (\yb\otimes\yb) d\yb \\
& = \int_{B_\del(\bm{0})}\frac{1}{\del^{d+2}}\gamma\left(\frac{|\yb|}{\del}\right) (M(\xb)\yb\otimes M(\xb)\yb) d\yb  \\
&= M(\xb) \int_{B_\del(\bm{0})}\left(\frac{1}{\del^{d+2}}\gamma\left(\frac{|\yb|}{\del}\right) (\yb\otimes \yb) d\yb \right) M(\xb)  \\
&= 2 (M(\xb))^2 = 2 A(\xb),
\end{split}
\]
and thus $\cL_\del u(\xb)  = L u(\xb)$ when $u(\zb) = (\zb-\xb)^{\bm\al}$ with $|\bm\al|=2$. 
Since $\cL_\del(\xb)$ agrees with  $ L u(\xb)$ for  $u(\zb) = (\zb-\xb)^{\bm\al}$ with $|\bm\al|\leq 3$,  \cref{eq:NonlocalLocaloperator} is true. 
\end{proof}

We now consider an open bounded domain $\Om\subset\R^d$. For $\xb\in\Om$ near the boundary of $\Om$, the definition in \cref{NonlocalEllipticOp} requires the values of $u$ outside $\Om$. 
Therefore nonlocal equations on bounded domains are usually accompanied by volumetric constraints (\cite{DGLZ12}) imposed over a boundary layer surrounding $\Om$.   
In our case, we need to define the boundary interaction layer as 
\[
\Om_{\cI_{\del}} =\{ \xb\in\R^d\backslash\Om:  \dist(\xb,\partial\Om)< \del\} . 
\]
We then denote the extended domain $\Om_\del := \Om \cup \Om_{\cI_{\del}}  $.  The consistency of $\cL_\del$ to $L$ is indicated in the following lemma. 

\begin{lem}
\label{lem:truncationerror}
Let $\cL_{\del} u $ be defined by \eqref{NonlocalEllipticOp} and $C>0$ being a generic constant. 
Let $\del_0>0$ be a fixed number. 
\begin{enumerate}
\item If $u\in C^2(\overline{\Om_{\del_0}})$, then $|\cL_{\del} u(\xb) - L u (\xb)| \to 0$ as $\del\to0$ for all $\xb\in \Om$. 
\item If $u\in C^{k,\alpha}(\overline{\Om_{\del_0}})$ for $k=2$ or $3$ and $\alpha\in (0,1]$, then $$|\cL_{\del} u(\xb) - L u (\xb)| \leq C |u|_{C^{k,\alpha}(\overline{\Om_{\del_0}})}\del^{k-2+\alpha}$$  for all $\xb\in \Om$ and $\del\leq \del_0$.  
\end{enumerate}
\end{lem}
\begin{proof}
Notice that 
\[
u(\xb +\yb) - u(\xb) = \int_0^1 \frac{d}{dt} u(\xb + t \yb) dt =  \int_0^1 \nabla u(\xb + t \yb) \cdot \yb dt. 
\]
Therefore 
\[
\begin{split}
\cL_{\del} u(\xb) &=  \int_{\cE_\del^{\xb}(\bm{0})} \rho_\del(\xb,\yb) \yb^T \int_0^1 \nabla u(\xb + t \yb)  dt d\yb \\
&=  \int_{\cE_\del^{\xb}(\bm{0})} \rho_\del(\xb,\yb) \yb^T \cdot\int_0^1 (\nabla u(\xb + t \yb) - \nabla  u(\xb)) dt d\yb 
\end{split}
\]
where we have used $\int_{\cE_\del^{\xb}(\bm{0})} \rho_\del(\xb,\yb) \yb d\yb =\bm{0} $ for the second line above. Then use 
\[
\nabla u(\xb + t \yb) - \nabla  u(\xb) = \int_0^1 D^2 u (\xb + st \yb) \cdot t\yb ds,
\]
we have 
\[
\cL_{\del} u(\xb) =  \int_{\cE_\del^{\xb}(\bm{0})} \rho_\del(\xb,\yb) (\yb\otimes \yb) : \int_0^1 t \int_0^1 D^2 u (\xb + st\yb) ds dt d\yb  
\]
for all $\xb \in \Om$. On the other hand, from calculations in \Cref{lem:polynomialconsistency}, one can show for $\xb\in\Om$, 
\[
\begin{aligned}
L u(\xb) &= \int_{\cE_\del^{\xb}(\bm{0})} \rho_\del(\xb,\yb) \frac{(\yb\otimes \yb)}{2} :  D^2 u (\xb) d\yb\\
&=  \int_{\cE_\del^{\xb}(\bm{0})} \rho_\del(\xb,\yb) (\yb\otimes \yb) : \int_0^1 t \int_0^1 D^2 u (\xb) ds dt d\yb.
\end{aligned}
\]

For $u\in C^2(\overline{\Om_{\del_0}})$, we have  $D^2 u (\xb + st\yb)  \to D^2 u (\xb)$ as $\del\to 0$ since $\yb\in \cE_\del^{\xb}(\bm{0})$. Therefore case $(1)$ holds. 
For $u\in C^{2,\alpha}(\overline{\Om_{\del_0}})$, we have 
\[
|D^2 u (\xb + st \yb) - D^2 u (\xb)|_{\infty} \leq  |u|_{C^{2,\alpha}(\overline{\Om_{\del_0}})}  |st \yb |^{\alpha}. 
\]
So 
\[
\begin{split}
|\cL_{\del} u(\xb)  - L u(\xb) |& \leq C |u|_{C^{2,\alpha}(\overline{\Om_{\del_0}})}    \int_{\cE_\del^{\xb}(\bm{0})} \rho_\del(\xb,\yb)|\yb|^{2+\alpha} d\yb \\
& \leq C |u|_{C^{2,\alpha}(\overline{\Om_{\del_0}})}  \del^{\alpha} \int_{\cE_\del^{\xb}(\bm{0})} \rho_\del(\xb,\yb)|\yb|^{2} d\yb  \\
& =C |u|_{C^{2,\alpha}(\overline{\Om_{\del_0}})}  \del^{\alpha} \int_{B_\del(\bm{0})} \frac{1}{\del^{d+2}} \ga_\del\left(\frac{|\yb|}{\del} \right)\yb^T A(\xb) \yb d\yb  \\
& \leq C|u|_{C^{2,\alpha}(\overline{\Om_{\del_0}})}  \del^{\alpha} \int_{B_\del(\bm{0})} \frac{1}{\del^{d+2}} \ga_\del\left(\frac{|\yb|}{\del} \right) |\yb|^2 d\yb\\
&= 2d C |u|_{C^{2,\alpha}(\overline{\Om_{\del_0}})}  \del^{\alpha} .
\end{split}
\]
Finally, if $u\in C^{3,\alpha}(\overline{\Om_{\del_0}})$, then we can write
\[
\begin{split}
&\quad \ \cL_{\del} u(\xb)  - L u(\xb)\\
&= \int_{\cE_\del^{\xb}(\bm{0})} \rho_\del(\xb,\yb) (\yb\otimes \yb) : \int_0^1 t \int_0^1 (D^2 u (\xb + st \yb) - D^2 u (\xb) ) ds dt d\yb   \\
&= \int_{\cE_\del^{\xb}(\bm{0})} \rho_\del(\xb,\yb) (\yb\otimes \yb\otimes \yb):  \int_0^1 t^2 \int_0^1  s \int_0^1  D^3 u (\xb + rst \yb) dr ds dt d\yb \\
&= \int_{\cE_\del^{\xb}(\bm{0})} \rho_\del(\xb,\yb) (\yb\otimes \yb\otimes \yb):  \int_0^1 t^2 \int_0^1  s \int_0^1 \left( D^3 u (\xb + rst \yb)  - D^3 u(\xb) \right) dr ds dt d\yb . 
\end{split}
\]
Therefore, by the same reasoning as before, we have 
\[
|\cL_{\del} u(\xb)  - L u(\xb) | \leq C |u|_{C^{3,\alpha}(\overline{\Om_{\del_0}})}    \int_{\cE_\del^{\xb}(\bm{0})} \rho_\del(\xb,\yb)|\yb|^{3+\alpha} d\yb \leq 2dC |u|_{C^{3,\alpha}(\overline{\Om_{\del_0}})}  \del^{1+\alpha}.  
\]
\end{proof} 


\section{Meshfree discretization}
\label{sec:meshfree}

Meshfree methods have been widely used in simulations of scientific problems. 
There are many existing meshfree approaches for solving PDEs on scattered datasets, such as vortex methods \cite{cottet2000vortex}, smooth particle hydrodynamics (SPH) \cite{gingold1977smoothed}, radial basis functions (RBF)\cite{buhmann2003radial}, moving least squares (MLS) / generalized moving least squares (GMLS) \cite{lancaster1981surfaces,mirzaei2012generalized,wendland2004scattered}, and reproducing kernel particle method (RKPM) \cite{liu1996overview}. 
The key idea in MLS and GMLS is a local fitting of data using least squares approximation. 
They can also be written as a weighted $l_2$-type local optimization under certain reproducing conditions. 
In \cite{seibold2008minimal,davydov2018minimal}, weighted $l_1$-type optimization was discussed for the sparsity of stencils. 
It was shown that using $l_1$-type optimization, the number of nonzero weights generated is at most the number of constraints in the reproducing condition. 
Such property is important to keep the linear system sparse, especially when the elliptic problem is nearly degenerate ($\varrho\ll 1$), in the case of which 
the searching region becomes large (cf.~\Cref{thm:nbh_criterion}). 
For the rest of this section, we discuss the meshfree method for solving \cref{eq:elliptic} based on the nonlocal relaxation and the convergence of the numerical method. 

\subsection{Optimization based meshfree discretization}
Our numerical method is inspired by the meshfree finite difference method presented in \cite{seibold2008minimal} for solving the classical Poisson equation. 
The focus here is on the generation of positive stencils which lead to monotone schemes. 
The desirability of positive stencils was observed in other meshless methods, see e.g., \cite{demkowicz1984some,liszka1996hp}, although there was no guarantee of positive stencils in these works.  
We now present a reformulation of the meshfree method in \cite{seibold2008minimal} as a nonlocal relaxation method on which a generalization to elliptic equations is based. 
Given a point cloud $X= \{\xb_i\} \subset \R^d$ with $h$ being its associated fill distance to be defined later, it is proposed in  \cite{seibold2008minimal} to discretize the Laplace operator by
\beq
\label{eq:laplace_approx}
\Delta u (\xb_i)\approx \Delta_h u (\xb_i) = \sum_{\xb_j \in B_\del(\xb_i)} \beta_{j,i} (u(\xb_j) - u(\xb_i)),
\eeq
where the weights $\{ \beta_{j,i}\}$ are determined by the following linear minimization problem in order to achieve the so-called minimal positive stencils:
\beq
\label{eq:Seibold_original}
\begin{aligned}
\{ \beta_{j,i}\} &=  \argmin \sum_{j} \frac{\beta_{j, i}}{W(|\xb_j - \xb_i|)} \\
\text{s.t. } \beta_{j,i}\geq 0 &\text{ and } \Delta_h u (\xb_i) = \Delta u (\xb_i)\;  \forall u \in \cP_p(\R^d) 
\end{aligned}
\eeq
In \cite{seibold2008minimal}, the polynomial space $\cP_p(\R^d)$ is taken to be $\cP_2(\R^d)$ with $p=2$. 
The parameter $\del$ in \cref{eq:laplace_approx} is determined in relation to $h$ such that the feasible set of the minimization problem is non-empty. 
The weight function $W(r)$ is suggested in \cite{seibold2008minimal} as $W(r) = r^{-\alpha}$ for $\alpha>0$. It is not hard to see that when we choose the nonlocal kernel function 
$\ga(r) = C r^{-\alpha} \chi_{\{ |r|<1\}} $ that satisfies \cref{eq:kernel_secondmoment}, then  by letting $ \beta_{j,i} =  \frac{1}{\del^{d+2}} \ga\left(\frac{|\xb_j -\xb_i|}{\del}\right)\om_{j,i} $,  
the minimization problem \eqref{eq:Seibold_original} is equivalent to 
\beq
\label{eq:Seibold_equiv}
\begin{aligned}
\{ \om_{j,i}\} &=  \argmin   \sum_{j} \om_{j, i} \\
\text{s.t. } \om_{j,i}\geq 0 &\text{ and } \widetilde{\cL}_{\del}^h u (\xb_i) = \widetilde{\cL}_\del u (\xb_i)\;  \forall u \in \cP_p(\R^d).
\end{aligned}
\eeq
where $ \widetilde{\cL}_{\del}^h$ is the nonlocal approximate operator defined by 
\beq
\widetilde{\cL}_{\del}^h  u(\xb_i) = \sum_{\xb_j \in B_\del(\xb_i)}  \frac{1}{\del^{d+2}} \ga\left(\frac{|\xb_j -\xb_i|}{\del}\right)\om_{j,i} (u(\xb_j) - u(\xb_i)). 
\eeq 
Based on this observation, we propose the discretization of nonlocal elliptic operator \eqref{NonlocalEllipticOp} as 
\beq
\label{eq:nonlocaldiscrete}
\cL_{\del}^h u(\xb_i)= \sum_{\xb_j \in \cE_\del^{\xb_i}(\xb_i)} \rho_\del(\xb_i,\xb_j -\xb_i)   \om_{j,i} (u(\xb_j) - u(\xb_i)). 
\eeq
where $\rho_\del(\xb_i,\xb_j -\xb_i)=\frac{1}{\del^{d+2}} \ga\left(\frac{|M(\xb_i)^{-1}(\xb_j -\xb_i)|}{\del}\right) \det(M(\xb_i))^{-1} $.
The weights $\{ \om_{j,i}\}$ in \cref{eq:nonlocaldiscrete} are solved by the minimization problem
\beq
\label{eq:nonlocalminimization}
\begin{aligned}
\{ \om_{j,i}\} &=  \argmin   \sum_{j} \om_{j, i} \\
\text{s.t. } \om_{j,i}\geq 0 &\text{ and }  \cL_{\del}^h  u (\xb_i)=\cL_\del u (\xb_i) \;  \forall u \in \cP_p(\R^d).  
\end{aligned}
\eeq
The well-posedness \eqref{eq:nonlocalminimization} is guaranteed only if the feasible set is non-empty. 
We will discuss in \Cref{subsec:neighborhood_criteria} the neighborhood criteria for  non-emptiness of the feasible set 
\beq
\label{eq:feasibleset}
S_{\del,h,p}(\xb_i):= \left\{ \{ \om_{j,i}\}: \om_{j,i}\geq 0 \text{ and } \cL_{\del}^h  u (\xb_i) = \cL_\del u (\xb_i) \;  \forall u \in \cP_p(\R^d) \right\}. 
\eeq

\begin{remark}
Notice that the constrained optimization problem presented in \cref{eq:nonlocalminimization} can be categorized as a linear programming problem. This type of problem can be efficiently solved using the simplex method \cite{dantzig1990origins}. Moreover, $\mathcal{L}_{\delta}$ coincide with $L$ when $p \leq 3$. 
\end{remark}

\subsection{Boundary treatment}
\label{subsec:boundarytreatment}
For an open and bounded domain $\Om\subset \R^d$, we take a point cloud $X =\{ \xb_i\}_{i=1}^M \subset \Om_\del$ and define its associated fill distance 
\beq
\label{eq:filldistance}
h := \sup_{\xb \in \Om_\del} \min_{1\leq i\leq M} |\xb-\xb_i| 
\eeq
following the convention in \cite{wendland2004scattered}. 
Assume that $\{ \xb_i\}_{i=1}^N\subset \Om$. 
For $\xb_i\in\Om$ near the boundary of $\partial\Om$, the defintion in \cref{eq:nonlocaldiscrete} may require the value of $u(\xb_j)$ for $\xb_j \in \Om_{\cI_\del}$. 
Therefore, extensions of the boundary values from $\partial\Om$ to $ \Om_{\cI_\del}$ are needed. 
However, it is usually hard to find an easy way to do the extension, especially in higher dimensions, to guarantee a second-order convergence rate for nonlocal solutions. 
We propose an alternative way for the boundary treatment. 

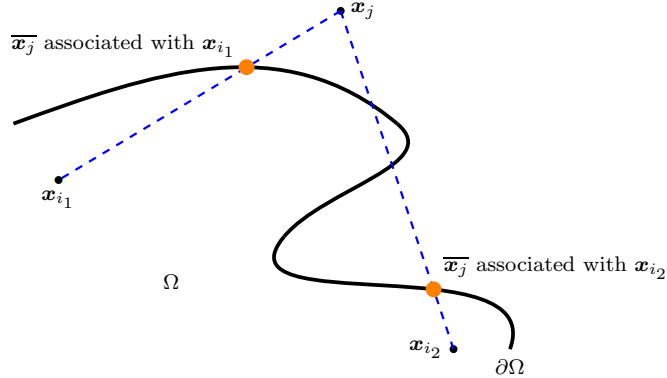
\begin{figure}[htp]
    \centering
    \begin{tikzpicture}[scale=1.5]
    \coordinate (A) at (-2.4, 1);
    \coordinate (B) at (1, 1);
    \coordinate (C) at (0, 0);
    \coordinate (D) at (2, -1);
    \draw[line width=0.5mm, name path=A--B] (A) to[out=20, in=140] (B);
    \draw[line width=0.5mm, name path=B--C] (B) to[out=-40, in=50] (C);
    \draw[line width=0.5mm, name path=D--C] (D) to[out=70, in=230, distance=1cm] (C);

    \coordinate (x1) at (-2, 0.5);
    \coordinate (x2) at (0.5, 2);
    \coordinate (x3) at (1.5, -1);
    \fill (x1) circle (1pt);
    \fill (x2) circle (1pt);
    \fill (x3) circle (1pt);
    \node[below] at (x1) {$\scriptstyle \boldsymbol{x}_{i_{1}}$};
    \node[right] at (x2) {$\scriptstyle \boldsymbol{x}_{j}$};
    \node[left] at (x3) {$\scriptstyle \boldsymbol{x}_{i_{2}}$};

    \draw[blue, dashed, line width=0.3mm, name path=2--1] (x2) -- (x1);
    \draw[blue, dashed, line width=0.3mm, name path=2--3] (x2) -- (x3);

    \path[name intersections={of=A--B and 2--1, by=E}];
    \fill[orange] (E) circle (2pt);
    \node[above left] at (E) {$\scriptstyle \overline{\boldsymbol{x}_{j}} \text{ associated with } \boldsymbol{x}_{i_{1}}$};

    \path[name intersections={of=D--C and 2--3, by=F}];
    \fill[orange] (F) circle (2pt);
    \node[above right] at (F) {$\scriptstyle \overline{\boldsymbol{x}_{j}} \text{ associated with } \boldsymbol{x}_{i_{2}}$};

    \node[below] at (D) {$\scriptstyle \partial \Omega$};
    \node[below] at (-1, -0.25) {$\scriptstyle \Omega$};

    \end{tikzpicture}
    \caption{Illustration of the projection}
    \label{fig:projection}
\end{figure}
For $\xb_i \in \Om$ and $\xb_j \in \cE_\del^{\xb_i}(\xb_i) $, we define $\overline{\xb_j} = \xb_j$ if $\xb_j \in \overline{\Om}$, otherwise $\overline{\xb_j} \in \partial\Om$ is defined as the projection of $\xb_j$ onto $ \partial\Om$ such that the line from $\xb_i $ to $\overline{\xb_j}$ is contained in $\overline{\Om}$. 
Notice that the projected point $\overline{\xb_j}$ depends on both $\xb_j$ and $\xb_i$. Here for notational convenience, we have omitted the dependence on $i$ and simply denoted the projected point as $\overline{\xb_j}$. 
See \Cref{fig:projection} as an illustration of the projection. 
We then define the approximate operator associated with $\del$, $h$ and $\Om$ as 
\beq
\label{NonlocalEllipticOpBdry}
\cL_{\del,\Om}^h u(\xb_i)= \sum_{\xb_j \in \cE_\del^{\xb_i}(\xb_i)} \rho_\del(\xb_i, \overline{\xb_j} -\xb_i) \om_{j,i} (u(\overline{\xb_j}) - u(\xb_i))
\eeq
where $\{ \om_{j, i}\}$ is solved from 
\beq
\label{eq:nonlocalminimization_Om}
\{ \om_{j,i}\} = \underset{\{ \om_{j,i}\} \in \overline{S}_{\del,h,p}(\xb_i)}{\argmin} \sum_{j} \om_{j, i}
\eeq
where $ \overline{S}_{\del,h,p}(\xb_i)$ is the feasible set defined as 
\beq
\label{eq:feasibleset_Om}
\overline{S}_{\del,h,p}(\xb_i):= \left\{ \{ \om_{j,i}\}: \om_{j,i}\geq 0 \text{ and } \cL_{\del,\Om}^h  u (\xb_i) = \cL_\del u (\xb_i) \;  \forall u \in \cP_p(\R^d) \right\}. 
\eeq
We will address in  \Cref{subsec:neighborhood_criteria} the feasibility of the minimization problem \eqref{eq:nonlocalminimization_Om}. 

With the definition of the discrete operator $\cL_{\del,\Om}^h$, we define the discrete problem as to find a function $u_\del^h : \{\xb_i\}_{i=1}^N \cup\partial\Om \to \R$ such that
\beq
\label{eq:elliptic_discrete}
\left\{ 
\begin{aligned}
- \cL_{\del,\Om}^h u_{\del}^h (\xb_i) &= f(\xb_i) \quad \xb_i \in \Om \\
u_{\del}^h(\xb) &= g(\xb) \quad \xb \in \partial\Om 
\end{aligned}
\right.
\eeq
Notice that in the above equation, the boundary condition is imposed on the boundary set $\partial\Omega$. The discrete problem is well-posed provided that the weights $\{ \omega_{j,i}\}$ in \cref{NonlocalEllipticOpBdry} are nonnegative. In the subsequent subsections, we will first prove the discrete maximum principle for \cref{eq:elliptic_discrete} (which implies the well-posedness of the discrete system) assuming the existence of nonnegative weights, and then discuss sufficient conditions to find the positive stencils.

\subsection{Convergence analysis}

We first provide a truncation error analysis for the discrete operator $\cL_{\del,\Om}^h$ and then the convergence is followed by the monotonicity of the numerical scheme. In this subsection, the errors are presented in terms of $\del$, which denotes the upper bound of the semi-major axis of the elliptical neighborhood of each point. Following the neighborhood criteria in \Cref{subsec:neighborhood_criteria}, the convergence errors are finally presented in terms of $h$.  

\begin{lem}
\label{lem:truncationerror_discrete}
Take a point cloud $X =\{ \xb_i\}_{i=1}^M \subset \Om_\del$ with $\{ \xb_i\}_{i=1}^N \subset \Om$. 
Assume also that $\overline{S}_{\del,h,p}(\xb_i)$ is not empty and $C>0$ is a generic constant. 
\begin{enumerate}
\item If $p\geq 2$ and $u\in C^2(\overline{\Om})$, then $|\cL_{\del,\Om}^h u(\xb_i) - L u (\xb_i)| \to 0$ as $\del\to0$ for all $\xb_i\in \Om$. 
\item If $p\geq 2$ and $u\in C^{2,\alpha}(\overline{\Om})$ for $\alpha\in (0,1]$, then   $|\cL_{\del,\Om}^h u(\xb_i) - L u (\xb_i)| \leq C |u|_{C^{2,\alpha}(\overline{\Om})}\del^{\alpha}$ for all $\xb_i\in \Om$.
\item If $p\geq 3$ and $u\in C^{3,\alpha}(\overline{\Om})$ for $\alpha\in (0,1]$, then   $|\cL_{\del,\Om}^h u(\xb_i) - L u (\xb_i)| \leq C |u|_{C^{3,\alpha}(\overline{\Om})}\del^{1+\alpha}$  for all $\xb_i\in \Om$.
\end{enumerate}
\end{lem}
\begin{proof}
Consider a fixed $\xb_i \in \Om$. 
The proof follows closely from the proof of \Cref{lem:truncationerror} by noticing that $\{\om_{j,i}\} \in \overline{S}_{\del,h,p}(\xb_i) $ implies that
\beq
\int_{ \cE_\del^{\xb_i}(\bm{0})} \rho_\del(\xb_i,\yb) \yb^{\bm \al} d\yb = \sum_{\xb_j \in \cE_\del^{\xb_i}(\xb_i)} \rho_\del(\xb_i,\overline{\xb_j} -\xb_i) (\overline{\xb_j} -\xb_i)^{\bm \al} \om_{j,i} 
\eeq
for $\bm\al = (\al_1, \al_2,\cdots ,\al_d) \in (\Z^+\cup\{ 0\})^d$ with $|\bm\al | = \sum_i \al_i\leq p $. 
\end{proof}

\begin{remark}
\label{rmk:truncationerror}
In practice, we often observe superconvergence for $p=2$, which is likely due to symmetry.  
When $p=2$ and $u\in C^{3,\alpha}(\overline{\Om})$, a more precise error estimate for $\xb_i\in  \Om$ is given by
\[
|\cL_{\del,\Om}^h u(\xb_i) - L u (\xb_i)| \leq C \left(|u|_{C^{3}(\overline{\Om})}  T_3(\xb_i) + |u|_{C^{3,\alpha}(\overline{\Om})}\del^{1+\alpha}  \right)
\]
where $T_3(\xb_i)=\max_{|\bm\al|=3} |\sum_{ \xb_j \in \cE_\del^{\xb_i}(\xb_i)} \rho_\del(\xb_i,\overline{\xb_j} -\xb_i) (\overline{\xb_j} -\xb_i)^{\bm \al} \om_{j,i}  | $
\end{remark}

\begin{lem}[Discrete maximum principle]
\label{lem:DMP}
Let $\Om\subset\R^d$ be an open, bounded and simply connected domain. Take a point cloud $X =\{ \xb_i\}_{i=1}^M \subset \Om_\del$ with $\{ \xb_i\}_{i=1}^N \subset \Om$. 
Assume that there exists $\xb_i\in \Om$ such that $\cE_\del^{\xb_i}(\xb_i) \cap \Om^c \neq \emptyset$. 
If $u\in C(\{ \xb_i\}_{i=1}^N \cup \partial\Om)$ and $\cL_{\del,\Om}^h u (\xb_i) \geq 0$ for all $\xb_i\in \Om$, then 
\[
\max_{\xb_i\in\Om} u(\xb_i) \leq \max_{\xb\in\partial\Om}  u(\xb). 
\]
\end{lem}
\begin{proof}
First notice that $\max_{\xb\in\partial\Om}  u(\xb)$ is well-defined since $\partial\Om$ is a closed set and $u$ is continuous on $\partial\Om$.
Assume that $\max_{\xb_i\in\Om} u(\xb_i)  >\max_{\xb\in\partial\Om}  u(\xb)$, then there exists $\xb_k \in\Om$ such that 
\[
u(\xb_k) = \max_{\xb_i\in\Om} u(\xb_i) \geq u(\xb) \quad \forall \xb\in \{ \xb_i\}_{i=1}^N \cup \partial\Om.  
\]
Therefore 
\[
\cL_{\del,\Om}^h u (\xb_k) \leq 0.  
\]
By the assumption, we must have $\cL_{\del,\Om} u (\xb_k)=0$ and $u(\xb_j) = u(\xb_k)$ for $\xb_j \in \cE_\del^{\xb_k}(\xb_k) \cap\Om$. 
Continue this process we can shown that $u$ is constant on $\{ \xb_i\}_{i=1}^N \subset\Om$. 
Choose $\xb_i \in \Om$ such that there exists $\xb_j \in  \cE_\del^{\xb_i}(\xb_i) \cap \Om^c$, so $\overline{\xb_j}\in \partial\Om$. 
However, since we can argue that $\cL_{\del,\Om}^h u (\xb_i) = 0$, it implies $u(\xb_i ) =  u(\overline{\xb_j}) $, which is contradiction to  $\max_{\xb_i\in\Om} u(\xb_i)  >\max_{\xb\in\partial\Om}  u(\xb)$. 
\end{proof}

\begin{thm}
\label{thm:convergence}
Take a point cloud $X =\{ \xb_i\}_{i=1}^M \subset \Om_\del$ and assume that $\overline{S}_{\del,h,p}(\xb_i)$ is not empty.  
Let $u$ and $u_\del^h$ be the solutions to \cref{eq:elliptic,eq:elliptic_discrete} respectively.
\begin{enumerate}
\item If $p\geq 2$ and $u\in C^{2,\alpha}(\overline{\Om})$ for $\alpha\in (0,1]$, then  
\[
\max_{\xb_i\in\Om}|u(\xb_i) - u_\del^h(\xb_i)| \leq  C |u|_{C^{2,\alpha}(\overline{\Om})}\del^{\alpha}. 
\] 
\item If $p\geq 3$ and $u\in C^{3,\alpha}(\overline{\Om})$ for $\alpha\in (0,1]$, then   
\[
\max_{\xb_i\in\Om}|u(\xb_i) - u_\del^h(\xb_i)| \leq C |u|_{C^{3,\alpha}(\overline{\Om})}\del^{1+\alpha}. 
\]
\end{enumerate}
\end{thm}

\begin{remark}
\label{rmk:superconvergence}
Using the truncation error estimate in \Cref{rmk:truncationerror}, one can show that if $p= 2$ and $u\in C^{3,\alpha}(\overline{\Om})$, then 
\[
\max_{\xb_i\in\Om}|u(\xb_i) - u_\del^h(\xb_i)| \leq C ( |u|_{C^{3}(\overline{\Om})} \tau + |u|_{C^{3,\alpha}(\overline{\Om})}\del^{1+\alpha}).  
\]
where 
\[
\tau = \max_{\xb_i\in\Om} T_3(\xb_i)
\]
In practice, $\tau$ might be a very small number depending on the point cloud. Therefore, superconvergence may be observed. 
\end{remark}

\begin{proof}[Proof of {\Cref{thm:convergence}}.]
We only show the proof for the first case and the second case can be similarly shown. 
Denote $e_\del^h(\xb)=u(\xb) - u_\del^h(\xb)$ for $\xb\in \{ \xb_i\}_{i=1}^N\cup \partial\Om$ and $T_\del^h(\xb_i) = \cL_{\del,\Om}^h u(\xb_i) - L u (\xb_i)$
for $\xb_i\in\Om$. Notice that $\cL_{\del,\Om}^h e_\del^h (\xb_i) = T_\del^h(\xb_i)$ for  $\xb_i\in\Om$. By \Cref{lem:truncationerror_discrete}, we have
\[
K: = \max_{\xb_i\in\Om} |T_\del^h(\xb_i)| \leq   C |u|_{C^{2,\alpha}(\overline{\Om})} \del^{\alpha} .  
\]
Take $\xb^\ast\in \Om$ such that $\Om\subset B_{R}(\xb^\ast)$ for some $R>0$.
Then define $\Phi(\xb) = (\xb -\xb^\ast)^T(\xb -\xb^\ast)/(2d)$, we have 
\[
\cL_{\del,\Om}^h \Phi(\xb) = A(\xb): D^2 \Phi (\xb) =\frac{1}{d} \sum_{i=1}^d a_{ii}(\xb)  =\frac{1}{d}   \sum_{i=1}^d \eb_i^T A(\xb)  \eb_i \geq \la,
\]
where $\eb_i \in \R^d$ is the unit vector with the $i$-th component equal $1$.  Therefore we have 
\[
\cL_{\del,\Om}^h \left(\frac{K}{\la}  \Phi  + e_\del^h\right)(\xb_i) \geq0  \quad \forall \xb_i \in\Om,
\]
and by \Cref{lem:DMP}, we have 
\[
\begin{split}
\max_{\xb_i \in\Om} e_\del^h(\xb_i) &\leq \max_{\xb_i \in\Om} \left(\frac{K}{\la}  \Phi(\xb_i)  + e_\del^h(\xb_i)\right)  \leq \max_{\xb\in\partial\Om}\left(\frac{K}{\la}  \Phi(\xb)  + e_\del^h(\xb)\right) \\
& = \frac{K}{\la} \max_{\xb\in\partial\Om} \Phi(\xb) \leq   \frac{K}{\la} \frac{R^2}{2d}\leq \frac{C R^2}{2\la  d}  |u|_{C^{2,\alpha}(\overline{\Om})} \del^{\alpha} . 
\end{split}
\]
Similar estimates can be done for $-e_\del^h(\xb_i)$ and therefore the proof is complete. 
\end{proof}

\subsection{Neighborhood criteria}
\label{subsec:neighborhood_criteria}

In this subsection, we will discuss the neighborhood criteria that guarantee positive stencils. We only discuss the case $p=2$ in this subsection. The case $p=3$ is much harder to characterize which will be left for future work. 

First of all, there is a sufficient criterion for positive stencils for solving the Laplace equation, and it is presented as a cone condition in \cite{seibold2008minimal} for $d=2$ or $d=3$. 
For any $\xb_i \in \Om$ and unit vector $\vb \in \R^d$, we define an associated cone $\cC^{\vb}_\del(\xb_i)$ in $B_\del(\xb_i)$ by
\beq
\label{eq:cone}
\cC^{\vb}_\del(\xb_i) := \left\{\xb \in  B_\del(\xb_i) :  \xb^T \vb \geq \frac{1}{\sqrt{1+\sig_d}} |\xb|^2 \right\}
\eeq
where $\sig_d = \sqrt{2} -1$ (a cone with total opening angle $45^\circ$) for $d=2 $ and $\sig_d = \sqrt{(3-\sqrt{6})/6}$ (a cone with total opening angle $33.7^\circ$)
for $d=3$. With a rephrasing of words, we quote the result in \cite[Theorems 9 and 10]{seibold2008minimal} in the following lemma. 

\begin{lem}[Theorems 9 and 10 in \cite{seibold2008minimal}]
\label{lem:Seibold}
Take a point cloud $X = \{\xb_i\}_{i=1}^M \subset \Om_\del \subset \R^d$ and let $\xb_i \in \Om$ be fixed. 
If for any unit vector $\vb \in \R^d$,  $\cC^{\vb}_\del(\xb_i)\cap X\backslash \{ \xb_i\} \neq \emptyset$, then the feasible set to problem \eqref{eq:Seibold_equiv} 
with $p=2$ is not empty. 
\end{lem}

To discuss the neighborhood criteria for our problem, we first notice that for $\xb_i \in \Om$, one can define a one-to-one mapping between $B_\del(\xb_i)$ and $\cE_\del^{\xb_i}(\xb_i)$ by
\[
T_i \xb = \xb_i + M(\xb_i) (\xb -\xb_i) \quad \xb \in B_\del(\xb_i). 
\]
The inverse of $T_i$ is then given by 
\[
T_i^{-1} \xb = \xb_i + M(\xb_i)^{-1} (\xb -\xb_i) \quad \xb \in \cE_\del^{\xb_i}(\xb_i). 
\]
We also denote $T_i(D) = \{ \yb = T_i \xb: \xb\in D\}$ and $T_i^{-1} (D) = \{ \yb = T_i^{-1} \xb: \xb\in D\}$ for any set $D\subset\R^d$. 
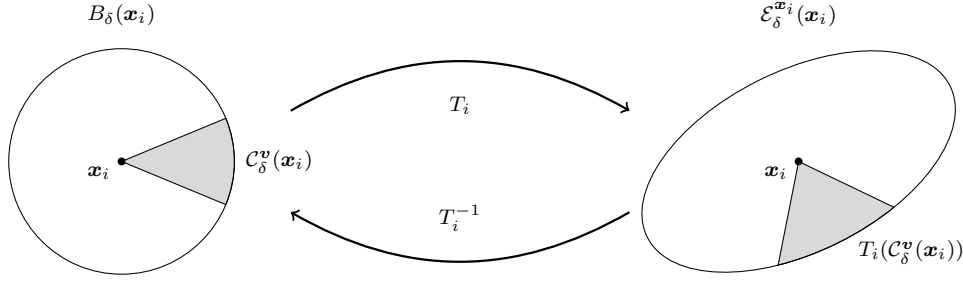
\begin{figure}[htp]
    \centering
    \begin{tikzpicture}[scale=1.5]
    \coordinate (c1) at (-3, 0);
    \coordinate (c2) at (3, 0);

    \draw (c1) circle (1);
    \fill (c1) circle (1pt);
    \node at ($(c1) + (0, 1.3)$) {$\scriptstyle B_{\delta}(\boldsymbol{x}_{i})$};
    \node at ($(c1) - (0.2, 0.1)$) {$\scriptstyle \boldsymbol{x}_{i}$};
    \node at ($(c1) + (1.4, 0)$) {$\scriptstyle \mathcal{C}_{\delta}^{\boldsymbol{v}}(\boldsymbol{x}_{i})$};
    \filldraw[fill opacity=0.3, fill=gray] (c1) -- +(-22.5:1) arc (-22.5:22.5:1) -- cycle;

    \draw[rotate around={-63.4322:(c2)}] (c2) ellipse (0.8018 and 1.5);
    \fill (c2) circle (1pt);
    \node at ($(c2) + (0, 1.3)$) {$\scriptstyle \mathcal{E}_{\delta}^{\boldsymbol{x}_{i}}(\boldsymbol{x}_{i})$};
    \node at ($(c2) - (0.2, 0.1)$) {$\scriptstyle \boldsymbol{x}_{i}$};
    \node at ($(c2) + (1.0, -0.8)$) {$\scriptstyle T_{i}(\mathcal{C}_{\delta}^{\boldsymbol{v}}(\boldsymbol{x}_{i}))$};
    \filldraw[rotate around={-63.4322:(c2)}, fill opacity=0.3, fill=gray] (c2) -- +(-22.5:0.8018 and 1.5) arc (-22.5:22.5:0.8018 and 1.5) -- cycle;

    \draw[->, line width=0.3mm] ($(c1) + (1.5, 0.45)$) to[out=30, in=150] ($(c2) + (-1.5, 0.45)$);
    \draw[->, line width=0.3mm] ($(c2) + (-1.5, -0.45)$) to[out=-150, in=-30] ($(c1) + (1.5, -0.45)$);
    \node at ($(c1)!0.5!(c2) + (0, 0.5)$) {$\scriptstyle T_{i}$};
    \node at ($(c1)!0.5!(c2) - (0, 0.5)$) {$\scriptstyle T_{i}^{-1}$};

    \end{tikzpicture}
    \caption{Illustration of $T_{i}$ and $T_{i}^{-1}$}
    \label{fig:ellipse_transform}
\end{figure}
\begin{lem}
\label{lem:equivalence}
Take a point cloud $X = \{\xb_i\}_{i=1}^M \subset \Om_\del \subset \R^d$ and let $\xb_i \in \Om$ be fixed.
For any $\xb_j \in \cE_\del^{\xb_i}(\xb_i)\cap X\backslash \{ \xb_i\}$, we write $\widetilde{\xb}_j  = T_i^{-1} \xb_j$. 
Let $\cL_\del^h u (\xb_i )$ be defined by \cref{eq:nonlocaldiscrete}  and 
\beq
\label{eq:nonlocaldiscrete_Laplace}
\widetilde{\cL}_\del^h u (\xb_i ) :=  \sum_{\widetilde{\xb}_j \in B_\del(\xb_i)} \frac{1}{\del^{d+2}} \ga\left(\frac{|\widetilde{\xb}_j -\xb_i|}{\del}\right)\widetilde{\om}_{j,i} (u(\widetilde{\xb}_j) - u(\xb_i)). 
\eeq
The following statements are equivalent.
\begin{enumerate}
\item There exists $\{ \widetilde{\om}_{j,i}\geq 0\}$ such that $\widetilde{\cL}_\del^h u (\xb_i) = \widetilde{\cL}_\del u(\xb_i)$ $\forall u \in\cP_2 (\R^d)$.  
\item There exists $\{ {\om}_{j,i}\geq 0\}$ such that ${\cL}_\del^h u (\xb_i) = {\cL}_\del u(\xb_i)$ $\forall u \in\cP_2 (\R^d)$.  
\end{enumerate}
\end{lem}
\begin{proof}
By definition, we see that for any $\xb_j \in \cE_\del^{\xb_i}(\xb_i) \iff \widetilde{\xb}_j \in B_\del(\xb_i)$. 
We show that (1) and (2) are equivalent by letting $\widetilde{\om}_{j,i} =  \det(M(\xb_i))^{-1}   \om_{j,i} $. 
Assume that (2) is true, then since $\cL_\del u   = 0$ if $u\in \cP_1(\R^d)$, we have 
\[
\begin{split}
\bm{0} &=  \sum_{\xb_j \in \cE_\del^{\xb_i}(\xb_i)} \frac{1}{\del^{d+2}} \ga\left(\frac{|M(\xb_i)^{-1}(\xb_j -\xb_i)|}{\del}\right) \det(M(\xb_i))^{-1}   \om_{j,i}  M(\xb_i)^{-1}(\xb_j -\xb_i) \\
& =  \sum_{\widetilde{\xb}_j \in \cE_\del^{\xb_i}(\xb_i)} \frac{1}{\del^{d+2}} \ga\left(\frac{|\widetilde{\xb}_j - \xb_i|}{\del}\right) \det(M(\xb_i))^{-1}   \om_{j,i}  ( \widetilde{\xb}_j - \xb_i ).
\end{split}
\]
By letting $\widetilde{\om}_{j,i} =  \det(M(\xb_i))^{-1}   \om_{j,i} $, we see that $\widetilde{\cL}_\del^h u (\xb_i) = 0 =\widetilde{\cL}_\del u(\xb_i)$ $\forall u \in\cP_1 (\R^d)$. 
Next by using $  {\cL}_\del^h u (\xb_i) = {\cL}_\del u(\xb_i) = A(\xb_i):D^2 u(\xb_i)$ for $u \in\cP_2 (\R^d)$, we have 
\[
\begin{split}
 &2 = M(\xb_i)^{-1} (2A(\xb_i)) M(\xb_i)^{-1} \\
=& M(\xb_i)^{-1}  \sum_{\xb_j \in \cE_\del^{\xb_i}(\xb_i)} \frac{1}{\del^{d+2}} \ga\left(\frac{|M(\xb_i)^{-1}(\xb_j -\xb_i)|}{\del}\right) \cdot \\
&\qquad \qquad \qquad \qquad \qquad \det(M(\xb_i))^{-1}   \om_{j,i}  \left( (\xb_j -\xb_i)\otimes (\xb_j -\xb_i) \right) M(\xb_i)^{-1} \\
=&  \sum_{\widetilde{\xb}_j \in \cE_\del^{\xb_i}(\xb_i)} \frac{1}{\del^{d+2}} \ga\left(\frac{|\widetilde{\xb}_j - \xb_i|}{\del}\right)   \widetilde{\om}_{j,i}  ( \widetilde{\xb}_j - \xb_i ) \otimes  ( \widetilde{\xb}_j - \xb_i ) .
\end{split}
\]
This implies that $\widetilde{\cL}_\del^h u (\xb_i) =\widetilde{\cL}_\del u(\xb_i) =\Delta u(\xb_i)$ when $u\in \cP_2(\R^d)$.  Therefore (2) implies (1). Similarly, we can show (1) also implies (2).
\end{proof}

The following result is an implication of  \Cref{lem:Seibold} and \Cref{lem:equivalence}. 
\begin{coro}
\label{coro:cone_condition}
Take a point cloud $X = \{\xb_i\}_{i=1}^M \subset \Om_\del \subset \R^d$ and let $\xb_i \in \Om$ be fixed. 
If for any unit vector $\vb \in \R^d$,  $T_i(\cC^{\vb}_\del(\xb_i))\cap X\backslash \{ \xb_i\} \neq \emptyset$, then $S_{\del, h, 2}(\xb_i)$ and $\overline{S}_{\del, h, 2}(\xb_i)$ are not empty. 
\end{coro}
\begin{proof}
First of all, from  \Cref{lem:Seibold} and \Cref{lem:equivalence}, it is easy to see that $T_i(\cC^{\vb}_\del(\xb_i))\cap X\backslash \{ \xb_i\} \neq \emptyset$ for all unit vector $\vb\in\R^d$ implies 
$S_{\del, h, 2}(\xb_i)$ (as defined in \cref{eq:feasibleset}) is not empty. 
Now if we define a new point cloud $\overline{X}$ by replacing all  $\xb_j \in   X\cap\cE_\del^{\xb_i} (\xb_i)\backslash\{\xb_i\}$ with $\overline{\xb_j}$  in $X$. Then since $\overline{\xb_j}$ lies on the line connecting $\xb_i$ and $\xb_j$, we see that 
\[
T_i(\cC^{\vb}_\del(\xb_i))\cap X\backslash \{ \xb_i\} \neq \emptyset \implies T_i(\cC^{\vb}_\del(\xb_i))\cap \overline{X}\backslash \{ \xb_i\} \neq \emptyset ,
\]
and the latter implies  $\overline{S}_{\del, h, 2}(\xb_i)$ is not empty by the same reasoning.
See an illustration of the sets $T_i(\cC^{\vb}_\del(\xb_i))\cap X\backslash \{ \xb_i\} $ and $T_i(\cC^{\vb}_\del(\xb_i))\cap \overline{X}\backslash \{ \xb_i\}$ in \Cref{fig:cone_map}. 
\end{proof}
\begin{figure}[htp]
    \centering
    \begin{tikzpicture}[scale=1.3]
    \coordinate (x0) at (0, 0);

    \coordinate (A) at (0.6, 2.5);
    \coordinate (B) at (2.0, -1.8);
    \draw[line width=0.4mm, name path=A--B] (A) to[out=-40, in=80] (B);

    \draw[rotate around={-63.4322:(x0)}, dashed] ($(x0) + (0:1.6036 and 3)$) arc (0:180:1.6036 and 3);
    \fill (x0) circle (1pt);
    \node at ($(x0) + (0.3, 1.3)$) {$\scriptstyle \mathcal{E}_{\delta}^{\boldsymbol{x}_{i}}(\boldsymbol{x}_{i})$};
    \node at ($(x0) - (0.2, 0.1)$) {$\scriptstyle \boldsymbol{x}_{i}$};
    \node at ($(x0) + (3.5, 0.5)$) {$\scriptstyle T_{i}(\mathcal{C}_{\delta}^{\boldsymbol{v}}(\boldsymbol{x}_{i}))$};
    \filldraw[rotate around={-63.4322:(x0)}, fill opacity=0.3, fill=gray] (x0) -- +(40:1.6036 and 3) arc (40:85:1.6036 and 3) -- cycle;

    \coordinate (x1) at (0.7, -0.03);
    \coordinate (x2) at (1.2, 0.45);
    \coordinate (x3) at (2.35, 0.15);
    \coordinate (x4) at (2.6, 0.7);

    \fill (x1) circle (1pt);
    \fill (x2) circle (1pt);
    \fill (x3) circle (1pt);
    \fill (x4) circle (1pt);

    \node[below] at (x1) {$\scriptstyle \boldsymbol{x}_{j_{1}}$};
    \node[above left] at (x2) {$\scriptstyle \boldsymbol{x}_{j_{2}}$};
    \node[below right] at (x3) {$\scriptstyle \boldsymbol{x}_{j_{3}}$};
    \node[above] at (x4) {$\scriptstyle \boldsymbol{x}_{j_{4}}$};

    \draw[blue, dashed, line width=0.2mm, name path=0--3] (x0) -- (x3);
    \draw[blue, dashed, line width=0.2mm, name path=0--4] (x0) -- (x4);

    \path[name intersections={of=A--B and 0--3, by=E}];
    \fill[orange] (E) circle (1pt);
    \node[below left, red] at (E) {$\scriptstyle \overline{\boldsymbol{x}_{j_{3}}}$};

    \path[name intersections={of=A--B and 0--4, by=F}];
    \fill[orange] (F) circle (1pt);
    \node[above left, red] at (F) {$\scriptstyle \overline{\boldsymbol{x}_{j_{4}}}$};

    \node[below] at (B) {$\scriptstyle \partial \Omega$};
    \node[below] at ($(x0) + (-0.7, 2.2)$) {$\scriptstyle \Omega$};

    \end{tikzpicture}
    \caption{Illustration of \Cref{coro:cone_condition}.}
    \label{fig:cone_map}
\end{figure}

Although \Cref{coro:cone_condition} is a complete characterization of a sufficient condition for the well-posedness of \cref{eq:nonlocalminimization_Om}.
It is hard to use in practice. In the following, we proceed to give a sufficient condition that is easy to use in the case $d=2$. 
We leave the proof of the following theorem in \Cref{sec:appendixA}. 
\begin{thm}
\label{thm:nbh_criterion}
Let $d=2$ or $d=3$ and $h$ be the fill distance defined in \cref{eq:filldistance}. Let $\lambda_1 = \lambda_1(\xb_i)$ denote the smallest eigenvalue of $A(\xb_i)$.
Then there exists a constant $c=c(d)>0$ depending only on $d$ such that if $$\delta \geq c h (\lambda_1)^{-1/2}, $$ then $S_{\del, h, 2}(\xb_i)$ and $\overline{S}_{\del, h, 2}(\xb_i)$ are not empty.
Since $\lambda_1\leq \varrho$,
this implies the existence of positive stencils given $\delta \geq c h (\varrho)^{-1/2}$. 
\end{thm}

\begin{remark}
\label{rem:nbh_est}
Notice that for a given point $\xb$, the elliptical searching region surrounding $\xb$ has semi-axes $\{\del \sqrt{\lambda_i(\xb)} \}_{i=1}^d$ where $\lambda_i(\xb)$ denotes the $i$-th smallest eigenvalue of $A(\xb_i)$. 
The estimate in \Cref{thm:nbh_criterion} is near-optimal in the sense that the semi-minor axis $\del\sqrt{\lambda_1(\xb)}$ of the searching neighborhood can be made proportional to $h$ asymptotically for the existence of positive stencils.
By \Cref{thm:nbh_criterion},  we can choose an elliptical neighborhood of $\xb$ whose volume is proportional to $h^d (\prod_2^{d}\varrho_i)^{-1/2}$,
where $\varrho_i=\varrho_i(\xb):= \lambda_1(\xb)/\lambda_i(\xb)\geq \varrho$. This implies that the number of points within the searching neighborhood of $\xb$ is proportional to $\varrho_2^{-1/2}\leq \varrho^{-1/2}$ in 2d and $(\varrho_2\varrho_3)^{-1/2}\leq \varrho^{-1}$ in 3d. 
However, we do not have an explicit estimate of the constant $c=c(d)$ in \Cref{thm:nbh_criterion}. 
In practice, we estimate this constant numerically which is described in detail in \Cref{subsec:searching}. 
\end{remark}

Combining \Cref{thm:nbh_criterion} with \Cref{thm:convergence}, one can take $\del =c h\varrho^{-1/2} $ and then the convergence rate is given in terms of $h$. 
This is summarized in the following Corollary.   
\begin{coro}
Let $\del =c h\varrho^{-1/2} $ where $c=c(d)$ is the constant in \Cref{thm:nbh_criterion}, then with a generic constant $C>0$,
\beq
\max_{\xb_i\in\Om}|u(\xb_i) - u_\del^h(\xb_i)| \leq C |u|_{C^{2+k,\alpha}(\overline{\Om})} \varrho^{-(k+\alpha)/2} h^{k+\alpha}. 
\eeq 
for $k=0$ or $k=1$ and $\alpha\in (0,1]$. 
\end{coro}


\section{Algorithm Design \& Complexity Analysis}
\label{sec:algorithm}
In this section, we explain our algorithms in detail, mainly focusing on point cloud generation and matrix assembly.

\subsection{Point cloud generation}
In order to perform numerical experiments, some criteria need to be given on the point cloud geometry. 
For this, we first need to define two geometric quantities with respect to point clouds in addition to the fill distance defined in \cref{eq:filldistance}. 
For a point cloud $X =\{ \xb_i\}_{i=1}^M \subset \Om_\del$, we define the separation distance $\zeta$ as
\begin{equation}
\label{eq:separation}
    \zeta
    := \frac{1}{2} \min_{1 \leq i < j \leq M} |\boldsymbol{x}_{i} - \boldsymbol{x}_{j}|. 
\end{equation}
In addition, for the points $\{ \xb_i\}_{i=1}^N$ inside $\Om$, we denote $\kappa$ the minimum distance to the boundary, i.e., 
\begin{equation}
\label{eq:mindistbdry}
    \kappa
    := \min_{1 \leq i \leq N} \text{dist}(\boldsymbol{x}_{i}, \partial \Omega).
\end{equation}
With these geometric quantities, we now define proper point clouds to be used in numerical experiments.  
\begin{defn}
\label{def:pointcloud}
Let $X =\{ \xb_i\}_{i=1}^M \subset \Om_\del$ be a point cloud with its geometric quantities $h$, $\zeta$ and $\kappa$ defined in \cref{eq:filldistance,eq:separation,eq:mindistbdry}, respectively. 
Given a set of positive constants $\{ c_h, c_\zeta, c_\kappa\}$, we say $X$ is a proper point cloud (with respect to the constants $\{ c_h, c_\zeta, c_\kappa\}$) if it satisfies the following conditions:
\begin{enumerate}[(i)]
    \item $h \leq c_h \left ( \frac{|\Om_\del|}{|X|}\right )^{1 / d}$;
    \item $\zeta \geq c_\zeta h$;
    \item $\kappa \geq c_\kappa h$.
\end{enumerate}
Notice that $|\Om_\del|$ denotes the $d$-dimensional Lebesgue measure of $\Om_\del$ and $|X| = M$. 
\end{defn}
\begin{remark}
Condition (ii) in \Cref{def:pointcloud} essentially requires the point cloud to be quasi-uniform (\cite{wendland2004scattered}), and condition (iii) 
requires a certain distance from the interior points to the boundary set so that interior points would not be too close to the boundary points after the mapping described in \Cref{subsec:boundarytreatment} (this allows numerically solving \cref{eq:nonlocalminimization_Om}).
Notice that if (ii) is satisfied, then there exists $C=C(c_\zeta, d)>0$ such that $h\leq C\left ( \frac{|\Om_\del|}{|X|}\right )^{1/d}$.
In practice, we impose condition (i) with a chosen constant $c_h>0$ to have explicit control over the fill distance.  
In our numerical experiments in \Cref{sec:num}, we take $c_h = 1$, $c_\zeta = 0.175$, and $c_\kappa= 0.25$.
\end{remark}

Discussion on the generation of proper point clouds is found in \Cref{sec:appendixB}.

\subsection{Matrix assembly}
The major effort in matrix assembly is the generation of the weights  $\{ \om_{j,i}\}$ defined in \cref{eq:nonlocalminimization_Om}. 
Notice that with respect to each point cloud and coefficient matrix $A(\xb)$, we need to solve $N$ number of linear minimization problems to get the weights where $N$ denotes the number of interior points. 
For each $\xb_i\in \Om$,  we first need to find all the points inside the searching area, i.e., the domain $\cE_\del^{\xb_i}(\xb_i)$ for a given a $\del>0$, then solve the linear minimization problem to get the stencil.  

We now describe the process of finding points inside an elliptical searching area. 
Notice that the proper point cloud is quasi-uniform, an easy way to accelerate this procedure is dividing the domain into same-size axis-aligned blocks \cite{seibold2008minimal,wendland2004scattered}. 
We call these blocks voxels. Alternatively, point clouds can also be managed by k-d trees \cite{bentley1975multidimensional,wendland2004scattered}. To search neighbors in the given elliptical area $\cE_\del^{\xb_i}(\xb_i)$, we first compute which voxel contains the current point $\xb_i$, then search all the neighboring voxels that intersect non-trivially with $\cE_\del^{\xb_i}(\xb_i)$. If a voxel that intersects non-trivially with the searching area is contained in that area, then we add all the points in the voxel to the result set; 
otherwise, points in the voxel need to be checked one by one.
The intersection algorithm of voxels with ellipses/ellipsoids is crucial and we now describe it below. 

Let $\cH$ denote a $d$-dimensional (hyper)rectangle and $\cE$ a $d$-dimensional ellipsoid. $\cH$ and $\cE$ are both open sets.  
We present an intersection detection algorithm in Algorithm \ref{alg:intersection_rectangle_ellipsoid} that distinguishes the following cases:\\
\indent {\it Case 1.} $\cH$ does not intersect with $\cE$;  \\
\indent {\it Case 2.} $\cH$ is contained in $\cE$; \\
\indent {\it Case 3.} $\cH$ intersects with $\cE$ but is not contained in $\cE$. \\

\begin{algorithm}[htbp]
function intersection($\cH$, $\cE$)\\
\uIf{the center of $\cH$ is inside $\cE$ }
{
    \uIf{all the vertices of $\cH$ are inside $\overline{\cE}$ }
    {
        \textbf{return} Case 2; 
    }
    \Else
    {
        \textbf{return} Case 3; 
    }
}
\uElseIf{the center of $\cE$ is inside $\cH$ }
{
    \textbf{return} Case 3; 
}
\Else
{
    \uIf{any one of the faces of $\cH$ intersects with $\cE$ }
    {
        \textbf{return} Case 3; 
    }
    \Else
    {
        \textbf{return} Case 1; 
    }
}
\caption{Intersection detection of hyperrectangles with ellipsoids.}
\label{alg:intersection_rectangle_ellipsoid}
\end{algorithm}

Notice that in Algorithm \ref{alg:intersection_rectangle_ellipsoid}, the most time-consuming part is the intersection detection of faces of $\cH$ with $\cE$. 
In 2d, the faces of a rectangle are line segments. Intersection detection of a line segment with an ellipse is relatively easy to carry out. 
One can first find the intersection (if exists) of the underlying line with the ellipse, which is a line segment (see \cref{fig:ellipse_intersection} as an illustration), by solving a quadratic equation. 
Then the intersection of two line segments can be easily checked.  
In 3d, to check whether a face of a 3d rectangle intersects with an ellipsoid, we first find the intersection area (if exists) of the underlying plane with the ellipse. Then since the intersection area (see \cref{fig:ellipsoid_intersection} as an illustration) is an ellipse, the problem is then reduced to the intersection detection of two-dimensional rectangles with ellipses. 
This can be further extended to higher dimensions, and a $d$-dimensional intersection problem can be reduced to a $d - 1$-dimensional problem by this reasoning. 
Let $Q_{I}(d)$ denote the complexity of the intersection algorithm in $d$ dimensions. Notice that a $d$-dimensional hyperrectangle has $2d$ faces, we can then deduce the recurrence relation
\begin{equation*}
Q_{I}(d) \leq c d  Q_{I}(d - 1) 
\qquad \text{with} \qquad
Q_{I}(1) = \mathcal{O}(1).
\end{equation*}
for some constant $c>0$ independent of $d$. 
Finally, the recurrence relation leads to 
\begin{equation}\label{eq:intersection_complextiy}
Q_{I}(d) = \mathcal{O}(c^{d} d!). 
\end{equation}

\begin{figure}[htp]
    \centering
    \begin{subfigure}[b]{0.475\textwidth}
        \centering
        \begin{tikzpicture}[scale=0.6]
        \begin{scope}[rotate=60, xscale=3, yscale=2]
            \coordinate (O) at (0, 0);
            \draw (O) circle (1);
        \end{scope}
        \coordinate (A) at (-2.6, 1.0);
        \coordinate (B) at ($(A) + (5.6, 0)$);
        \draw[line width=1] (A) -- (B);

        \end{tikzpicture}
        \caption{$d = 2$, ellipse intersects with line}
        \label{fig:ellipse_intersection}
    \end{subfigure}
    \hfill
    \begin{subfigure}[b]{0.475\textwidth}
        \centering
        \begin{tikzpicture}[scale=0.6]
        \begin{scope}[rotate=60, xscale=3, yscale=2, shift={(0,0)}]
            \coordinate (O) at (0,0);

            \draw (O) circle (1);
            \draw[dashed] ($(O) + (1, 0)$) arc (0:180:1 and 0.3);
            \draw ($(O) - (1, 0)$) arc (180:360:1 and 0.3);
            \draw[rotate=90] ($(O) + (1, 0)$) arc (0:180:1 and 0.3);
            \draw[rotate=90, dashed] ($(O) - (1, 0)$) arc (180:360:1 and 0.3);
        \end{scope}
        \coordinate (A) at (-1.6, 3.0);
        \coordinate (B) at ($(A) + (5.8, 0)$);
        \coordinate (C) at ($(B) + (-1.0, -1.4)$);
        \coordinate (D) at ($(A) + (-1.0, -1.4)$);
        \coordinate (E) at ($(A)!0.5!(C) + (-0.08, 0.12)$);
        \draw[line width=1] (A) -- (B) -- (C) -- (D) -- (A);
        \draw[densely dotted, line width=1] (E) ellipse (0.95 and 0.3);

        \end{tikzpicture}
        \caption{$d = 3$, ellipsoid intersects with plane}
        \label{fig:ellipsoid_intersection}
    \end{subfigure}
    \caption{Intersection illustration}
    \label{fig:intersection}
\end{figure}
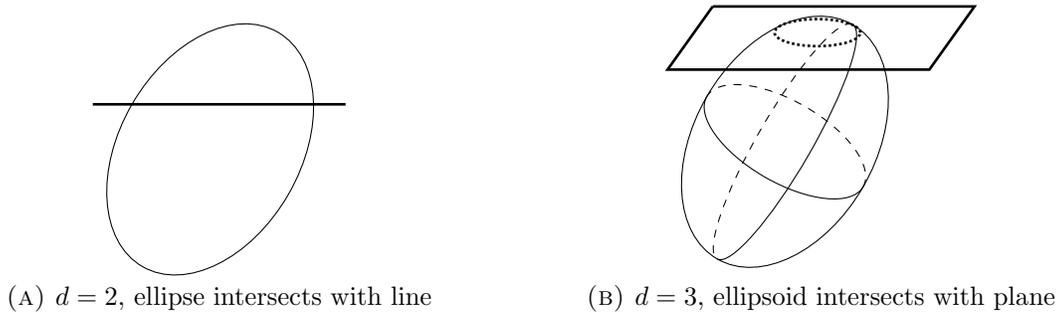

\begin{remark}
The recursive algorithm for face-ellipsoid intersection detection gives a complexity (\cref{eq:intersection_complextiy}) that grows quickly with dimension. In this work, we only consider $d=2$ or $d=3$ so that  $Q_{I}(2)$ or $Q_{I}(3)$ can be treated as constants.
It will be of future interest to explore better intersection detection algorithms in higher dimensions. 
\end{remark}

\begin{remark}
The complexity for testing whether a hyperrectangle is contained in an ellipsoid is less than $Q_{I}(d)$. 
In fact, since a hyperrectangle has $2^d$ vertices, this gives the complexity $\mathcal{O}(2^{d} d^{2})$.
Notice that when the ellipsoid is not aligned with the axes, we need to do the mapping with complexity $\mathcal{O}(d^{2})$ first, and then check with complexity $\mathcal{O}(d)$ for each vertex.
\end{remark}

For $\xb_i\in\Om$, let $q(\xb_i)$ denote the number of points in the searching area $\cE_\del^{\xb_i}(\xb_i)$, then the corresponding searching process needs $\mathcal{O}(q(\xb_i))$ intersection detections. For point clouds managed by k-d trees, it can be shown that $\mathcal{O}(q(\xb_i) \log M)$ intersection detections are needed for such range query \cite{wendland2004scattered}.
Once we find all the points inside $\cE_\del^{\xb_i}(\xb_i)$, we proceed to solve the linear minimization problem \cref{eq:nonlocalminimization_Om}. Recall that the constrained linear minimization problem \cref{eq:nonlocalminimization_Om} is a linear programming problem, hence we adopt the simplex method \cite{dantzig1990origins}, which on average has a linear complexity in $q(\xb_i)$ when the dimension is fixed \cite{smale1983average}.

Combining the above discussions, when the dimension $d$ is fixed, the total average complexity of finding a stencil for a  given interior point $\xb_i$ is $\mathcal{O}(q(\xb_i))$.
Note that by the quasi-uniform assumption (condition (2) in \Cref{def:pointcloud}), we have $q(\xb_i) \leq C \frac{| \cE_\del^{\xb_i}(\xb_i)|}{h^d}$ for some $C>0$. The volume $| \cE_\del^{\xb_i}(\xb_i)|$ depends on $\del$ and the coefficient matrix $A(\xb_i)$. 
In practice, we take $h$ to be proportional to $\del$, and therefore  $ \frac{| \cE_\del^{\xb_i}(\xb_i)|}{h^d} = \cO(1)$ considering $A(\xb)$ to be fixed. 
As a result, the total complexity of searching for neighbors near a given point can be considered as a constant. 
In the near degenerate case, i.e., $\varrho\ll 1$, $q(\xb_i)$ may grow with the decrease of $\varrho$ as mentioned in \Cref{rem:nbh_est}.

Traversing all $N$ number of interior points, we can get all the weights $\{ \om_{j,i}\}$ to complete the matrix assembly process. Therefore, the whole complexity of assembling a matrix is given by $\mathcal{O}(N)$ for a fixed problem. In addition, notice that the weights generation process is embarrassingly parallelizable, the actual computational time can be further reduced by parallelization.

\begin{remark}
From the $l_1$ type minimization problem, we get a minimal positive stencil, and therefore the assembled matrix is sparse. 
It is also recommended that a reindexing process be applied to the point cloud to reduce the bandwidth of the assembled matrix. The simplest way to do this is to sort all the interior points by coordinates so that the index distance between two close points is not too large.
\end{remark}

\begin{remark}
One may encounter memory issues using exact solvers when the linear system gets large. 
Iterative methods can be used in the case of large and sparse linear systems.  
We use the biconjugate gradient stabilized method (BiCGSTAB) \cite{van1992bi} to approximately solve the sparse linear system when $N$ is large. 
\end{remark}

\subsection{Searching area estimate}
\label{subsec:searching}
\Cref{thm:nbh_criterion} does not specify the constant $c=c(d)>0$ which determines the searching neighborhoods.
Here we discuss how to determine the searching neighborhoods in practice.
For a given fill distance $h$, we let $\del = c h (\varrho)^{-1/2}$, where the determination of $c>0$ is described below. 
Then the searching neighborhood of a point $\xb$ is the domain $\cE_\del^{\xb}(\xb)$. 
 
We now discuss the choice of $c>0$ in practice. 
We first discuss the 2d case,  and then use the 2d result to approximately estimate the searching area in 3d.
Without loss of generality, we fix $\boldsymbol{x}_{i} \in \Omega$ and assume that
\begin{equation*}
A(\boldsymbol{x}_{i})
=
\begin{pmatrix}
\varrho & 0\\
0 & 1
\end{pmatrix}.
\end{equation*}
According to \Cref{lem:appendix_1}, we need to find the smallest radius of the inscribed circles of the domains $\{ T_{i}(\mathcal{C}_{\del}^{\boldsymbol{v}}(\boldsymbol{x}_{i}))\}_{\vb\in\R^2, |\vb|=1}$. The problem is a rescaling of the case $\del=1$, as illustrated by \cref{fig:inscribed_circle}.
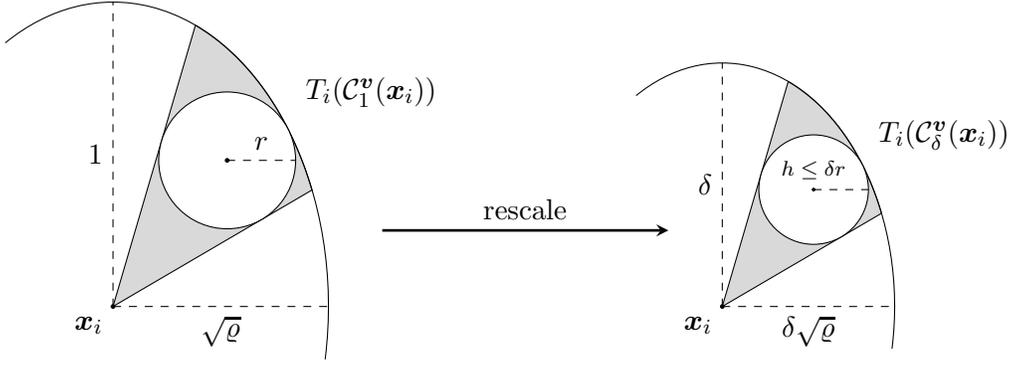
\begin{figure}[htp]
    \centering
    \begin{tikzpicture}[scale=0.9]
    \coordinate (x0) at (0, 0);

    \pgfmathsetmacro{\FACTOR}{4.5}
    \pgfmathsetmacro{\shift}{\FACTOR * 2}
    \pgfmathsetmacro{\rescale}{0.8}

    \pgfmathsetmacro{\rho}{0.5}
    \pgfmathsetmacro{\varTheta}{22.5}
    \pgfmathsetmacro{\varDelta}{45}
    \pgfmathsetmacro{\t}{0.449471706233453 / \FACTOR}

    \pgfmathsetmacro{\semiMinor}{sqrt(\rho) * \FACTOR}
    \pgfmathsetmacro{\semiMajor}{\FACTOR}

    \pgfmathsetmacro{\pAx}{\semiMinor * cos(\varTheta)}
    \pgfmathsetmacro{\pAy}{\semiMajor * sin(\varTheta)}
    \pgfmathsetmacro{\dA}{sqrt(\pAx ^ 2 + \pAy ^ 2)}
    \pgfmathsetmacro{\pBx}{\semiMinor * cos(\varTheta + \varDelta)}
    \pgfmathsetmacro{\pBy}{\semiMajor * sin(\varTheta + \varDelta)}
    \pgfmathsetmacro{\dB}{sqrt(\pBx ^ 2 + \pBy ^ 2)}

    \coordinate (p1) at ($(x0) + (\pAx, \pAy)$);
    \coordinate (p2) at ($(x0) + (\pBx, \pBy)$);

    \pgfmathsetmacro{\angleStart}{-10}
    \pgfmathsetmacro{\angleEnd}{120}

    \pgfmathsetmacro{\pOx}{\t * (\dB * \pAx + \dA * \pBx)}
    \pgfmathsetmacro{\pOy}{\t * (\dB * \pAy + \dA * \pBy)}
    \pgfmathsetmacro{\R}{\t * abs(\pAx * \pBy - \pAy * \pBx)}

    \coordinate (p0) at ($(x0) + (\pOx, \pOy)$);

    \draw ($(x0) + (\angleStart:{\semiMinor} and {\semiMajor})$) arc (\angleStart:\angleEnd:{\semiMinor} and {\semiMajor});
    \fill (x0) circle (1pt);
    \node[below left] at (x0) {$\boldsymbol{x}_{i}$};

    \filldraw[fill opacity=0.3, fill=gray] (x0) -- +(\varTheta:{\semiMinor} and {\semiMajor}) arc (\varTheta:\varTheta + \varDelta:{\semiMinor} and {\semiMajor}) -- cycle;
    \filldraw[fill=white] (p0) circle (\R);
    \fill (p0) circle (1pt);

    \draw[dashed] (x0) -- node[midway, below] {$\sqrt{\varrho}$} ($(x0) + (\semiMinor, 0)$);
    \draw[dashed] (x0) -- node[midway, left] {$1$} ($(x0) + (0, \semiMajor)$);
    \draw[dashed] (p0) -- node[midway, above] {$r$} ($(p0) + (\R, 0)$);
    \node[right] at ($(p0) + (\R, \R)$) {$T_{i}(\mathcal{C}_{1}^{\boldsymbol{v}}(\boldsymbol{x}_{i}))$};

    \draw ($(x0) + (\shift, 0) + (\angleStart:{\semiMinor * \rescale} and {\semiMajor * \rescale})$) arc (\angleStart:\angleEnd:{\semiMinor * \rescale} and {\semiMajor * \rescale});
    \fill ($(x0) + (\shift, 0)$) circle (1pt);
    \node[below left] at ($(x0) + (\shift, 0)$) {$\boldsymbol{x}_{i}$};

    \filldraw[fill opacity=0.3, fill=gray] ($(x0) + (\shift, 0)$) -- +(\varTheta:{\semiMinor * \rescale} and {\semiMajor * \rescale}) arc (\varTheta:\varTheta + \varDelta:{\semiMinor * \rescale} and {\semiMajor * \rescale}) -- cycle;
    \filldraw[fill=white] ($(x0) + (\pOx * \rescale, \pOy * \rescale) + (\shift, 0)$) circle (\R * \rescale);
    \fill ($(x0) + (\pOx * \rescale, \pOy * \rescale) + (\shift, 0)$) circle (1pt * \rescale);

    \draw[dashed] ($(x0) + (\shift, 0)$) -- node[midway, below] {$\delta \sqrt{\varrho}$} ($(x0) + (\semiMinor * \rescale, 0) + (\shift, 0)$);
    \draw[dashed] ($(x0) + (\shift, 0)$) -- node[midway, left] {$\delta$} ($(x0) + (0, \semiMajor * \rescale) + (\shift, 0)$);
    \draw[dashed] ($(x0) + (\pOx * \rescale, \pOy * \rescale) + (\shift, 0)$) -- ($(x0) + (\pOx * \rescale, \pOy * \rescale) + (\R * \rescale, 0) + (\shift, 0)$);
    \node[above] at ($(x0) + (\pOx * \rescale, \pOy * \rescale) + (\shift, 0)$) {\scriptsize $h \leq \delta r$};
    \node[right] at ($(x0) + (\pOx * \rescale, \pOy * \rescale) + (\R * \rescale, \R * \rescale) + (\shift, 0)$) {$T_{i}(\mathcal{C}_{\delta}^{\boldsymbol{v}}(\boldsymbol{x}_{i}))$};

    \draw[-stealth, line width=1pt] ($(x0) + (1.25 * \semiMinor, 0.25 * \semiMajor)$) -- node[midway, above] {rescale} ($(x0) + (\shift - 0.25 * \semiMinor, 0.25 * \semiMajor)$);

    \end{tikzpicture}
    \caption{An illustration of $T_{i}(\mathcal{C}_{\del}^{\boldsymbol{v}}(\boldsymbol{x}_{i}))$ and its inscribed circle as a rescaling from the $\del=1$ case.}
    \label{fig:inscribed_circle}
\end{figure}
Therefore we only need to consider the case $\del=1$ and find $r(\varrho):=\min_{\vb\in\R^2, |\vb|=1} T_{i}(\mathcal{C}_{1}^{\boldsymbol{v}}(\boldsymbol{x}_{i}))$. 
The detailed procedure for finding $r(\varrho)$ numerically is provided in \Cref{sec:appendixC}.
One may choose $c =\max_{\varrho \in (0, 1]} \sqrt{\varrho}/r(\varrho) $ and then by letting $\del =  c h (\varrho)^{-1/2}$ we have the desired relation $h\leq \del r$. 
In practice, we find that the $\sqrt{\varrho}/r(\varrho) $ is a bit smaller with smaller $\varrho > 0$ using the estimate of $r(\varrho)$.
Since smaller constant $c$ leads to a smaller searching neighborhood, we, therefore, suggest taking different $c$ for different values of $\varrho \in (0, 1]$.
In particular, in 2d, by numerical approximations, we have the following estimate of $c$,
\[
c = c_{\text{2d}}(\varrho) := 2.836 \chi_{(0, 0.01]}(\varrho) + 2.901 \chi_{(0.01, 0.1]}(\varrho) + 3.614 \chi_{(0.1, 1]}(\varrho).
\]
In 3d, it is difficult to estimate the radius of the inscribed ball in  $T_{i}(\mathcal{C}_{\del}^{\boldsymbol{v}}(\boldsymbol{x}_{i}))$. 
Therefore, we only take the intersecting ellipses of a given ellipsoid with the three planes that go through its principal axes, and
perform the 2d estimate described above to obtain an estimate of the constant $c > 0$ in 3d.  
The result is given as follows.
\[
c = c_{\text{3d}}(\varrho) := 3.623 \chi_{(0, 0.01]}(\varrho) + 3.776 \chi_{(0.01, 0.1]}(\varrho) + 4.450 \chi_{(0.1, 1]}(\varrho).
\]

In practice, we find that solutions often exist for even smaller searching areas and this means that one may take even smaller values of $c>0$ to further 
decrease the computational cost. We suggest taking $c = \frac{1}{\sqrt{3}} c_{\text{\normalfont 2d}}(\varrho) \approx 0.58 c_{\text{\normalfont 2d}}(\varrho)$ in 2d and $c = \frac{1}{\sqrt[3]{18}} c_{\text{\normalfont 3d}} \approx 0.38 C_{\text{\normalfont 3d}}(\varrho)$ in 3d first, and if no solution exists resetting $c = c_{\text{\normalfont 2d}}(\varrho)$ in 2d and $c = c_{\text{\normalfont 3d}}(\varrho)$ in 3d.
This procedure could reduce the number of points in a searching neighborhood by a large factor. 


\section{Numerical Results}
\label{sec:num}
In this section, we report the results of numerical experiments for the study of the numerical accuracy of our method.
We present 2d numerical results in \Cref{sec:num2d} and 3d numerical results in \Cref{sec:num3d}. 

\subsection{2d numerical tests}
\label{sec:num2d}
We test our numerical algorithm in 2d using two domains. 
The first domain is a unit disk given by $\{ x_1^2+x_2^2<1\}$, and the second domain is an
L-shaped domain given by  $(-1,1)^2\backslash [0,1]^2$.
For the kernel function, we use a truncated fractional kernel $\ga(r) = C_\alpha r^{-\alpha} \chi_{\{ |r|<1\}}$ with $\alpha \in (2, d+2)$ that satisfies \cref{eq:kernel_secondmoment}. By \cite[Theorem 6]{seibold2008minimal}, $\alpha>2$ is necessary for the linear programming problem to select points close to the central point.
For both 2d and 3d, we use $\alpha =3$ in our numerical experiments.
Other choices such as the truncated Gaussian kernel (\cite{aksoylu2020choice}) may also be used and one can observe similar convergence rates with proper parameter tuning of the truncated Gaussian kernel. Detailed discussions of other kernel selections are omitted.
We implement the numerical algorithm with $p=2$. 
Smooth manufactured solutions are used in our tests with the right-hand side of \cref{eq:elliptic} computed based on them. 

\subsubsection{Tests for continuous coefficient matrices}
We first test our algorithm for continuous coefficient matrices. Our baseline is $A_{0}(\boldsymbol{x}) = I$ with $\varrho = 1$.
A list of coefficient matrices used in numerical experiments is given below. 
\begin{equation*}
\begin{matrix}
\# & A(\boldsymbol{x}) & \varrho\\
1 &
\begin{pmatrix}
1 - 0.5 |x_{1}| & 0\\
0 & 0.25 + 0.25 |x_{2}|
\end{pmatrix}
& 0.2500\\[1em]
2 & \displaystyle \frac{1}{2.21}
\begin{pmatrix}
2 - |x_{1}| & 0.5\\
0.5 & 0.5 + 0.5 |x_{2}|
\end{pmatrix}
& 0.0864\\[1em]
3 &
\begin{pmatrix}
1 - 0.5 |x_{1}| & 0\\
0 & 0.025 + 0.025 |x_{2}|
\end{pmatrix}
& 0.0250\\[1em]
4 &
\begin{pmatrix}
1 - 0.5 |x_{1}| & 0\\
0 & 0.0025 + 0.0025 |x_{2}|
\end{pmatrix}
& 0.0025\\[1em]
5 & \displaystyle \frac{1}{2.001}
\begin{pmatrix}
2 - |x_{1} (0.5 - x_{2})| & 0.025\\
0.025 & 0.01 + 0.0025 x_{1} \exp(x_{2})
\end{pmatrix}
& 0.0014
\end{matrix}
\end{equation*}
Notice that the value $\varrho$ is computed approximately in the domain $[-1, 1]^{2}$, which contains both the unit disk and the L-shaped domain as subsets. 
Numerical results are presented in \cref{fig:2d_continuous_test_1,fig:2d_continuous_test_2,fig:2d_continuous_test_3}. 
We observe second-order convergence in $h$ for all cases, which is better than the theoretical analysis in \Cref{thm:convergence} for $p=2$. 
This superconvergence phenomenon may likely be due to the cancellation of terms as mentioned in  \Cref{rmk:truncationerror,rmk:superconvergence}. 
In \cref{fig:2d_continuous_test_1,fig:2d_continuous_test_2}, we test our method on two manufactured solutions $u_{1}^{(2d)}(x_1,x_2) = x_1  x_2 +\cos(x_1)\exp(x_2)$ and 
$u_{2}^{(2d)}(x_1,x_2) = (x_1+x_2)^4 \cos(x_1(x_1+2x_2))$. The numerical errors in these graphs grow as  $\varrho$ becomes smaller as predicted by theory. 
In some very special cases, the numerical errors may behave differently as  $\varrho\to0$, and one example is with $u_{3}^{(2d)}(x_1,x_2) =x_1^2 +\sin(x_2)\exp(x_2^2 - 1)$
illustrated by \cref{fig:2d_continuous_test_3}. 
The reason for this abnormal behavior is because the elliptic operator degenerates to $\partial^2_{x_1}$ as $\varrho\to0$ by our choices of $A(\xb)$ and the exact solution in this case is 
a second-order polynomial in $x_1$ which can be exactly reproduced by our method. 

\begin{figure}[htp]
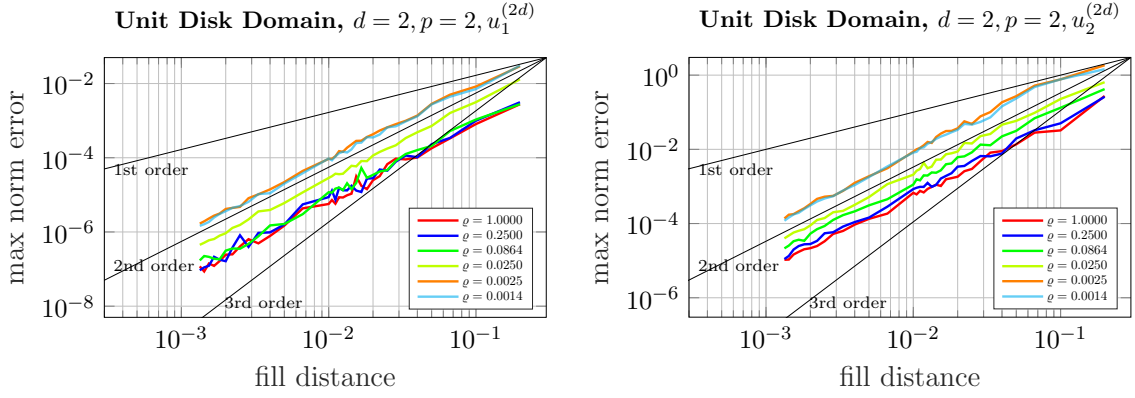

    \centering
    \begin{subfigure}[b]{0.490\textwidth}
        \centering
        \resizebox{\textwidth}{!}{

        }
        \phantomsubcaption
    \end{subfigure}
    \caption{2d tests on the unit disk domain with continuous coefficient matrices}
    \label{fig:2d_continuous_test_1}
\end{figure}

\begin{figure}[htp]
    \centering
    \begin{subfigure}[b]{0.490\textwidth}
        \centering
        \resizebox{\textwidth}{!}{
        %
        }
        \phantomsubcaption
    \end{subfigure}
    \caption{2d tests on the L-shaped domain with continuous coefficient matrices}
    \label{fig:2d_continuous_test_2}
\end{figure}

\begin{figure}[htp]
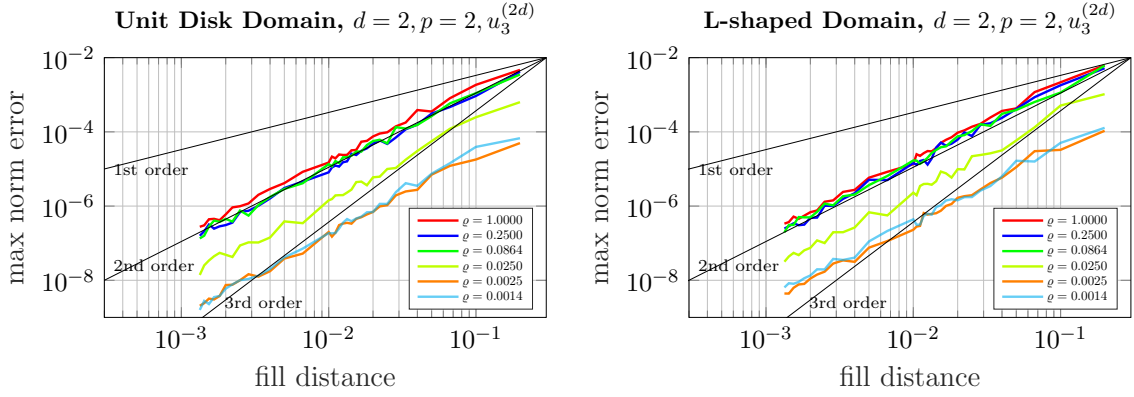

    \centering
    \begin{subfigure}[b]{0.490\textwidth}
        \centering
        \resizebox{\textwidth}{!}{
        %
        }
        \phantomsubcaption
    \end{subfigure}
    \caption{2d tests for a separable function}
    \label{fig:2d_continuous_test_3}
\end{figure}

\subsubsection{Tests for discontinuous coefficient matrices}
We now show numerical results for discontinuous coefficient matrices. 
Notice that when $A(\xb)$ is discontinuous, the elliptic equation in the non-divergence form cannot be recast into a variational form. 
Therefore, non-variational methods are especially important in this case. 
We divide the computational domains into smaller blocks and define piecewise constant coefficient matrices with respect to the blocks. 
More specifically, for $n\in\N$, we divide the domain $[-1, 1]^{2}$ into $(2n+1)^2$ blocks, and define the corresponding piecewise constant 
coefficient matrix 
\[
A_{\psi}(\boldsymbol{x}, n) := (B_{\psi}(\boldsymbol{x}, n) + B_{\psi}^{T}(\boldsymbol{x}, n) + 4 I) / 8,
\]
where $B_{\psi}(\boldsymbol{x}, n)$ is genereated by \texttt{mt19937}\cite{matsumoto1998mersenne} (a pesudorandom number generator) with seed $\psi(\boldsymbol{x}, n) := \text{round}(x_{1} * n) * 2 + \text{round}(x_{2} * n) * 3 \mod 2^{32}$. Here $\text{round}(x)$ maps $x$ to the closest integer.
A list of coefficient matrices used in our experiments is given below. 
\begin{equation*}
\begin{matrix}
\# & A(\boldsymbol{x}) & \varrho & \text{description}\\
6 & A_{\psi}(\boldsymbol{x}, 10^{10}) & 0.2500 & \text{dense blocks}\\
7 & A_{\psi}(\boldsymbol{x}, 10^{4}) & 0.2500 & \text{medium blocks}\\
8 & A_{\psi}(\boldsymbol{x}, 10^{0}) & 0.2500 & \text{loose blocks}
\end{matrix}
\end{equation*}
In addition, we have the last example matrix:
\begin{equation*}
A_{9}(\boldsymbol{x}) =
\left \{
\begin{aligned}
&A_{2}(\boldsymbol{x}), & x_{1} < 0\\
&A_{3}(\boldsymbol{x}), &\text{otherwise}\\
\end{aligned}
\right .
\qquad \text{with} \qquad
\varrho = 0.0250.
\end{equation*}
Numerical results are presented in \cref{fig:2d_discontinuous_test_1,fig:2d_discontinuous_test_2}.
We observe similar second-order convergence in $h$ for all cases. 

\begin{figure}[htp]
    \centering
    \begin{subfigure}[b]{0.490\textwidth}
        \centering
        \resizebox{\textwidth}{!}{
        \begin{tikzpicture}
        \begin{axis}[%
        width=2.55in,
        height=1.5in,
        scale only axis,
        xmode=log,
        xmin=3e-04,
        xmax=0.3,
        xminorticks=true,
        xlabel style={font=\color{white!15!black}},
        xlabel={\large fill distance},
        ymode=log,
        ymin=3e-09,
        ymax=0.03,
        yminorticks=true,
        ylabel style={font=\color{white!15!black}},
        ylabel={\large max norm error},
        title style={align=center, font=\bfseries},
        title={\small Unit Disk Domain, $d = 2, p = 2, u_{1}^{(2d)}$},
        xmajorgrids,
        xminorgrids,
        ymajorgrids,
        yminorgrids,
        legend style={at={(0.97,0.03)}, nodes={scale=0.5, transform shape}, anchor=south east, legend cell align=left, align=left, draw=white!15!black}
        ]
        \addplot [color=red, line width=1pt] table[row sep=crcr]{%
        0.198983854587716   0.00296577793017649\\
        0.0998705812472178  0.000817950495761721\\
        0.0666408526555842  0.000385874382076157\\
        0.0499910849999656  0.000189764879025001\\
        0.0399982806717673  0.000138949269300248\\
        0.0333321068124764  8.67543134694948e-05\\
        0.0285696201019069  6.35899456433719e-05\\
        0.0249919930274932  9.37092977401077e-05\\
        0.0222120342891224  5.88060487596298e-05\\
        0.0199999963261506  2.68230149016446e-05\\
        0.0181812956477009  2.80411377995282e-05\\
        0.0166651641450463  1.85842442277373e-05\\
        0.015384328963273   2.48048283424396e-05\\
        0.0142856175244916  1.33831664168582e-05\\
        0.0133326387907105  1.0669517885975e-05\\
        0.0124986840340357  9.69307713627465e-06\\
        0.0117645043173365  1.18817889962486e-05\\
        0.0111102401232889  1.19780008527304e-05\\
        0.0105261748821938  1.0703669259815e-05\\
        0.00999921217628962 7.91856919635237e-06\\
        0.00666657614059933 3.52617312193715e-06\\
        0.00499998131875709 1.53046819484359e-06\\
        0.00399999800442724 2.48179027861894e-06\\
        0.00333331174692903 1.21078977421973e-06\\
        0.0028571421843815  5.84068232223345e-07\\
        0.00249999897088586 4.67324691388882e-07\\
        0.00222222129568555 4.62395989497821e-07\\
        0.00199999219026464 2.2220215445401e-07\\
        0.00181818033591482 1.70191301274158e-07\\
        0.00166666587024734 2.15901772682159e-07\\
        0.00153845446252526 1.29261109815815e-07\\
        0.00142856662480272 1.62499884437395e-07\\
        0.00133333308002394 1.737635195731e-07\\
        };
        \addlegendentry{$\varrho = 0.2500$, dense blocks} 

        \addplot [color=blue, line width=1pt] table[row sep=crcr]{%
        0.198983854587716   0.00355648974374634\\
        0.0998705812472178  0.000785753512027032\\
        0.0666408526555842  0.000330468527856231\\
        0.0499910849999656  0.000168796897087153\\
        0.0399982806717673  0.000148511419869379\\
        0.0333321068124764  6.44837627343797e-05\\
        0.0285696201019069  8.11517258654249e-05\\
        0.0249919930274932  6.34163346477923e-05\\
        0.0222120342891224  4.01840633155182e-05\\
        0.0199999963261506  4.26752879674197e-05\\
        0.0181812956477009  2.40296734583811e-05\\
        0.0166651641450463  1.28112928257096e-05\\
        0.015384328963273   2.26270364445114e-05\\
        0.0142856175244916  1.57890936112182e-05\\
        0.0133326387907105  1.60251843386305e-05\\
        0.0124986840340357  1.07352463658561e-05\\
        0.0117645043173365  9.36355539593592e-06\\
        0.0111102401232889  5.86217559805924e-06\\
        0.0105261748821938  8.3472212453195e-06\\
        0.00999921217628962 5.11915551726361e-06\\
        0.00666657614059933 2.36789948304672e-06\\
        0.00499998131875709 1.42074302056372e-06\\
        0.00399999800442724 2.16088483839805e-06\\
        0.00333331174692903 8.5593584908672e-07\\
        0.0028571421843815  5.10724368463222e-07\\
        0.00249999897088586 7.9809988084989e-07\\
        0.00222222129568555 7.40206266724996e-07\\
        0.00199999219026464 2.70090640963616e-07\\
        0.00181818033591482 1.51250756141863e-07\\
        0.00166666587024734 1.57944943346422e-07\\
        0.00153845446252526 1.43531939778008e-07\\
        0.00142856662480272 1.35552208124778e-07\\
        0.00133333308002394 1.36157126018333e-07\\
        };
        \addlegendentry{$\varrho = 0.2500$, medium blocks} 

        \addplot [color=green, line width=1pt] table[row sep=crcr]{%
        0.198983854587716   0.00280007433798324\\
        0.0998705812472178  0.000869733727272504\\
        0.0666408526555842  0.000449278302583345\\
        0.0499910849999656  0.000177322748520781\\
        0.0399982806717673  0.000104326400050425\\
        0.0333321068124764  0.000106695023052161\\
        0.0285696201019069  7.41443138907449e-05\\
        0.0249919930274932  6.02055132057355e-05\\
        0.0222120342891224  4.20003861414386e-05\\
        0.0199999963261506  4.3875728019005e-05\\
        0.0181812956477009  2.64345452367643e-05\\
        0.0166651641450463  1.58369140550452e-05\\
        0.015384328963273   2.36843893275385e-05\\
        0.0142856175244916  1.03797121620008e-05\\
        0.0133326387907105  1.50299028081058e-05\\
        0.0124986840340357  1.0901457164092e-05\\
        0.0117645043173365  7.99205777379797e-06\\
        0.0111102401232889  5.6687331759786e-06\\
        0.0105261748821938  6.25804540654862e-06\\
        0.00999921217628962 7.08605099641879e-06\\
        0.00666657614059933 2.12929541842932e-06\\
        0.00499998131875709 1.73093881139508e-06\\
        0.00399999800442724 1.41274249054746e-06\\
        0.00333331174692903 8.88725174341687e-07\\
        0.0028571421843815  4.28351885628508e-07\\
        0.00249999897088586 5.59158756185951e-07\\
        0.00222222129568555 6.87344192007444e-07\\
        0.00199999219026464 2.25612364901373e-07\\
        0.00181818033591482 2.89557771537829e-07\\
        0.00166666587024734 1.24594258910093e-07\\
        0.00153845446252526 1.61242538654349e-07\\
        0.00142856662480272 1.14882284352191e-07\\
        0.00133333308002394 1.17091270901426e-07\\
        };
        \addlegendentry{$\varrho = 0.2500$, loose blocks} 

        \addplot [color=orange, line width=1pt] table[row sep=crcr]{%
        0.198983854587716   0.00892909015955623\\
        0.0998705812472178  0.00223189608313579\\
        0.0666408526555842  0.00119999166646689\\
        0.0499910849999656  0.000410373393363805\\
        0.0399982806717673  0.000334123455114188\\
        0.0333321068124764  0.000257003287522983\\
        0.0285696201019069  0.000136089798223882\\
        0.0249919930274932  0.000149147238073022\\
        0.0222120342891224  0.000101163522673531\\
        0.0199999963261506  8.09309460816365e-05\\
        0.0181812956477009  4.96047863265403e-05\\
        0.0166651641450463  4.44955302061345e-05\\
        0.015384328963273   4.57720211417545e-05\\
        0.0142856175244916  4.22705507139298e-05\\
        0.0133326387907105  4.66701866286279e-05\\
        0.0124986840340357  2.7802642346475e-05\\
        0.0117645043173365  3.12202006814699e-05\\
        0.0111102401232889  3.035065504009e-05\\
        0.0105261748821938  2.9146945007108e-05\\
        0.00999921217628962 2.87898819866239e-05\\
        0.00666657614059933 6.20376774174147e-06\\
        0.00499998131875709 3.64143127251637e-06\\
        0.00399999800442724 2.37842446604297e-06\\
        0.00333331174692903 1.48691920420774e-06\\
        0.0028571421843815  1.74489678106404e-06\\
        0.00249999897088586 7.40969757551113e-07\\
        0.00222222129568555 7.59466624700167e-07\\
        0.00199999219026464 5.77473601826028e-07\\
        0.00181818033591482 5.18026723295151e-07\\
        0.00166666587024734 3.47569045988649e-07\\
        0.00153845446252526 2.87523793218725e-07\\
        0.00142856662480272 1.90151078216516e-07\\
        0.00133333308002394 1.84854642926169e-07\\
        };
        \addlegendentry{$\varrho = 0.0250$, two parts} 

        \addplot [color=black, forget plot]
          table[row sep=crcr]{%
        0.3   0.03\\
        3e-05  3e-06\\
        };
        \node[right] at (3e-04, 3e-05) {\tiny 1st order};
        \addplot [color=black, forget plot]
          table[row sep=crcr]{%
        0.3   0.03\\
        3e-05  3e-10\\
        };
        \node[above right] at (3e-04, 3e-08) {\tiny 2nd order};
        \addplot [color=black, forget plot]
          table[row sep=crcr]{%
        0.3   0.03\\
        3e-05  3e-14\\
        };
        \node[above right] at (1.7e-03, 3e-09) {\tiny 3rd order};
        \end{axis}

        \end{tikzpicture}
        }
        \phantomsubcaption
    \end{subfigure}
    \hfill
    \begin{subfigure}[b]{0.490\textwidth}
        \centering
        \resizebox{\textwidth}{!}{
        \begin{tikzpicture}
        \begin{axis}[%
        width=2.55in,
        height=1.5in,
        scale only axis,
        xmode=log,
        xmin=3e-04,
        xmax=0.3,
        xminorticks=true,
        xlabel style={font=\color{white!15!black}},
        xlabel={\large fill distance},
        ymode=log,
        ymin=3e-07,
        ymax=3,
        yminorticks=true,
        ylabel style={font=\color{white!15!black}},
        ylabel={\large max norm error},
        title style={align=center, font=\bfseries},
        title={\small Unit Disk Domain, $d = 2, p = 2, u_{2}^{(2d)}$},
        xmajorgrids,
        xminorgrids,
        ymajorgrids,
        yminorgrids,
        legend style={at={(0.97,0.03)}, nodes={scale=0.5, transform shape}, anchor=south east, legend cell align=left, align=left, draw=white!15!black}
        ]
        \addplot [color=red, line width=1pt] table[row sep=crcr]{%
        0.198983854587716   0.290680232162311\\
        0.0998705812472178  0.0556517041193813\\
        0.0666408526555842  0.0304011550775205\\
        0.0499910849999656  0.0203948632126013\\
        0.0399982806717673  0.0120557304402611\\
        0.0333321068124764  0.00818793657756212\\
        0.0285696201019069  0.00537074992418685\\
        0.0249919930274932  0.00366322512305572\\
        0.0222120342891224  0.0033637624530235\\
        0.0199999963261506  0.00290535996680119\\
        0.0181812956477009  0.0023619781904034\\
        0.0166651641450463  0.00217139846468339\\
        0.015384328963273   0.00161166579949024\\
        0.0142856175244916  0.00145485946985524\\
        0.0133326387907105  0.00121446119260005\\
        0.0124986840340357  0.00114154339272787\\
        0.0117645043173365  0.00118481498051382\\
        0.0111102401232889  0.00117145448869926\\
        0.0105261748821938  0.000955260186398688\\
        0.00999921217628962 0.00075748876314119\\
        0.00666657614059933 0.000287296153133165\\
        0.00499998131875709 0.000158575799889205\\
        0.00399999800442724 0.000120963562928234\\
        0.00333331174692903 9.23941655375238e-05\\
        0.0028571421843815  6.19234749565933e-05\\
        0.00249999897088586 4.38022711833908e-05\\
        0.00222222129568555 3.01008356410337e-05\\
        0.00199999219026464 3.40848919577752e-05\\
        0.00181818033591482 2.79719368224685e-05\\
        0.00166666587024734 1.99228192523382e-05\\
        0.00153845446252526 1.8882725239977e-05\\
        0.00142856662480272 1.59906073345928e-05\\
        0.00133333308002394 1.65201968208573e-05\\
        };
        \addlegendentry{$\varrho = 0.2500$, dense blocks} 

        \addplot [color=blue, line width=1pt] table[row sep=crcr]{%
        0.198983854587716   0.288449355737296\\
        0.0998705812472178  0.0565835715681466\\
        0.0666408526555842  0.0322753322335778\\
        0.0499910849999656  0.0196670468186998\\
        0.0399982806717673  0.00997916670010202\\
        0.0333321068124764  0.00836845902741334\\
        0.0285696201019069  0.00715235089410104\\
        0.0249919930274932  0.00454863666073635\\
        0.0222120342891224  0.0034582800460663\\
        0.0199999963261506  0.00345749448102395\\
        0.0181812956477009  0.00215583915302997\\
        0.0166651641450463  0.00203526121671566\\
        0.015384328963273   0.00159620687649675\\
        0.0142856175244916  0.00139481718583101\\
        0.0133326387907105  0.00133264052565751\\
        0.0124986840340357  0.00142703817509793\\
        0.0117645043173365  0.00100286648201742\\
        0.0111102401232889  0.00113581690430942\\
        0.0105261748821938  0.000743368465857652\\
        0.00999921217628962 0.000927964875711007\\
        0.00666657614059933 0.000222247391379493\\
        0.00499998131875709 0.000171201084279815\\
        0.00399999800442724 0.000116150111502566\\
        0.00333331174692903 8.97925138801936e-05\\
        0.0028571421843815  8.49037424776045e-05\\
        0.00249999897088586 4.76451254285415e-05\\
        0.00222222129568555 3.43969438404423e-05\\
        0.00199999219026464 3.06580563519399e-05\\
        0.00181818033591482 2.7968670083478e-05\\
        0.00166666587024734 2.19695562128086e-05\\
        0.00153845446252526 1.77460453272271e-05\\
        0.00142856662480272 1.57929237327714e-05\\
        0.00133333308002394 1.45905873818464e-05\\
        };
        \addlegendentry{$\varrho = 0.2500$, medium blocks} 

        \addplot [color=green, line width=1pt] table[row sep=crcr]{%
        0.198983854587716   0.321072020232384\\
        0.0998705812472178  0.0569947476089211\\
        0.0666408526555842  0.0315289048523587\\
        0.0499910849999656  0.0193607052468067\\
        0.0399982806717673  0.0103206462181962\\
        0.0333321068124764  0.00913293686593408\\
        0.0285696201019069  0.00601590699983667\\
        0.0249919930274932  0.00369952358918757\\
        0.0222120342891224  0.00326846432267536\\
        0.0199999963261506  0.00305580367952096\\
        0.0181812956477009  0.00273309183729409\\
        0.0166651641450463  0.00206270428665606\\
        0.015384328963273   0.00161940287097317\\
        0.0142856175244916  0.00127430912728888\\
        0.0133326387907105  0.00143750306761714\\
        0.0124986840340357  0.00105472113489713\\
        0.0117645043173365  0.001123286710766\\
        0.0111102401232889  0.0010524284259229\\
        0.0105261748821938  0.000823471967676181\\
        0.00999921217628962 0.000795224233188896\\
        0.00666657614059933 0.000319101862928828\\
        0.00499998131875709 0.000160324190540573\\
        0.00399999800442724 0.000109246783196904\\
        0.00333331174692903 9.66560620685453e-05\\
        0.0028571421843815  7.08625532428453e-05\\
        0.00249999897088586 5.48767951743523e-05\\
        0.00222222129568555 3.05792719982323e-05\\
        0.00199999219026464 2.75486135646874e-05\\
        0.00181818033591482 2.62738728752376e-05\\
        0.00166666587024734 2.47177375952701e-05\\
        0.00153845446252526 1.98371490313098e-05\\
        0.00142856662480272 1.39725704764171e-05\\
        0.00133333308002394 1.18151750145845e-05\\
        };
        \addlegendentry{$\varrho = 0.2500$, loose blocks} 

        \addplot [color=orange, line width=1pt] table[row sep=crcr]{%
        0.198983854587716   0.491468051658335\\
        0.0998705812472178  0.190685153093759\\
        0.0666408526555842  0.0883578692611009\\
        0.0499910849999656  0.0359907126275024\\
        0.0399982806717673  0.0459865509460435\\
        0.0333321068124764  0.0258923605984912\\
        0.0285696201019069  0.0197610291361237\\
        0.0249919930274932  0.0136655033029259\\
        0.0222120342891224  0.0147951178753708\\
        0.0199999963261506  0.00860406008846276\\
        0.0181812956477009  0.00546288175365123\\
        0.0166651641450463  0.00697166989541764\\
        0.015384328963273   0.00532241693967839\\
        0.0142856175244916  0.00427979113958232\\
        0.0133326387907105  0.00308675340416087\\
        0.0124986840340357  0.00242066075082326\\
        0.0117645043173365  0.00311924993786661\\
        0.0111102401232889  0.00282589844372216\\
        0.0105261748821938  0.00250632767309433\\
        0.00999921217628962 0.00198462908678676\\
        0.00666657614059933 0.00112900804257809\\
        0.00499998131875709 0.00047704521238412\\
        0.00399999800442724 0.000306813958185326\\
        0.00333331174692903 0.000216909721958092\\
        0.0028571421843815  0.000199122783114181\\
        0.00249999897088586 0.000113372302876558\\
        0.00222222129568555 9.59397423336839e-05\\
        0.00199999219026464 8.21497304043239e-05\\
        0.00181818033591482 6.89762566380425e-05\\
        0.00166666587024734 5.52458900082708e-05\\
        0.00153845446252526 5.46333639190877e-05\\
        0.00142856662480272 4.16883683667102e-05\\
        0.00133333308002394 3.7984548001857e-05\\
        };
        \addlegendentry{$\varrho = 0.0250$, two parts} 

        \addplot [color=black, forget plot]
          table[row sep=crcr]{%
        0.3     3\\
        3e-05   3e-04\\
        };
        \node[right] at (3e-04, 3e-03) {\tiny 1st order};
        \addplot [color=black, forget plot]
          table[row sep=crcr]{%
        0.3     3\\
        3e-05   3e-08\\
        };
        \node[above right] at (3e-04, 3e-06) {\tiny 2nd order};
        \addplot [color=black, forget plot]
          table[row sep=crcr]{%
        0.3     3\\
        3e-05   3e-12\\
        };
        \node[above right] at (1.7e-03, 3e-07) {\tiny 3rd order};
        \end{axis}

        \end{tikzpicture}
        }
        \phantomsubcaption
    \end{subfigure}
    \caption{2d tests on the unit disk domain with discontinuous coefficient matrices}
    \label{fig:2d_discontinuous_test_1}
\end{figure}
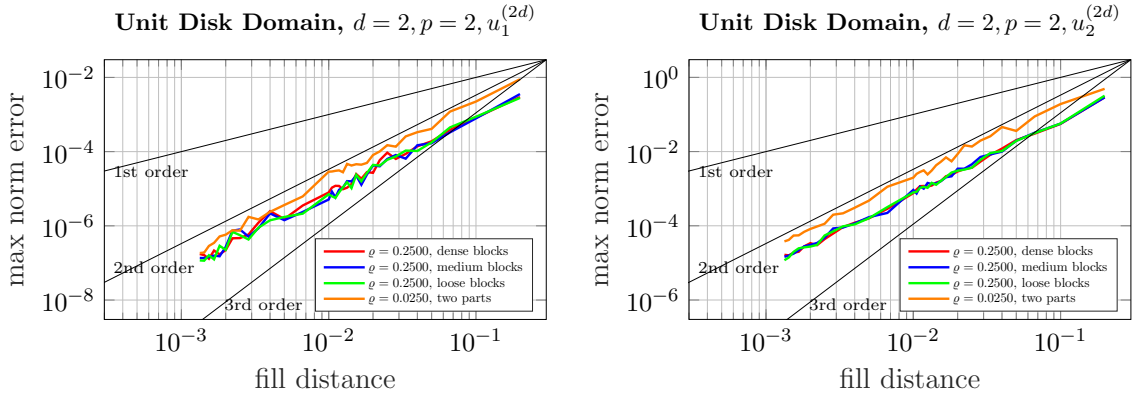

\begin{figure}[htp]
    \centering
    \begin{subfigure}[b]{0.490\textwidth}
        \centering
        \resizebox{\textwidth}{!}{
        \begin{tikzpicture}
        \begin{axis}[%
        width=2.55in,
        height=1.5in,
        scale only axis,
        xmode=log,
        xmin=3e-04,
        xmax=0.3,
        xminorticks=true,
        xlabel style={font=\color{white!15!black}},
        xlabel={\large fill distance},
        ymode=log,
        ymin=3e-09,
        ymax=0.03,
        yminorticks=true,
        ylabel style={font=\color{white!15!black}},
        ylabel={\large max norm error},
        title style={align=center, font=\bfseries},
        title={\small L-shaped Domain, $d = 2, p = 2, u_{1}^{(2d)}$},
        xmajorgrids,
        xminorgrids,
        ymajorgrids,
        yminorgrids,
        legend style={at={(0.97,0.03)}, nodes={scale=0.5, transform shape}, anchor=south east, legend cell align=left, align=left, draw=white!15!black}
        ]
        \addplot [color=red, line width=1pt] table[row sep=crcr]{%
        0.198983854587716   0.00235657952815038\\
        0.0998705812472178  0.000578776904978096\\
        0.0666408526555842  0.000329691473801907\\
        0.0499910849999656  0.000174627034041697\\
        0.0399982806717673  8.22312995519825e-05\\
        0.0333321068124764  6.60786570847449e-05\\
        0.0285696201019069  6.42189568957896e-05\\
        0.0249919930274933  7.96774713575088e-05\\
        0.0222120342891224  3.20664333195442e-05\\
        0.0199999963261506  3.55390109556541e-05\\
        0.0181812956477009  2.39865783304705e-05\\
        0.0166651641450463  1.95968802263469e-05\\
        0.015384328963273   1.25248567708347e-05\\
        0.0142856175244916  1.8483456534657e-05\\
        0.0133326387907105  6.99626386224672e-06\\
        0.0124986840340357  6.93651335736956e-06\\
        0.0117645043173365  8.69416518844979e-06\\
        0.0111102401232889  8.2313722948868e-06\\
        0.0105263066128438  1.49000398033117e-05\\
        0.0099995353399169  3.98351374175654e-06\\
        0.00666657614059933 3.27376572495375e-06\\
        0.00499998131875709 1.42106327527536e-06\\
        0.00399999800442724 1.83794457964126e-06\\
        0.00333331174692903 5.13173754379537e-07\\
        0.0028571421843815  6.65468580463013e-07\\
        0.00249999897088586 4.67799772696509e-07\\
        0.00222222129568555 2.46037979545477e-07\\
        0.00199999219026464 2.69966661692322e-07\\
        0.00181818033591482 1.46400799927449e-07\\
        0.00166666587024734 1.70122193221545e-07\\
        0.00153845446252526 1.52364037830921e-07\\
        0.00142856662480272 8.87859636922173e-08\\
        0.00133333308002394 9.29713714903357e-08\\
        };
        \addlegendentry{$\varrho = 0.2500$, dense blocks} 

        \addplot [color=blue, line width=1pt] table[row sep=crcr]{%
        0.198983854587716   0.00263363432968566\\
        0.0998705812472178  0.000613198544793025\\
        0.0666408526555842  0.000297602019610488\\
        0.0499910849999656  0.000142049176264369\\
        0.0399982806717673  9.02929952126819e-05\\
        0.0333321068124764  6.19588609072075e-05\\
        0.0285696201019069  5.11700123089209e-05\\
        0.0249919930274933  4.95715380466333e-05\\
        0.0222120342891224  3.48416661155593e-05\\
        0.0199999963261506  2.87589761971407e-05\\
        0.0181812956477009  1.57691085198586e-05\\
        0.0166651641450463  1.73845375555093e-05\\
        0.015384328963273   1.81351302952493e-05\\
        0.0142856175244916  1.77243342969202e-05\\
        0.0133326387907105  9.65556052157091e-06\\
        0.0124986840340357  1.01337433893178e-05\\
        0.0117645043173365  1.00785537042736e-05\\
        0.0111102401232889  5.92576884972384e-06\\
        0.0105263066128438  1.06679878022486e-05\\
        0.0099995353399169  5.1286208757606e-06\\
        0.00666657614059933 2.70078264319462e-06\\
        0.00499998131875709 1.27406424477172e-06\\
        0.00399999800442724 1.96299147425805e-06\\
        0.00333331174692903 4.95832329638901e-07\\
        0.0028571421843815  3.85906368860489e-07\\
        0.00249999897088586 4.61076458790899e-07\\
        0.00222222129568555 5.284062836175e-07\\
        0.00199999219026464 1.70970242074908e-07\\
        0.00181818033591482 2.62581191234901e-07\\
        0.00166666587024734 1.16751692980088e-07\\
        0.00153845446252526 1.14628105674086e-07\\
        0.00142856662480272 1.15945126610839e-07\\
        0.00133333308002394 1.16759414581225e-07\\
        };
        \addlegendentry{$\varrho = 0.2500$, medium blocks} 

        \addplot [color=green, line width=1pt] table[row sep=crcr]{%
        0.198983854587716   0.00235063430489979\\
        0.0998705812472178  0.000687454565427403\\
        0.0666408526555842  0.000371142176376349\\
        0.0499910849999656  0.000167762177397446\\
        0.0399982806717673  9.8662538425609e-05\\
        0.0333321068124764  0.000123904444548018\\
        0.0285696201019069  5.3979835869411e-05\\
        0.0249919930274933  6.43854445279146e-05\\
        0.0222120342891224  3.85588038209761e-05\\
        0.0199999963261506  3.33080254759643e-05\\
        0.0181812956477009  1.88627952808496e-05\\
        0.0166651641450463  2.10434251386449e-05\\
        0.015384328963273   2.03661588857873e-05\\
        0.0142856175244916  1.39980218452074e-05\\
        0.0133326387907105  1.14594206228613e-05\\
        0.0124986840340357  1.10849869739971e-05\\
        0.0117645043173365  7.8692184215523e-06\\
        0.0111102401232889  5.82642255142485e-06\\
        0.0105263066128438  7.42644909756951e-06\\
        0.0099995353399169  6.84085883628782e-06\\
        0.00666657614059933 2.29824471942308e-06\\
        0.00499998131875709 8.90510872153527e-07\\
        0.00399999800442724 1.34082143110614e-06\\
        0.00333331174692903 4.13971156465465e-07\\
        0.0028571421843815  3.65460680606944e-07\\
        0.00249999897088586 4.4167688173502e-07\\
        0.00222222129568555 5.83314005764279e-07\\
        0.00199999219026464 2.29352772440627e-07\\
        0.00181818033591482 2.4640154761002e-07\\
        0.00166666587024734 1.18343446486868e-07\\
        0.00153845446252526 1.70999767235003e-07\\
        0.00142856662480272 7.5598704629698e-08\\
        0.00133333308002394 1.01360873694389e-07\\
        };
        \addlegendentry{$\varrho = 0.2500$, loose blocks} 

        \addplot [color=orange, line width=1pt] table[row sep=crcr]{%
        0.198983854587716   0.00804144886644631\\
        0.0998705812472178  0.00109371186265894\\
        0.0666408526555842  0.000472568954079766\\
        0.0499910849999656  0.000235009881075443\\
        0.0399982806717673  0.000182418001132112\\
        0.0333321068124764  0.000130837327836697\\
        0.0285696201019069  9.07119981746196e-05\\
        0.0249919930274933  5.70691769243847e-05\\
        0.0222120342891224  4.78059792792251e-05\\
        0.0199999963261506  2.55858460249669e-05\\
        0.0181812956477009  2.6435386403012e-05\\
        0.0166651641450463  3.32360118933117e-05\\
        0.015384328963273   1.83247530267661e-05\\
        0.0142856175244916  1.46858014690165e-05\\
        0.0133326387907105  1.70334668795347e-05\\
        0.0124986840340357  1.48584654787731e-05\\
        0.0117645043173365  1.99337047925852e-05\\
        0.0111102401232889  1.2668525488313e-05\\
        0.0105263066128438  1.40940658321931e-05\\
        0.0099995353399169  1.02388287542277e-05\\
        0.00666657614059933 2.77864924214111e-06\\
        0.00499998131875709 1.79809513722873e-06\\
        0.00399999800442724 2.06562100202312e-06\\
        0.00333331174692903 1.21823609866567e-06\\
        0.0028571421843815  7.90222143387531e-07\\
        0.00249999897088586 3.1747558226769e-07\\
        0.00222222129568555 4.82403330281933e-07\\
        0.00199999219026464 2.79493483579252e-07\\
        0.00181818033591482 3.70766207180573e-07\\
        0.00166666587024734 2.20314910981401e-07\\
        0.00153845446252526 1.81007540289713e-07\\
        0.00142856662480272 1.4829506755909e-07\\
        0.00133333308002394 1.65967592558047e-07\\
        };
        \addlegendentry{$\varrho = 0.0250$, two parts} 

        \addplot [color=black, forget plot]
          table[row sep=crcr]{%
        0.3   0.03\\
        3e-05  3e-06\\
        };
        \node[right] at (3e-04, 3e-05) {\tiny 1st order};
        \addplot [color=black, forget plot]
          table[row sep=crcr]{%
        0.3   0.03\\
        3e-05  3e-10\\
        };
        \node[above right] at (3e-04, 3e-08) {\tiny 2nd order};
        \addplot [color=black, forget plot]
          table[row sep=crcr]{%
        0.3   0.03\\
        3e-05  3e-14\\
        };
        \node[above right] at (1.7e-03, 3e-09) {\tiny 3rd order};
        \end{axis}

        \end{tikzpicture}
        }
        \phantomsubcaption
    \end{subfigure}
    \hfill
    \begin{subfigure}[b]{0.490\textwidth}
        \centering
        \resizebox{\textwidth}{!}{
        \begin{tikzpicture}
        \begin{axis}[%
        width=2.55in,
        height=1.5in,
        scale only axis,
        xmode=log,
        xmin=3e-04,
        xmax=0.3,
        xminorticks=true,
        xlabel style={font=\color{white!15!black}},
        xlabel={\large fill distance},
        ymode=log,
        ymin=3e-07,
        ymax=3,
        yminorticks=true,
        ylabel style={font=\color{white!15!black}},
        ylabel={\large max norm error},
        title style={align=center, font=\bfseries},
        title={\small L-shaped Domain, $d = 2, p = 2, u_{2}^{(2d)}$},
        xmajorgrids,
        xminorgrids,
        ymajorgrids,
        yminorgrids,
        legend style={at={(0.97,0.03)}, nodes={scale=0.5, transform shape}, anchor=south east, legend cell align=left, align=left, draw=white!15!black}
        ]
        \addplot [color=red, line width=1pt] table[row sep=crcr]{%
        0.198983854587716   0.649436010556696\\
        0.0998705812472178  0.0682869052488434\\
        0.0666408526555842  0.0473663797496817\\
        0.0499910849999656  0.0319806783240204\\
        0.0399982806717673  0.0209940431558433\\
        0.0333321068124764  0.0156597376139835\\
        0.0285696201019069  0.00846570887734899\\
        0.0249919930274933  0.00667364426817407\\
        0.0222120342891224  0.00433172416923178\\
        0.0199999963261506  0.0035735588042165\\
        0.0181812956477009  0.00319530812413049\\
        0.0166651641450463  0.00455037895060606\\
        0.015384328963273   0.00189464784853666\\
        0.0142856175244916  0.00244538206496614\\
        0.0133326387907105  0.0016630377082536\\
        0.0124986840340357  0.0013769699416688\\
        0.0117645043173365  0.00158811551444504\\
        0.0111102401232889  0.00130023650643807\\
        0.0105263066128438  0.000948371752416577\\
        0.0099995353399169  0.00104528230539014\\
        0.00666657614059933 0.000489446347994082\\
        0.00499998131875709 0.00019515287945282\\
        0.00399999800442724 0.000130551152868641\\
        0.00333331174692903 0.00013927261106339\\
        0.0028571421843815  8.97863994966386e-05\\
        0.00249999897088586 5.43830684254232e-05\\
        0.00222222129568555 3.29838898762702e-05\\
        0.00199999219026464 4.74591825474491e-05\\
        0.00181818033591482 2.97219713232266e-05\\
        0.00166666587024734 2.2614039048463e-05\\
        0.00153845446252526 2.03532863558564e-05\\
        0.00142856662480272 1.62330839073022e-05\\
        0.00133333308002394 1.68530906492492e-05\\
        };
        \addlegendentry{$\varrho = 0.2500$, dense blocks} 

        \addplot [color=blue, line width=1pt] table[row sep=crcr]{%
        0.198983854587716   0.520476146891479\\
        0.0998705812472178  0.0643562778795629\\
        0.0666408526555842  0.0472691709869002\\
        0.0499910849999656  0.0304765899387307\\
        0.0399982806717673  0.0197954143313162\\
        0.0333321068124764  0.0158658915993826\\
        0.0285696201019069  0.00766417193028146\\
        0.0249919930274933  0.00433496798810928\\
        0.0222120342891224  0.00429194918246578\\
        0.0199999963261506  0.00448354257016526\\
        0.0181812956477009  0.00273029828935534\\
        0.0166651641450463  0.00410297438475737\\
        0.015384328963273   0.00218295407512326\\
        0.0142856175244916  0.00178510276843724\\
        0.0133326387907105  0.00165292220083813\\
        0.0124986840340357  0.00160369091649315\\
        0.0117645043173365  0.00127905375966897\\
        0.0111102401232889  0.00122226496358924\\
        0.0105263066128438  0.000895361992405697\\
        0.0099995353399169  0.00102324006767796\\
        0.00666657614059933 0.000549635791609049\\
        0.00499998131875709 0.000221346617940377\\
        0.00399999800442724 0.000141178834348832\\
        0.00333331174692903 0.000132239452313776\\
        0.0028571421843815  6.72997121640151e-05\\
        0.00249999897088586 6.83032218251611e-05\\
        0.00222222129568555 4.35282296287198e-05\\
        0.00199999219026464 5.05639092240884e-05\\
        0.00181818033591482 3.63781812504627e-05\\
        0.00166666587024734 3.09724153986224e-05\\
        0.00153845446252526 2.30210340852466e-05\\
        0.00142856662480272 2.17731850975511e-05\\
        0.00133333308002394 2.01943304869445e-05\\
        };
        \addlegendentry{$\varrho = 0.2500$, medium blocks} 

        \addplot [color=green, line width=1pt] table[row sep=crcr]{%
        0.198983854587716   0.411330973599294\\
        0.0998705812472178  0.0610591244748697\\
        0.0666408526555842  0.0514939167082771\\
        0.0499910849999656  0.0301694028471722\\
        0.0399982806717673  0.0202876252959108\\
        0.0333321068124764  0.0181889789887979\\
        0.0285696201019069  0.00721422222733636\\
        0.0249919930274933  0.00626631666432464\\
        0.0222120342891224  0.00488098970047091\\
        0.0199999963261506  0.00509434805492859\\
        0.0181812956477009  0.00389049643502659\\
        0.0166651641450463  0.00403194068534596\\
        0.015384328963273   0.00265702590593886\\
        0.0142856175244916  0.0023575584807487\\
        0.0133326387907105  0.00155373048301932\\
        0.0124986840340357  0.00124274063292873\\
        0.0117645043173365  0.00160047678368347\\
        0.0111102401232889  0.00163860791965043\\
        0.0105263066128438  0.000905037588947444\\
        0.0099995353399169  0.00121321809194086\\
        0.00666657614059933 0.000541153183694298\\
        0.00499998131875709 0.00025712138847922\\
        0.00399999800442724 0.000173230596706819\\
        0.00333331174692903 0.000166058370733246\\
        0.0028571421843815  9.39688649337356e-05\\
        0.00249999897088586 7.52663756138361e-05\\
        0.00222222129568555 4.63114784459151e-05\\
        0.00199999219026464 4.94303927656148e-05\\
        0.00181818033591482 5.08860376076115e-05\\
        0.00166666587024734 2.9823422906361e-05\\
        0.00153845446252526 1.98013282709519e-05\\
        0.00142856662480272 2.22224725456854e-05\\
        0.00133333308002394 2.08344444549979e-05\\
        };
        \addlegendentry{$\varrho = 0.2500$, loose blocks} 

        \addplot [color=orange, line width=1pt] table[row sep=crcr]{%
        0.198983854587716   0.669146746816984\\
        0.0998705812472178  0.139638097185541\\
        0.0666408526555842  0.0891829482888782\\
        0.0499910849999656  0.0433307665093789\\
        0.0399982806717673  0.0226400384778547\\
        0.0333321068124764  0.0224289570217051\\
        0.0285696201019069  0.0135517563885069\\
        0.0249919930274933  0.00971017114854156\\
        0.0222120342891224  0.00712947200705893\\
        0.0199999963261506  0.00736696685908456\\
        0.0181812956477009  0.00615510592919932\\
        0.0166651641450463  0.00407846007365953\\
        0.015384328963273   0.00313566690303579\\
        0.0142856175244916  0.00367966506942885\\
        0.0133326387907105  0.00386627919219507\\
        0.0124986840340357  0.00200736509844113\\
        0.0117645043173365  0.00265188127398108\\
        0.0111102401232889  0.00189479535763448\\
        0.0105263066128438  0.00177081789693823\\
        0.0099995353399169  0.00181335060776178\\
        0.00666657614059933 0.000716790826299274\\
        0.00499998131875709 0.00056560242747139\\
        0.00399999800442724 0.000247762136137197\\
        0.00333331174692903 0.000201178147642977\\
        0.0028571421843815  0.000130985941908079\\
        0.00249999897088586 8.04863442850356e-05\\
        0.00222222129568555 6.73573502316671e-05\\
        0.00199999219026464 8.23829619225336e-05\\
        0.00181818033591482 4.25429555317436e-05\\
        0.00166666587024734 4.12899994763904e-05\\
        0.00153845446252526 3.32016289696924e-05\\
        0.00142856662480272 2.88241191865879e-05\\
        0.00133333308002394 2.78063156917696e-05\\
        };
        \addlegendentry{$\varrho = 0.0250$, two parts} 

        \addplot [color=black, forget plot]
          table[row sep=crcr]{%
        0.3     3\\
        3e-05   3e-04\\
        };
        \node[right] at (3e-04, 3e-03) {\tiny 1st order};
        \addplot [color=black, forget plot]
          table[row sep=crcr]{%
        0.3     3\\
        3e-05   3e-08\\
        };
        \node[above right] at (3e-04, 3e-06) {\tiny 2nd order};
        \addplot [color=black, forget plot]
          table[row sep=crcr]{%
        0.3     3\\
        3e-05   3e-12\\
        };
        \node[above right] at (1.7e-03, 3e-07) {\tiny 3rd order};
        \end{axis}

        \end{tikzpicture}
        }
        \phantomsubcaption
    \end{subfigure}
    \caption{2d tests on the L-shaped domain with discontinuous coefficient matrices}
    \label{fig:2d_discontinuous_test_2}
\end{figure}
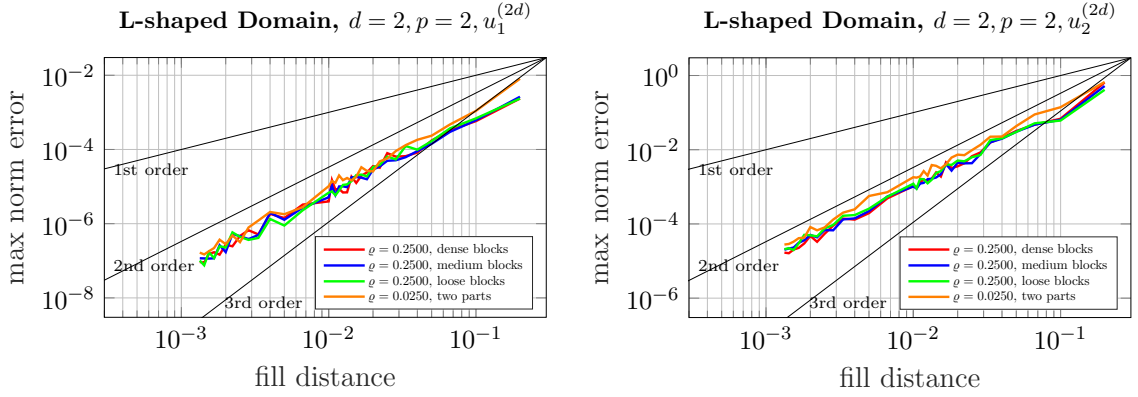

\subsection{3d numerical tests}
\label{sec:num3d}
3d numerical tests are performed over the unit sphere given by $\{ x_{1}^{2} + x_{2}^{2} + x_{3}^{2} < 1 \}$ and a 3d L-shaped domain given by $(-1, 1)^{3} \backslash [0, 1] \times [-1, 1] \times [0, 1]$.
The nonlocal kernel function is chosen to be the same one as in the 2d case. 
We again test our algorithm for smooth manufactured solutions and $p=2$. 

\subsubsection{Tests for continuous coefficient matrices}
For the test on continuous coefficient matrices, we use the following list of coefficient matrices.
Notice again that our baseline case is $A_{0}(\boldsymbol{x}) = I$ with $\varrho = 1$.
\begin{equation*}
\resizebox{\textwidth}{!}{$\displaystyle
\begin{matrix}
\# & A(\boldsymbol{x}) & \varrho\\
1 &
\begin{pmatrix}
1 - 0.5 |x_{1}| & 0 & 0\\
0 & 0.5 - 0.25 |x_{2}| & 0\\
0 & 0 & 0.25 + 0.25 |x_{3}|
\end{pmatrix}
& 0.2500\\[1.5em]
2 & \displaystyle \frac{1}{2.21}
\begin{pmatrix}
2 - |x_{1}| & 0 & 0.5\\
0 & 0.5 + 0.5 |x_{2}| & 0\\
0.5 & 0 & 1 - 0.5 |x_{3}|
\end{pmatrix}
& 0.0864\\[1.5em]
3 &
\begin{pmatrix}
1 - 0.5 |x_{1}| & 0 & 0\\
0 & 0.05 - 0.025 |x_{2}| & 0\\
0 & 0 & 0.025 + 0.025 |x_{3}|
\end{pmatrix}
& 0.0250\\[1.5em]
4 &
\begin{pmatrix}
1 - 0.5 |x_{1}| & 0 & 0\\
0 & 0.005 - 0.0025 |x_{2}| & 0\\
0 & 0 & 0.0025 + 0.0025 |x_{3}|
\end{pmatrix}
& 0.0025\\[1.5em]
5 & \displaystyle \frac{1}{2.001}
\begin{pmatrix}
2 - |x_{1} (0.5 - x_{2})| & -0.02 & 0.005\\
-0.02 & 0.005 + 0.005 |x_{1} + x_{3}| & -0.001\\
0.005 & -0.001 & 0.01 + 0.0025 x_{2} \exp(x_{3})
\end{pmatrix}
& 0.0014
\end{matrix}
$}
\end{equation*}
Here the value $\varrho$ is computed approximately in the domain $[-1, 1]^{3}$. 
Numerical results are presented in \cref{fig:3d_continuous_test_1,fig:3d_continuous_test_2} for the two manufactured solutions $u_{1}^{(3d)}(x_{1}, x_{2}, x_{3}) = x_{1} x_{2} + x_{1} x_{3} + x_{2} x_{3} + \cos(x_{1}) \exp(x_{2} + x_{3})$ and 
$u_{2}^{(3d)}(x_{1}, x_{2}, x_{3}) = (x_{1} + x_{2} + x_{3})^{4} \cos(x_{1} (x_{1} + 2 x_{2} + 2 x_{3}))$. 
We observe similar second-order convergence in $h$ for all cases and numerical errors grow as $\varrho$ becomes smaller. 

\begin{figure}[htp]
    \centering
    \begin{subfigure}[b]{0.490\textwidth}
        \centering
        \resizebox{\textwidth}{!}{
        \begin{tikzpicture}
        \begin{axis}[%
        width=2.55in,
        height=1.5in,
        scale only axis,
        xmode=log,
        xmin=6e-03,
        xmax=0.3,
        xminorticks=true,
        xlabel style={font=\color{white!15!black}},
        xlabel={\large fill distance},
        ymode=log,
        ymin=3e-06,
        ymax=0.07,
        yminorticks=true,
        ylabel style={font=\color{white!15!black}},
        ylabel={\large max norm error},
        title style={align=center, font=\bfseries},
        title={\small Unit Sphere Domain, $d = 3, p = 2, u_{1}^{(3d)}$},
        xmajorgrids,
        xminorgrids,
        ymajorgrids,
        yminorgrids,
        legend style={at={(0.97,0.03)}, nodes={scale=0.5, transform shape}, anchor=south east, legend cell align=left, align=left, draw=white!15!black}
        ]
        \addplot [color=red, line width=1pt] table[row sep=crcr]{%
        0.199976856129712   0.00416455386007053\\
        0.0999950754752416  0.000937482946753665\\
        0.0666664718175454  0.000294962079829375\\
        0.0499997004247293  0.000183030316736055\\
        0.0399997888105867  9.286768553185e-05\\
        0.0333331617388861  5.40008043206086e-05\\
        0.0285714246047077  4.39360273452039e-05\\
        0.0249999978650042  2.54960928900427e-05\\
        0.0222222025230041  2.09073585795139e-05\\
        0.0199999972740879  2.08611227630051e-05\\
        0.0181818117125527  1.29118635552317e-05\\
        0.0166666656465241  1.02738272089731e-05\\
        0.0153846118094882  9.1360106102556e-06\\
        0.0142857138581073  8.8820132950751e-06\\
        };
        \addlegendentry{$\varrho = 1.0000$} 

        \addplot [color=blue, line width=1pt] table[row sep=crcr]{%
        0.199976856129712   0.00383598148522646\\
        0.0999950754752416  0.00131740876782382\\
        0.0666664718175454  0.000442510327001155\\
        0.0499997004247293  0.000178968790396361\\
        0.0399997888105867  9.60660641875499e-05\\
        0.0333331617388861  6.85337279637821e-05\\
        0.0285714246047077  4.92385413046748e-05\\
        0.0249999978650042  3.72705512421412e-05\\
        0.0222222025230041  2.27032198027999e-05\\
        0.0199999972740879  2.3794612243222e-05\\
        0.0181818117125527  1.40449256043063e-05\\
        0.0166666656465241  1.28687937936967e-05\\
        0.0153846118094882  9.38559551721951e-06\\
        0.0142857138581073  7.16551617729877e-06\\
        };
        \addlegendentry{$\varrho = 0.2500$} 

        \addplot [color=green, line width=1pt] table[row sep=crcr]{%
        0.199976856129712   0.00465615162834898\\
        0.0999950754752416  0.0013871201886948\\
        0.0666664718175454  0.000343250497588432\\
        0.0499997004247293  0.000211856363594443\\
        0.0399997888105867  0.000103164042477566\\
        0.0333331617388861  8.54438221136178e-05\\
        0.0285714246047077  4.75723783570636e-05\\
        0.0249999978650042  3.4729777696807e-05\\
        0.0222222025230041  3.02212134624114e-05\\
        0.0199999972740879  2.12459482433758e-05\\
        0.0181818117125527  1.70374290453523e-05\\
        0.0166666656465241  1.36625229743004e-05\\
        0.0153846118094882  1.1914839481264e-05\\
        0.0142857138581073  9.68400126577507e-06\\
        };
        \addlegendentry{$\varrho = 0.0864$} 

        \addplot [color=lime, line width=1pt] table[row sep=crcr]{%
        0.199976856129712   0.0188406620946866\\
        0.0999950754752416  0.00395334923667923\\
        0.0666664718175454  0.00188034390631309\\
        0.0499997004247293  0.000873635954097729\\
        0.0399997888105867  0.000601395017238104\\
        0.0333331617388861  0.00037472881574474\\
        0.0285714246047077  0.000304856317777435\\
        0.0249999978650042  0.000227264034718289\\
        0.0222222025230041  0.000163352977101461\\
        0.0199999972740879  0.000129477711896531\\
        0.0181818117125527  0.000113894967293415\\
        0.0166666656465241  8.51622583892819e-05\\
        0.0153846118094882  7.22640924819551e-05\\
        0.0142857138581073  6.32267617919346e-05\\
        };
        \addlegendentry{$\varrho = 0.0250$} 

        \addplot [color=orange, line width=1pt] table[row sep=crcr]{%
        0.199976856129712   0.0548740389734579\\
        0.0999950754752416  0.0191488780238172\\
        0.0666664718175454  0.0101067087695206\\
        0.0499997004247293  0.00552641680629362\\
        0.0399997888105867  0.0035522572556701\\
        0.0333331617388861  0.00235456089235386\\
        0.0285714246047077  0.00194174983345707\\
        0.0249999978650042  0.00139932979676782\\
        0.0222222025230041  0.000964564064786178\\
        0.0199999972740879  0.000835348077606657\\
        0.0181818117125527  0.00075079893877783\\
        0.0166666656465241  0.000618925405776682\\
        0.0153846118094882  0.000516723369992977\\
        0.0142857138581073  0.000407212299864845\\
        };
        \addlegendentry{$\varrho = 0.0025$} 

        \addplot [color=cyan, line width=1pt, opacity=0.6] table[row sep=crcr]{%
        0.199976856129712   0.0555242075224274\\
        0.0999950754752416  0.0194900760839034\\
        0.0666664718175454  0.0104674630318757\\
        0.0499997004247293  0.00570580904670503\\
        0.0399997888105867  0.00355255239797314\\
        0.0333331617388861  0.00231972657092294\\
        0.0285714246047077  0.00168375661411968\\
        0.0249999978650042  0.00130258587234833\\
        0.0222222025230041  0.000869979071017379\\
        0.0199999972740879  0.000807257461822086\\
        0.0181818117125527  0.000692308339573344\\
        0.0166666656465241  0.000556028674934872\\
        0.0153846118094882  0.00048267735816987\\
        0.0142857138581073  0.000411243563122721\\
        };
        \addlegendentry{$\varrho = 0.0014$} 

        \addplot [color=black, forget plot]
          table[row sep=crcr]{%
        0.3   0.07\\
        3e-05  7e-06\\
        };
        \node[right] at (6e-03, 1.35e-03) {\tiny 1st order};
        \addplot [color=black, forget plot]
          table[row sep=crcr]{%
        0.3   0.07\\
        3e-05  7e-10\\
        };
        \node[above right] at (6e-03, 2.6e-05) {\tiny 2nd order};
        \addplot [color=black, forget plot]
          table[row sep=crcr]{%
        0.3   0.07\\
        3e-05  7e-14\\
        };
        \node[above right] at (1.23e-02, 3e-06) {\tiny 3rd order};
        \end{axis}

        \end{tikzpicture}
        }
        \phantomsubcaption
    \end{subfigure}
    \hfill
    \begin{subfigure}[b]{0.490\textwidth}
        \centering
        \resizebox{\textwidth}{!}{
        \begin{tikzpicture}
        \begin{axis}[%
        width=2.55in,
        height=1.5in,
        scale only axis,
        xmode=log,
        xmin=6e-03,
        xmax=0.3,
        xminorticks=true,
        xlabel style={font=\color{white!15!black}},
        xlabel={\large fill distance},
        ymode=log,
        ymin=7e-04,
        ymax=7,
        yminorticks=true,
        ylabel style={font=\color{white!15!black}},
        ylabel={\large max norm error},
        title style={align=center, font=\bfseries},
        title={\small Unit Sphere Domain, $d = 3, p = 2, u_{2}^{(3d)}$},
        xmajorgrids,
        xminorgrids,
        ymajorgrids,
        yminorgrids,
        legend style={at={(0.97,0.03)}, nodes={scale=0.5, transform shape}, anchor=south east, legend cell align=left, align=left, draw=white!15!black}
        ]
        \addplot [color=red, line width=1pt] table[row sep=crcr]{%
        0.199976856129712   0.22827743300953\\
        0.0999950754752416  0.101706943091798\\
        0.0666664718175454  0.0402333522512002\\
        0.0499997004247293  0.0208829874811005\\
        0.0399997888105867  0.0107359933474023\\
        0.0333331617388861  0.00762793523854977\\
        0.0285714246047077  0.00639744086838157\\
        0.0249999978650042  0.00475664810428733\\
        0.0222222025230041  0.00354889254098056\\
        0.0199999972740879  0.00307404944219125\\
        0.0181818117125527  0.00244664888800106\\
        0.0166666656465241  0.00192072546513877\\
        0.0153846118094882  0.00170075548583659\\
        0.0142857138581073  0.00144651361470638\\
        };
        \addlegendentry{$\varrho = 1.0000$} 

        \addplot [color=blue, line width=1pt] table[row sep=crcr]{%
        0.199976856129712   0.383407405708036\\
        0.0999950754752416  0.108731197243352\\
        0.0666664718175454  0.0476702645861631\\
        0.0499997004247293  0.0255233535934227\\
        0.0399997888105867  0.016206936065065\\
        0.0333331617388861  0.0115905628214916\\
        0.0285714246047077  0.00944449902411093\\
        0.0249999978650042  0.00690288184223697\\
        0.0222222025230041  0.0046926924111208\\
        0.0199999972740879  0.00398394369296384\\
        0.0181818117125527  0.00326093154355345\\
        0.0166666656465241  0.00275182131219776\\
        0.0153846118094882  0.00219742551483826\\
        0.0142857138581073  0.00194161456011011\\
        };
        \addlegendentry{$\varrho = 0.2500$} 

        \addplot [color=green, line width=1pt] table[row sep=crcr]{%
        0.199976856129712   0.637720024533205\\
        0.0999950754752416  0.175464518100957\\
        0.0666664718175454  0.0681874283354107\\
        0.0499997004247293  0.0414899467489909\\
        0.0399997888105867  0.0223528151846852\\
        0.0333331617388861  0.017671743795135\\
        0.0285714246047077  0.0144072129667965\\
        0.0249999978650042  0.011060139640251\\
        0.0222222025230041  0.00825091715666604\\
        0.0199999972740879  0.00580595309867937\\
        0.0181818117125527  0.00541657614865976\\
        0.0166666656465241  0.00447337667392089\\
        0.0153846118094882  0.00374781777728472\\
        0.0142857138581073  0.00314296279807458\\
        };
        \addlegendentry{$\varrho = 0.0864$} 

        \addplot [color=lime, line width=1pt] table[row sep=crcr]{%
        0.199976856129712   2.25585089459398\\
        0.0999950754752416  0.738161763029014\\
        0.0666664718175454  0.376322028227243\\
        0.0499997004247293  0.171941817895194\\
        0.0399997888105867  0.0933698239410981\\
        0.0333331617388861  0.0724642114337475\\
        0.0285714246047077  0.0490933375030105\\
        0.0249999978650042  0.0360496181383123\\
        0.0222222025230041  0.0251495572621505\\
        0.0199999972740879  0.0198703321305129\\
        0.0181818117125527  0.0177922999497997\\
        0.0166666656465241  0.0153723197151945\\
        0.0153846118094882  0.0125432946304374\\
        0.0142857138581073  0.0102328840981438\\
        };
        \addlegendentry{$\varrho = 0.0250$} 

        \addplot [color=orange, line width=1pt] table[row sep=crcr]{%
        0.199976856129712   5.98559548435468\\
        0.0999950754752416  2.90403282710471\\
        0.0666664718175454  1.84849038534435\\
        0.0499997004247293  1.13210491370226\\
        0.0399997888105867  0.69959873793514\\
        0.0333331617388861  0.547362956187548\\
        0.0285714246047077  0.349203918458873\\
        0.0249999978650042  0.25384612664621\\
        0.0222222025230041  0.200797906938894\\
        0.0199999972740879  0.17614930321054\\
        0.0181818117125527  0.122472103807605\\
        0.0166666656465241  0.100794337639694\\
        0.0153846118094882  0.0940670445598086\\
        0.0142857138581073  0.0772931984383218\\
        };
        \addlegendentry{$\varrho = 0.0025$} 

        \addplot [color=cyan, line width=1pt, opacity=0.6] table[row sep=crcr]{%
        0.199976856129712   5.24290377170214\\
        0.0999950754752416  2.79270408266146\\
        0.0666664718175454  1.54059440813433\\
        0.0499997004247293  1.01158429724165\\
        0.0399997888105867  0.60669488489846\\
        0.0333331617388861  0.453092334005545\\
        0.0285714246047077  0.269212494990374\\
        0.0249999978650042  0.252367117903295\\
        0.0222222025230041  0.163348878885572\\
        0.0199999972740879  0.138752689072966\\
        0.0181818117125527  0.105478937724394\\
        0.0166666656465241  0.0870057777415139\\
        0.0153846118094882  0.0703936608271492\\
        0.0142857138581073  0.06187395034208\\
        };
        \addlegendentry{$\varrho = 0.0014$} 

        \addplot [color=black, forget plot]
          table[row sep=crcr]{%
        0.3   7\\
        3e-05  7e-04\\
        };
        \node[right] at (6e-03, 1.35e-01) {\tiny 1st order};
        \addplot [color=black, forget plot]
          table[row sep=crcr]{%
        0.3   7\\
        3e-05  7e-08\\
        };
        \node[above right] at (6e-03, 2.6e-03) {\tiny 2nd order};
        \addplot [color=black, forget plot]
          table[row sep=crcr]{%
        0.3   7\\
        3e-05  7e-12\\
        };
        \node[above right] at (1.6e-02, 7e-04) {\tiny 3rd order};
        \end{axis}

        \end{tikzpicture}
        }
        \phantomsubcaption
    \end{subfigure}
    \caption{3d tests on the unit sphere domain with continuous coefficient matrices}
    \label{fig:3d_continuous_test_1}
\end{figure}
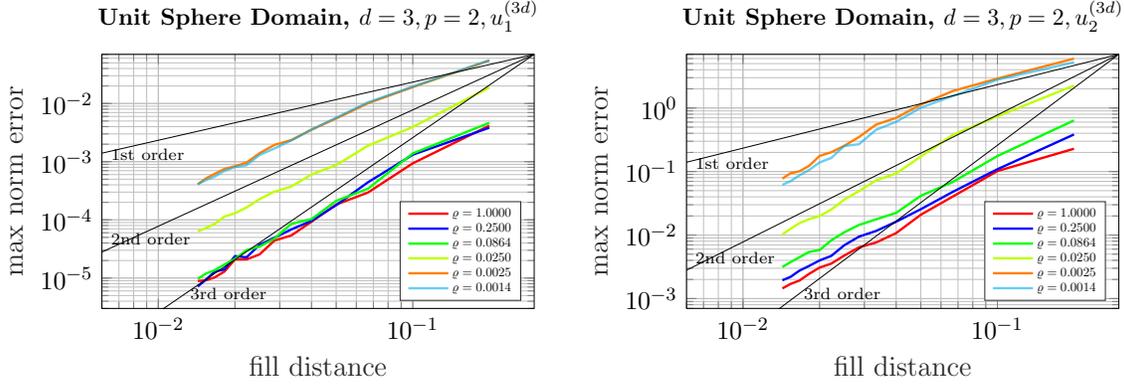

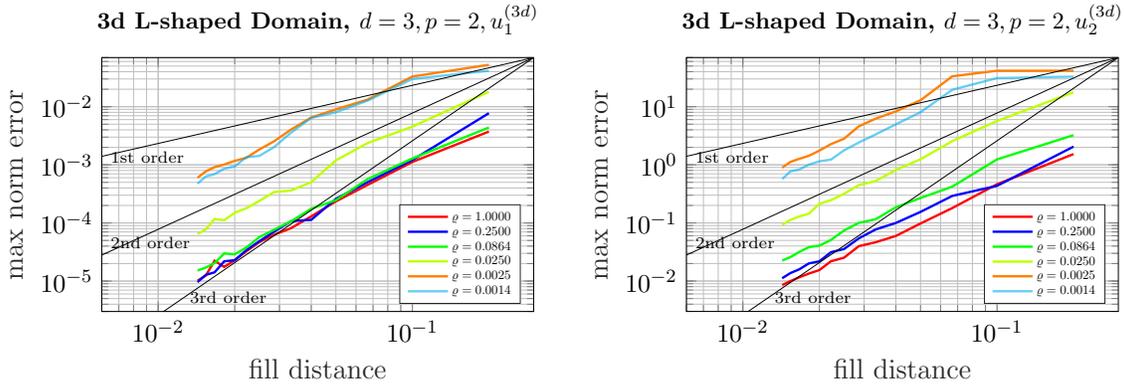
\begin{figure}[htp]
    \centering
    \begin{subfigure}[b]{0.490\textwidth}
        \centering
        \resizebox{\textwidth}{!}{
        \begin{tikzpicture}
        \begin{axis}[%
        width=2.55in,
        height=1.5in,
        scale only axis,
        xmode=log,
        xmin=6e-03,
        xmax=0.3,
        xminorticks=true,
        xlabel style={font=\color{white!15!black}},
        xlabel={\large fill distance},
        ymode=log,
        ymin=3e-06,
        ymax=0.07,
        yminorticks=true,
        ylabel style={font=\color{white!15!black}},
        ylabel={\large max norm error},
        title style={align=center, font=\bfseries},
        title={\small 3d L-shaped Domain, $d = 3, p = 2, u_{1}^{(3d)}$},
        xmajorgrids,
        xminorgrids,
        ymajorgrids,
        yminorgrids,
        legend style={at={(0.97,0.03)}, nodes={scale=0.5, transform shape}, anchor=south east, legend cell align=left, align=left, draw=white!15!black}
        ]
        \addplot [color=red, line width=1pt] table[row sep=crcr]{%
        0.199976856129712   0.00373902353121913\\
        0.0999950754752418  0.00112114313827805\\
        0.0666664718175454  0.000451881860073833\\
        0.0499997004247293  0.000235036092597785\\
        0.0399997888105867  0.000129812492884351\\
        0.0333331617388861  8.14289335187013e-05\\
        0.0285714246047077  6.26118421420685e-05\\
        0.0249999978650042  4.68789086340848e-05\\
        0.0222222025230041  3.25944955852364e-05\\
        0.0199999972740879  2.27849855147788e-05\\
        0.0181818117125527  1.75466146412617e-05\\
        0.0166666656465241  2.23956704337169e-05\\
        0.0153846118094882  1.24101457701364e-05\\
        0.0142857138581073  1.01794157272295e-05\\
        };
        \addlegendentry{$\varrho = 1.0000$} 

        \addplot [color=blue, line width=1pt] table[row sep=crcr]{%
        0.199976856129712   0.007742900847167\\
        0.0999950754752418  0.00121111755248826\\
        0.0666664718175454  0.000502969342621284\\
        0.0499997004247293  0.000259640733479394\\
        0.0399997888105867  0.000113131253157484\\
        0.0333331617388861  0.000107498847522258\\
        0.0285714246047077  7.07245208610985e-05\\
        0.0249999978650042  4.85707983539641e-05\\
        0.0222222025230041  3.40414005179923e-05\\
        0.0199999972740879  2.28740681373907e-05\\
        0.0181818117125527  2.18748905282951e-05\\
        0.0166666656465241  1.41536750923876e-05\\
        0.0153846118094882  1.27698953171773e-05\\
        0.0142857138581073  9.60323676224561e-06\\
        };
        \addlegendentry{$\varrho = 0.2500$} 

        \addplot [color=green, line width=1pt] table[row sep=crcr]{%
        0.199976856129712   0.00439180625325619\\
        0.0999950754752418  0.00128744609728981\\
        0.0666664718175454  0.000568691823754186\\
        0.0499997004247293  0.000246446396875211\\
        0.0399997888105867  0.000175268889572422\\
        0.0333331617388861  0.000106950010418938\\
        0.0285714246047077  7.64126414924959e-05\\
        0.0249999978650042  5.72257046460578e-05\\
        0.0222222025230041  3.78777092633342e-05\\
        0.0199999972740879  2.88410017361684e-05\\
        0.0181818117125527  2.99175907434446e-05\\
        0.0166666656465241  2.0411076833593e-05\\
        0.0153846118094882  1.70594401049939e-05\\
        0.0142857138581073  1.51907954260011e-05\\
        };
        \addlegendentry{$\varrho = 0.0864$} 

        \addplot [color=lime, line width=1pt] table[row sep=crcr]{%
        0.199976856129712   0.0180778867651248\\
        0.0999950754752418  0.00455446895814293\\
        0.0666664718175454  0.00237327632855333\\
        0.0499997004247293  0.00119250486441214\\
        0.0399997888105867  0.000502091340500499\\
        0.0333331617388861  0.000361858739250431\\
        0.0285714246047077  0.000340266823122271\\
        0.0249999978650042  0.000239536444433419\\
        0.0222222025230041  0.000179551478761475\\
        0.0199999972740879  0.000149802043209135\\
        0.0181818117125527  0.000111453370744696\\
        0.0166666656465241  0.00011700675385562\\
        0.0153846118094882  7.7386375044064e-05\\
        0.0142857138581073  6.4798537594335e-05\\
        };
        \addlegendentry{$\varrho = 0.0250$} 

        \addplot [color=orange, line width=1pt] table[row sep=crcr]{%
        0.199976856129712   0.052649940028336\\
        0.0999950754752418  0.0330809157276368\\
        0.0666664718175454  0.0132862563960927\\
        0.0499997004247293  0.00899182240904506\\
        0.0399997888105867  0.00652464545312714\\
        0.0333331617388861  0.00405318711223579\\
        0.0285714246047077  0.00254091626170716\\
        0.0249999978650042  0.00183908188973803\\
        0.0222222025230041  0.0013209330721955\\
        0.0199999972740879  0.0011699987980327\\
        0.0181818117125527  0.00101302552970495\\
        0.0166666656465241  0.000911481230040145\\
        0.0153846118094882  0.000771449991746831\\
        0.0142857138581073  0.000600031449728977\\
        };
        \addlegendentry{$\varrho = 0.0025$} 

        \addplot [color=cyan, line width=1pt, opacity=0.6] table[row sep=crcr]{%
        0.199976856129712   0.0415529238231027\\
        0.0999950754752418  0.0300560441091147\\
        0.0666664718175454  0.0127212298373816\\
        0.0499997004247293  0.00804172205952725\\
        0.0399997888105867  0.00642207664682548\\
        0.0333331617388861  0.00357585522077297\\
        0.0285714246047077  0.00201959695286336\\
        0.0249999978650042  0.00142441718850428\\
        0.0222222025230041  0.00135645961453701\\
        0.0199999972740879  0.000938890927238312\\
        0.0181818117125527  0.000857570289481302\\
        0.0166666656465241  0.000692143742142104\\
        0.0153846118094882  0.000629169728049828\\
        0.0142857138581073  0.000474812038244288\\
        };
        \addlegendentry{$\varrho = 0.0014$} 

        \addplot [color=black, forget plot]
          table[row sep=crcr]{%
        0.3   0.07\\
        3e-05  7e-06\\
        };
        \node[right] at (6e-03, 1.35e-03) {\tiny 1st order};
        \addplot [color=black, forget plot]
          table[row sep=crcr]{%
        0.3   0.07\\
        3e-05  7e-10\\
        };
        \node[above right] at (6e-03, 2.6e-05) {\tiny 2nd order};
        \addplot [color=black, forget plot]
          table[row sep=crcr]{%
        0.3   0.07\\
        3e-05  7e-14\\
        };
        \node[above right] at (1.23e-02, 3e-06) {\tiny 3rd order};
        \end{axis}

        \end{tikzpicture}
        }
        \phantomsubcaption
    \end{subfigure}
    \hfill
    \begin{subfigure}[b]{0.490\textwidth}
        \centering
        \resizebox{\textwidth}{!}{
        \begin{tikzpicture}
        \begin{axis}[%
        width=2.55in,
        height=1.5in,
        scale only axis,
        xmode=log,
        xmin=6e-03,
        xmax=0.3,
        xminorticks=true,
        xlabel style={font=\color{white!15!black}},
        xlabel={\large fill distance},
        ymode=log,
        ymin=3e-03,
        ymax=70,
        yminorticks=true,
        ylabel style={font=\color{white!15!black}},
        ylabel={\large max norm error},
        title style={align=center, font=\bfseries},
        title={\small 3d L-shaped Domain, $d = 3, p = 2, u_{2}^{(3d)}$},
        xmajorgrids,
        xminorgrids,
        ymajorgrids,
        yminorgrids,
        legend style={at={(0.97,0.03)}, nodes={scale=0.5, transform shape}, anchor=south east, legend cell align=left, align=left, draw=white!15!black}
        ]
        \addplot [color=red, line width=1pt] table[row sep=crcr]{%
        0.199976856129712   1.52885379131499\\
        0.0999950754752418  0.463019844893321\\
        0.0666664718175454  0.180024723895549\\
        0.0499997004247293  0.0977585720678391\\
        0.0399997888105867  0.059668105483869\\
        0.0333331617388861  0.0465563493191112\\
        0.0285714246047077  0.0397680724122402\\
        0.0249999978650042  0.0251999524816959\\
        0.0222222025230041  0.0219145246152088\\
        0.0199999972740879  0.0154404100273773\\
        0.0181818117125527  0.0134350568354922\\
        0.0166666656465241  0.0115916967710703\\
        0.0153846118094882  0.00997051102446633\\
        0.0142857138581073  0.0085327290215389\\
        };
        \addlegendentry{$\varrho = 1.0000$} 

        \addplot [color=blue, line width=1pt] table[row sep=crcr]{%
        0.199976856129712   2.05571516860588\\
        0.0999950754752418  0.434703527154692\\
        0.0666664718175454  0.291380383746624\\
        0.0499997004247293  0.154698459159347\\
        0.0399997888105867  0.0998507878954555\\
        0.0333331617388861  0.0773174405740171\\
        0.0285714246047077  0.0541201962746918\\
        0.0249999978650042  0.0355053492802622\\
        0.0222222025230041  0.03138299295642\\
        0.0199999972740879  0.0218398971404383\\
        0.0181818117125527  0.020080480869952\\
        0.0166666656465241  0.0159026571940366\\
        0.0153846118094882  0.0137026996878351\\
        0.0142857138581073  0.0111778094513095\\
        };
        \addlegendentry{$\varrho = 0.2500$} 

        \addplot [color=green, line width=1pt] table[row sep=crcr]{%
        0.199976856129712   3.2300463648275\\
        0.0999950754752418  1.24022547262672\\
        0.0666664718175454  0.420450205293712\\
        0.0499997004247293  0.268554738466769\\
        0.0399997888105867  0.186148731467219\\
        0.0333331617388861  0.11583833239656\\
        0.0285714246047077  0.100009155213453\\
        0.0249999978650042  0.0739935083383934\\
        0.0222222025230041  0.0509520433435959\\
        0.0199999972740879  0.0406415837423211\\
        0.0181818117125527  0.0383364089083678\\
        0.0166666656465241  0.0316157915376714\\
        0.0153846118094882  0.0258026373095745\\
        0.0142857138581073  0.0225145046627091\\
        };
        \addlegendentry{$\varrho = 0.0864$} 

        \addplot [color=lime, line width=1pt] table[row sep=crcr]{%
        0.199976856129712   17.7262806571091\\
        0.0999950754752418  5.67669023251261\\
        0.0666664718175454  2.55298051123786\\
        0.0499997004247293  1.24387589773097\\
        0.0399997888105867  0.815585470769981\\
        0.0333331617388861  0.527383063460373\\
        0.0285714246047077  0.445927864691962\\
        0.0249999978650042  0.317149155371329\\
        0.0222222025230041  0.246496560473775\\
        0.0199999972740879  0.2129454457741\\
        0.0181818117125527  0.144094964448369\\
        0.0166666656465241  0.129127806389615\\
        0.0153846118094882  0.112960368309666\\
        0.0142857138581073  0.0931180885245091\\
        };
        \addlegendentry{$\varrho = 0.0250$} 

        \addplot [color=orange, line width=1pt] table[row sep=crcr]{%
        0.199976856129712   41.5273247342514\\
        0.0999950754752418  41.577213037705\\
        0.0666664718175454  33.3256070522929\\
        0.0499997004247293  12.887578330347\\
        0.0399997888105867  8.28822761261404\\
        0.0333331617388861  6.23634170930789\\
        0.0285714246047077  4.57959614411098\\
        0.0249999978650042  2.80770808874557\\
        0.0222222025230041  2.26991563767774\\
        0.0199999972740879  1.75202348765021\\
        0.0181818117125527  1.42807054158502\\
        0.0166666656465241  1.26648790917505\\
        0.0153846118094882  1.12116363882667\\
        0.0142857138581073  0.89196933413713\\
        };
        \addlegendentry{$\varrho = 0.0025$} 

        \addplot [color=cyan, line width=1pt, opacity=0.6] table[row sep=crcr]{%
        0.199976856129712   32.8884210219002\\
        0.0999950754752418  31.05929013917\\
        0.0666664718175454  19.553113190792\\
        0.0499997004247293  8.02143710744921\\
        0.0399997888105867  4.97142678814014\\
        0.0333331617388861  3.33718414194545\\
        0.0285714246047077  2.41302282824337\\
        0.0249999978650042  1.77278915306158\\
        0.0222222025230041  1.22963876067363\\
        0.0199999972740879  1.14690075617701\\
        0.0181818117125527  1.00557829467404\\
        0.0166666656465241  0.827879057898993\\
        0.0153846118094882  0.770857662717523\\
        0.0142857138581073  0.568956309917684\\
        };
        \addlegendentry{$\varrho = 0.0014$} 

        \addplot [color=black, forget plot]
          table[row sep=crcr]{%
        0.3   70\\
        3e-05  7e-03\\
        };
        \node[right] at (6e-03, 1.35e-00) {\tiny 1st order};
        \addplot [color=black, forget plot]
          table[row sep=crcr]{%
        0.3   70\\
        3e-05  7e-07\\
        };
        \node[above right] at (6e-03, 2.6e-02) {\tiny 2nd order};
        \addplot [color=black, forget plot]
          table[row sep=crcr]{%
        0.3   70\\
        3e-05  7e-11\\
        };
        \node[above right] at (1.23e-02, 3e-03) {\tiny 3rd order};
        \end{axis}

        \end{tikzpicture}
        }
        \phantomsubcaption
    \end{subfigure}
    \caption{3d tests on the 3d L-shaped domain with continuous coefficient matrices}
    \label{fig:3d_continuous_test_2}
\end{figure}

\subsubsection{Tests for discontinuous coefficient matrices}
We now show numerical results for discontinuous coefficient matrices. 
Again, we divide the computational domains into smaller blocks and define piecewise constant coefficient matrices with respect to the blocks.
More specifically, for $n \in \mathbb{N}$, we divide the domain $[-1, 1]^{3}$ into $(2 n + 1)^{3}$ blocks, and define the corresponding piecewise constant 
coefficient matrix 
\[
A_{\psi}(\boldsymbol{x}, n) := (B_{\psi}(\boldsymbol{x}, n) + B_{\psi}^{T}(\boldsymbol{x}, n) + 4 I) / 10,
\]
where $B_{\psi}(\boldsymbol{x}, n)$ is generated by \texttt{mt19937} with seed $\psi(\boldsymbol{x}, n) := \text{round}(x_{1} * n) * 2 + \text{round}(x_{2} * n) * 3 + \text{round}(x_{3} * n) * 5 \mod 2^{32}$.
A list of coefficient matrices used in our experiments is given below. 
\begin{equation*}
\begin{matrix}
\# & A(\boldsymbol{x}) & \varrho & \text{description}\\
6 & A_{\psi}(\boldsymbol{x}, 10^{10}) & 0.1847 & \text{dense blocks}\\
7 & A_{\psi}(\boldsymbol{x}, 10^{4}) & 0.1847 & \text{medium blocks}\\
8 & A_{\psi}(\boldsymbol{x}, 10^{0}) & 0.1847 & \text{loose blocks}
\end{matrix}
\end{equation*}
In addition, we have the last example matrix:
\begin{equation*}
A_{9}(\boldsymbol{x}) =
\left \{
\begin{aligned}
&A_{2}(\boldsymbol{x}), & x_{1} < 0\\
&A_{3}(\boldsymbol{x}), &\text{otherwise}\\
\end{aligned}
\right .
\qquad \text{with} \qquad
\varrho = 0.0250.
\end{equation*}
Numerical results are presented in \cref{fig:3d_discontinuous_test_1,fig:3d_discontinuous_test_2}.
We observe similar second-order convergence in $h$ for all cases. 

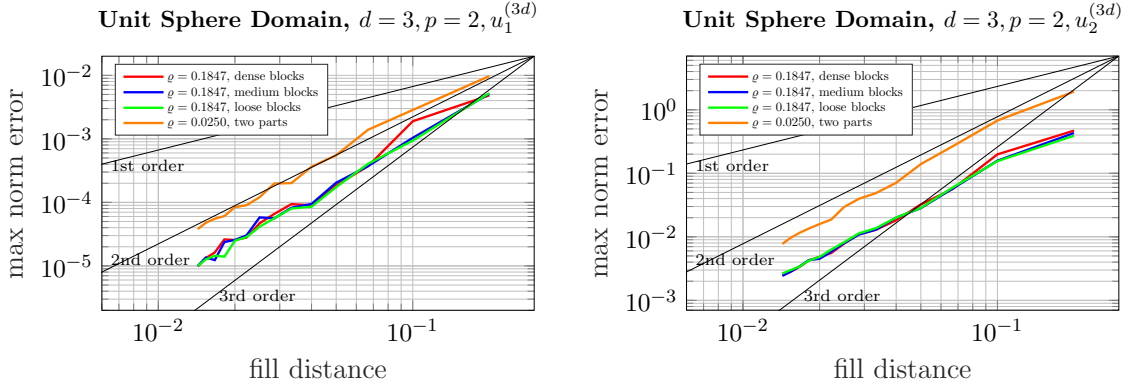
\begin{figure}[htp]
    \centering
    \begin{subfigure}[b]{0.490\textwidth}
        \centering
        \resizebox{\textwidth}{!}{
        \begin{tikzpicture}
        \begin{axis}[%
        width=2.55in,
        height=1.5in,
        scale only axis,
        xmode=log,
        xmin=6e-03,
        xmax=0.3,
        xminorticks=true,
        xlabel style={font=\color{white!15!black}},
        xlabel={\large fill distance},
        ymode=log,
        ymin=2e-06,
        ymax=0.02,
        yminorticks=true,
        ylabel style={font=\color{white!15!black}},
        ylabel={\large max norm error},
        title style={align=center, font=\bfseries},
        title={\small Unit Sphere Domain, $d = 3, p = 2, u_{1}^{(3d)}$},
        xmajorgrids,
        xminorgrids,
        ymajorgrids,
        yminorgrids,
        legend style={at={(0.03,0.97)}, nodes={scale=0.5, transform shape}, anchor=north west, legend cell align=left, align=left, draw=white!15!black}
        ]
        \addplot [color=red, line width=1pt] table[row sep=crcr]{%
        0.199976856129712   0.0048427092521699\\
        0.0999950754752416  0.00189236164146811\\
        0.0666664718175454  0.000371329753560801\\
        0.0499997004247293  0.000185319851477406\\
        0.0399997888105867  9.03846760178517e-05\\
        0.0333331617388861  9.38162515478069e-05\\
        0.0285714246047077  6.59189630050072e-05\\
        0.0249999978650042  4.68329675591406e-05\\
        0.0222222025230041  2.79531881330897e-05\\
        0.0199999972740879  2.55770238997144e-05\\
        0.0181818117125527  2.6148010286775e-05\\
        0.0166666656465241  1.6257025622135e-05\\
        0.0153846118094882  1.33187964075354e-05\\
        0.0142857138581073  9.80098114400363e-06\\
        };
        \addlegendentry{$\varrho = 0.1847$, dense blocks} 

        \addplot [color=blue, line width=1pt] table[row sep=crcr]{%
        0.199976856129712   0.00511775517064361\\
        0.0999950754752416  0.00103755900738278\\
        0.0666664718175454  0.000371463261691574\\
        0.0499997004247293  0.000204103647598153\\
        0.0399997888105867  9.45216840122498e-05\\
        0.0333331617388861  8.23372389073818e-05\\
        0.0285714246047077  5.60580863400162e-05\\
        0.0249999978650042  5.79983047517274e-05\\
        0.0222222025230041  3.03183953458408e-05\\
        0.0199999972740879  2.56208671669533e-05\\
        0.0181818117125527  2.37614597478064e-05\\
        0.0166666656465241  1.23255301720526e-05\\
        0.0153846118094882  1.35802481349145e-05\\
        0.0142857138581073  9.98850730615786e-06\\
        };
        \addlegendentry{$\varrho = 0.1847$, medium blocks} 

        \addplot [color=green, line width=1pt] table[row sep=crcr]{%
        0.199976856129712   0.00524226325708543\\
        0.0999950754752416  0.000938857630543133\\
        0.0666664718175454  0.000403381289623805\\
        0.0499997004247293  0.000174049858681524\\
        0.0399997888105867  8.54613855074682e-05\\
        0.0333331617388861  8.00378190555229e-05\\
        0.0285714246047077  5.62675837230131e-05\\
        0.0249999978650042  4.14540349407133e-05\\
        0.0222222025230041  2.85051500528155e-05\\
        0.0199999972740879  2.49416875868214e-05\\
        0.0181818117125527  1.39910981764579e-05\\
        0.0166666656465241  1.44719740431398e-05\\
        0.0153846118094882  1.27023047102881e-05\\
        0.0142857138581073  9.98558285258966e-06\\
        };
        \addlegendentry{$\varrho = 0.1847$, loose blocks} 

        \addplot [color=orange, line width=1pt] table[row sep=crcr]{%
        0.199976856129712   0.00972384797109149\\
        0.0999950754752416  0.00284950347370883\\
        0.0666664718175454  0.00138712345349701\\
        0.0499997004247293  0.00055194085982313\\
        0.0399997888105867  0.000364413630965732\\
        0.0333331617388861  0.000201061500207267\\
        0.0285714246047077  0.000198185441731891\\
        0.0249999978650042  0.000120566042727255\\
        0.0222222025230041  9.0357564797916e-05\\
        0.0199999972740879  8.43146355573587e-05\\
        0.0181818117125527  6.08794966954207e-05\\
        0.0166666656465241  5.53661511357717e-05\\
        0.0153846118094882  4.78805507491131e-05\\
        0.0142857138581073  3.80393924848121e-05\\
        };
        \addlegendentry{$\varrho = 0.0250$, two parts} 

        \addplot [color=black, forget plot]
          table[row sep=crcr]{%
        0.3   0.02\\
        3e-05  2e-06\\
        };
        \node[right] at (6e-03, 3.85e-04) {\tiny 1st order};
        \addplot [color=black, forget plot]
          table[row sep=crcr]{%
        0.3   0.02\\
        3e-05  2e-10\\
        };
        \node[above right] at (6e-03, 7.4e-06) {\tiny 2nd order};
        \addplot [color=black, forget plot]
          table[row sep=crcr]{%
        0.3   0.02\\
        3e-05  2e-14\\
        };
        \node[above right] at (1.6e-02, 2e-06) {\tiny 3rd order};
        \end{axis}

        \end{tikzpicture}
        }
        \phantomsubcaption
    \end{subfigure}
    \hfill
    \begin{subfigure}[b]{0.490\textwidth}
        \centering
        \resizebox{\textwidth}{!}{
        \begin{tikzpicture}
        \begin{axis}[%
        width=2.55in,
        height=1.5in,
        scale only axis,
        xmode=log,
        xmin=6e-03,
        xmax=0.3,
        xminorticks=true,
        xlabel style={font=\color{white!15!black}},
        xlabel={\large fill distance},
        ymode=log,
        ymin=7e-04,
        ymax=7,
        yminorticks=true,
        ylabel style={font=\color{white!15!black}},
        ylabel={\large max norm error},
        title style={align=center, font=\bfseries},
        title={\small Unit Sphere Domain, $d = 3, p = 2, u_{2}^{(3d)}$},
        xmajorgrids,
        xminorgrids,
        ymajorgrids,
        yminorgrids,
        legend style={at={(0.03,0.97)}, nodes={scale=0.5, transform shape}, anchor=north west, legend cell align=left, align=left, draw=white!15!black}
        ]
        \addplot [color=red, line width=1pt] table[row sep=crcr]{%
        0.199976856129712   0.469424245854305\\
        0.0999950754752416  0.198637251761131\\
        0.0666664718175454  0.0564724535758151\\
        0.0499997004247293  0.0323408272113053\\
        0.0399997888105867  0.0182268679181621\\
        0.0333331617388861  0.0129178890904957\\
        0.0285714246047077  0.0110199470202041\\
        0.0249999978650042  0.00785512065285054\\
        0.0222222025230041  0.00554933986775374\\
        0.0199999972740879  0.00483897404248523\\
        0.0181818117125527  0.00425782534463548\\
        0.0166666656465241  0.00336807825255092\\
        0.0153846118094882  0.00280921828524205\\
        0.0142857138581073  0.00249695987553755\\
        };
        \addlegendentry{$\varrho = 0.1847$, dense blocks} 

        \addplot [color=blue, line width=1pt] table[row sep=crcr]{%
        0.199976856129712   0.430634641311846\\
        0.0999950754752416  0.157987952744666\\
        0.0666664718175454  0.0573744549478546\\
        0.0499997004247293  0.028269233658988\\
        0.0399997888105867  0.0200298726219093\\
        0.0333331617388861  0.0129877884206915\\
        0.0285714246047077  0.0108491580664474\\
        0.0249999978650042  0.00797442253762903\\
        0.0222222025230041  0.00582892568075533\\
        0.0199999972740879  0.00450561630664437\\
        0.0181818117125527  0.00431942571172717\\
        0.0166666656465241  0.00340221337919711\\
        0.0153846118094882  0.00282670080850789\\
        0.0142857138581073  0.00242571947167369\\
        };
        \addlegendentry{$\varrho = 0.1847$, medium blocks} 

        \addplot [color=green, line width=1pt] table[row sep=crcr]{%
        0.199976856129712   0.391047157962815\\
        0.0999950754752416  0.155129979044834\\
        0.0666664718175454  0.0597029239458686\\
        0.0499997004247293  0.0285117887160364\\
        0.0399997888105867  0.020083827396461\\
        0.0333331617388861  0.0136875219995964\\
        0.0285714246047077  0.0113549777049582\\
        0.0249999978650042  0.00823516572794936\\
        0.0222222025230041  0.00628978259040291\\
        0.0199999972740879  0.00482580721428016\\
        0.0181818117125527  0.00424913521354786\\
        0.0166666656465241  0.00339260281624787\\
        0.0153846118094882  0.00300220055578415\\
        0.0142857138581073  0.00264440438117197\\
        };
        \addlegendentry{$\varrho = 0.1847$, loose blocks} 

        \addplot [color=orange, line width=1pt] table[row sep=crcr]{%
        0.199976856129712   1.92979859706332\\
        0.0999950754752416  0.677240045985596\\
        0.0666664718175454  0.265750175536597\\
        0.0499997004247293  0.139441668193262\\
        0.0399997888105867  0.0709700006772289\\
        0.0333331617388861  0.0487693963712732\\
        0.0285714246047077  0.0397349298771958\\
        0.0249999978650042  0.029574157959114\\
        0.0222222025230041  0.018700016075184\\
        0.0199999972740879  0.0158872131406591\\
        0.0181818117125527  0.0135471023751004\\
        0.0166666656465241  0.01154186018662\\
        0.0153846118094882  0.00964573699179239\\
        0.0142857138581073  0.00773783475577794\\
        };
        \addlegendentry{$\varrho = 0.0250$, two parts} 

        \addplot [color=black, forget plot]
          table[row sep=crcr]{%
        0.3   7\\
        3e-05  7e-04\\
        };
        \node[right] at (6e-03, 1.35e-01) {\tiny 1st order};
        \addplot [color=black, forget plot]
          table[row sep=crcr]{%
        0.3   7\\
        3e-05  7e-08\\
        };
        \node[above right] at (6e-03, 2.6e-03) {\tiny 2nd order};
        \addplot [color=black, forget plot]
          table[row sep=crcr]{%
        0.3   7\\
        3e-05  7e-12\\
        };
        \node[above right] at (1.6e-02, 7e-04) {\tiny 3rd order};
        \end{axis}

        \end{tikzpicture}
        }
        \phantomsubcaption
    \end{subfigure}
    \caption{3d tests on the unit sphere domain with discontinuous coefficient matrices}
    \label{fig:3d_discontinuous_test_1}
\end{figure}

\begin{figure}[htp]
    \centering
    \begin{subfigure}[b]{0.490\textwidth}
        \centering
        \resizebox{\textwidth}{!}{
        \begin{tikzpicture}
        \begin{axis}[%
        width=2.55in,
        height=1.5in,
        scale only axis,
        xmode=log,
        xmin=6e-03,
        xmax=0.3,
        xminorticks=true,
        xlabel style={font=\color{white!15!black}},
        xlabel={\large fill distance},
        ymode=log,
        ymin=2e-06,
        ymax=0.02,
        yminorticks=true,
        ylabel style={font=\color{white!15!black}},
        ylabel={\large max norm error},
        title style={align=center, font=\bfseries},
        title={\small 3d L-shaped Domain, $d = 3, p = 2, u_{1}^{(3d)}$},
        xmajorgrids,
        xminorgrids,
        ymajorgrids,
        yminorgrids,
        legend style={at={(0.03,0.97)}, nodes={scale=0.5, transform shape}, anchor=north west, legend cell align=left, align=left, draw=white!15!black}
        ]
        \addplot [color=red, line width=1pt] table[row sep=crcr]{%
        0.199976856129712   0.00713124088537853\\
        0.0999950754752418  0.00150429222800952\\
        0.0666664718175454  0.000554537216850726\\
        0.0499997004247293  0.000243947374046538\\
        0.0399997888105867  0.000162583903009939\\
        0.0333331617388861  0.000150281968947397\\
        0.0285714246047077  8.60134232967269e-05\\
        0.0249999978650042  6.74830280305017e-05\\
        0.0222222025230041  4.48648086432968e-05\\
        0.0199999972740879  3.63067744233447e-05\\
        0.0181818117125527  2.63470002668242e-05\\
        0.0166666656465241  2.37262446836084e-05\\
        0.0153846118094882  1.85597548938432e-05\\
        0.0142857138581073  1.62248068362736e-05\\
        };
        \addlegendentry{$\varrho = 0.1847$, dense blocks} 

        \addplot [color=blue, line width=1pt] table[row sep=crcr]{%
        0.199976856129712   0.00556369938950541\\
        0.0999950754752418  0.00121020561628349\\
        0.0666664718175454  0.000618425449898652\\
        0.0499997004247293  0.000228869472693738\\
        0.0399997888105867  0.000146830969838696\\
        0.0333331617388861  9.59884876388095e-05\\
        0.0285714246047077  8.35332695912072e-05\\
        0.0249999978650042  5.59916667883797e-05\\
        0.0222222025230041  4.64918371414491e-05\\
        0.0199999972740879  3.81644613989263e-05\\
        0.0181818117125527  2.94780793028693e-05\\
        0.0166666656465241  2.32376399784684e-05\\
        0.0153846118094882  2.16240653925226e-05\\
        0.0142857138581073  1.78206761951571e-05\\
        };
        \addlegendentry{$\varrho = 0.1847$, medium blocks} 

        \addplot [color=green, line width=1pt] table[row sep=crcr]{%
        0.199976856129712   0.00604135935821759\\
        0.0999950754752418  0.0015433558828688\\
        0.0666664718175454  0.00054106461843606\\
        0.0499997004247293  0.000224152651789655\\
        0.0399997888105867  0.00013940420691938\\
        0.0333331617388861  9.19699165859988e-05\\
        0.0285714246047077  6.66789714109939e-05\\
        0.0249999978650042  5.62096791849598e-05\\
        0.0222222025230041  4.47007033388402e-05\\
        0.0199999972740879  3.28618902076805e-05\\
        0.0181818117125527  2.17860127222913e-05\\
        0.0166666656465241  1.78671748329862e-05\\
        0.0153846118094882  1.78922504723289e-05\\
        0.0142857138581073  1.67065371803332e-05\\
        };
        \addlegendentry{$\varrho = 0.1847$, loose blocks} 

        \addplot [color=orange, line width=1pt] table[row sep=crcr]{%
        0.199976856129712   0.0105714566324464\\
        0.0999950754752418  0.00284684642283795\\
        0.0666664718175454  0.00112235476515687\\
        0.0499997004247293  0.000423214156735741\\
        0.0399997888105867  0.00025026951823981\\
        0.0333331617388861  0.000190365737866127\\
        0.0285714246047077  0.000184708496879793\\
        0.0249999978650042  0.000122238328136959\\
        0.0222222025230041  7.46017577748503e-05\\
        0.0199999972740879  7.05182525075898e-05\\
        0.0181818117125527  6.77976138385716e-05\\
        0.0166666656465241  4.66807894010657e-05\\
        0.0153846118094882  4.16215972038891e-05\\
        0.0142857138581073  3.76288925254542e-05\\
        };
        \addlegendentry{$\varrho = 0.0250$, two parts} 

        \addplot [color=black, forget plot]
          table[row sep=crcr]{%
        0.3   0.02\\
        3e-05  2e-06\\
        };
        \node[right] at (6e-03, 3.85e-04) {\tiny 1st order};
        \addplot [color=black, forget plot]
          table[row sep=crcr]{%
        0.3   0.02\\
        3e-05  2e-10\\
        };
        \node[above right] at (6e-03, 7.4e-06) {\tiny 2nd order};
        \addplot [color=black, forget plot]
          table[row sep=crcr]{%
        0.3   0.02\\
        3e-05  2e-14\\
        };
        \node[above right] at (1.6e-02, 2e-06) {\tiny 3rd order};
        \end{axis}

        \end{tikzpicture}
        }
        \phantomsubcaption
    \end{subfigure}
    \hfill
    \begin{subfigure}[b]{0.490\textwidth}
        \centering
        \resizebox{\textwidth}{!}{
        \begin{tikzpicture}
        \begin{axis}[%
        width=2.55in,
        height=1.5in,
        scale only axis,
        xmode=log,
        xmin=6e-03,
        xmax=0.3,
        xminorticks=true,
        xlabel style={font=\color{white!15!black}},
        xlabel={\large fill distance},
        ymode=log,
        ymin=7e-04,
        ymax=7,
        yminorticks=true,
        ylabel style={font=\color{white!15!black}},
        ylabel={\large max norm error},
        title style={align=center, font=\bfseries},
        title={\small 3d L-shaped Domain, $d = 3, p = 2, u_{2}^{(3d)}$},
        xmajorgrids,
        xminorgrids,
        ymajorgrids,
        yminorgrids,
        legend style={at={(0.03,0.97)}, nodes={scale=0.5, transform shape}, anchor=north west, legend cell align=left, align=left, draw=white!15!black}
        ]
        \addplot [color=red, line width=1pt] table[row sep=crcr]{%
        0.199976856129712   2.18128131475812\\
        0.0999950754752418  0.788757221864909\\
        0.0666664718175454  0.248374342305905\\
        0.0499997004247293  0.167618998348903\\
        0.0399997888105867  0.15652741898159\\
        0.0333331617388861  0.0806867249409144\\
        0.0285714246047077  0.053272535885192\\
        0.0249999978650042  0.0509903164193979\\
        0.0222222025230041  0.0302995160968997\\
        0.0199999972740879  0.0241088493748123\\
        0.0181818117125527  0.0256409973785274\\
        0.0166666656465241  0.0220245273129542\\
        0.0153846118094882  0.0193312899696032\\
        0.0142857138581073  0.0154253297704301\\
        };
        \addlegendentry{$\varrho = 0.1847$, dense blocks} 

        \addplot [color=blue, line width=1pt] table[row sep=crcr]{%
        0.199976856129712   1.81495813325303\\
        0.0999950754752418  0.609890947020183\\
        0.0666664718175454  0.427604194407916\\
        0.0499997004247293  0.157759740549483\\
        0.0399997888105867  0.116040157220869\\
        0.0333331617388861  0.083025572702379\\
        0.0285714246047077  0.0581234050823305\\
        0.0249999978650042  0.0481026993034881\\
        0.0222222025230041  0.0322194689077868\\
        0.0199999972740879  0.0270813747280556\\
        0.0181818117125527  0.023275305889296\\
        0.0166666656465241  0.0193694172419008\\
        0.0153846118094882  0.0167543955698619\\
        0.0142857138581073  0.0143983493454627\\
        };
        \addlegendentry{$\varrho = 0.1847$, medium blocks} 

        \addplot [color=green, line width=1pt] table[row sep=crcr]{%
        0.199976856129712   2.27601643705973\\
        0.0999950754752418  0.792488711312945\\
        0.0666664718175454  0.43892021991444\\
        0.0499997004247293  0.170902536006942\\
        0.0399997888105867  0.122992577122698\\
        0.0333331617388861  0.0838308786272002\\
        0.0285714246047077  0.0655563442119593\\
        0.0249999978650042  0.0493428044403821\\
        0.0222222025230041  0.0394699400014673\\
        0.0199999972740879  0.0272554510953888\\
        0.0181818117125527  0.025858937135645\\
        0.0166666656465241  0.0205682769088931\\
        0.0153846118094882  0.0184250550877785\\
        0.0142857138581073  0.0156458618057442\\
        };
        \addlegendentry{$\varrho = 0.1847$, loose blocks} 

        \addplot [color=orange, line width=1pt] table[row sep=crcr]{%
        0.199976856129712   3.23379644078507\\
        0.0999950754752418  1.47286229269617\\
        0.0666664718175454  0.418627383781597\\
        0.0499997004247293  0.265124804928899\\
        0.0399997888105867  0.183472997720553\\
        0.0333331617388861  0.114443881678636\\
        0.0285714246047077  0.0990737680676901\\
        0.0249999978650042  0.0731505349727399\\
        0.0222222025230041  0.0502326176073602\\
        0.0199999972740879  0.0432521251103637\\
        0.0181818117125527  0.0379649065137606\\
        0.0166666656465241  0.0314035077271129\\
        0.0153846118094882  0.0254491331746785\\
        0.0142857138581073  0.0222981445755757\\
        };
        \addlegendentry{$\varrho = 0.0250$, two parts} 

        \addplot [color=black, forget plot]
          table[row sep=crcr]{%
        0.3   7\\
        3e-05  7e-04\\
        };
        \node[right] at (6e-03, 1.35e-01) {\tiny 1st order};
        \addplot [color=black, forget plot]
          table[row sep=crcr]{%
        0.3   7\\
        3e-05  7e-08\\
        };
        \node[above right] at (6e-03, 2.6e-03) {\tiny 2nd order};
        \addplot [color=black, forget plot]
          table[row sep=crcr]{%
        0.3   7\\
        3e-05  7e-12\\
        };
        \node[above right] at (1.6e-02, 7e-04) {\tiny 3rd order};
        \end{axis}

        \end{tikzpicture}
        }
        \phantomsubcaption
    \end{subfigure}
    \caption{3d tests on the 3d L-shaped domain with discontinuous coefficient matrices}
    \label{fig:3d_discontinuous_test_2}
\end{figure}
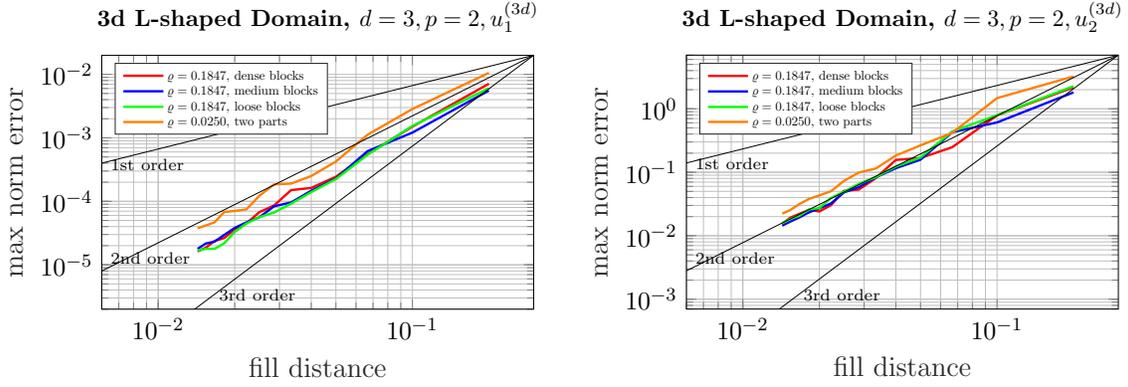


\section{Conclusion}
\label{sec:conclusion}
In this paper, we have presented a monotone meshfree finite difference method for linear elliptic equations in non-divergence form via integral relaxation. 
Minimal positive stencils are found through an $l_1$-type minimization problem within a local elliptical searching neighborhood of each point in a meshfree point cloud. 
For the treatment of Dirichlet boundary conditions, a mapping strategy near the boundary is incorporated into the numerical scheme. 
Convergence is guaranteed by the consistency and monotonicity of the scheme and efficient solvers can be designed by the sparsity of the resulting linear system.  
It is essential to characterize the shape and size of the elliptical searching neighborhood for the guarantee of positive stencils. 
Our theoretical result improves the previously known result for the stencil sizes when $\varrho$, the ratio between the smallest and the largest eigenvalues of the coefficient matrix, is a very small number.  
More precisely, our theory predicts that within an elliptical region with a semi-major axis proportional to $\varrho^{-1/2}h$, we are able to find a positive stencil. 
The searching region determines the size of the $l_1$-type minimization problem, and therefore the efficiency of our algorithm. 
Our theory predicts that the number of points within the searching neighborhood grows with $1/\varrho$ with a rate not worse than $\mathcal{O}(\varrho^{-1/2})$ in 2d and 
$\mathcal{O}(\varrho^{-1})$ in 3d.

We present algorithms for point cloud management and matrix assembly.  
We also give practical guidance for finding the elliptical searching neighborhood and present numerical experiments. 
Numerical tests are presented in both 2d and 3d for several different domains and coefficient matrices, including the near degenerate cases when $\varrho\ll 1$. 
While theoretical convergence in $h$ for the numerical method (when the polynomial order $p=2$) is only first order, we observe second-order convergence in all cases for manufactured smooth solutions. 
The super-convergence is likely due to the cancellation of odd order terms for the stencils obtained from the $l_1$-type  minimization problem.
A rigorous explanation for this phenomenon is still an open question.  

Our current study focuses on the case of $p = 2$ and $d \in \{ 2, 3 \}$ with Dirichlet boundary conditions.
Future work includes higher order methods, problems in higher dimensions, and Neumann boundary value problems.
Our convergence theorem assumes that the exact solution exists at least in $C^2$.
Further questions on the convergence of the method to viscosity solutions can be discussed, following the approaches discussed in \cite{barles1991convergence,feng2013recent,neilan2017numerical}. 
Extending the study to surface PDEs is also a natural direction for future research.
While we only test our algorithm for smooth manufactured solutions,  adaptive methods will be useful when solutions display singularity. 
For adaptive point cloud management, some data structures that support fast insertion and deletion may be needed, for instance, R-tree \cite{guttman1984r} and scapegoat k-d tree \cite{galperin1993scapegoat}. 
Furthermore, a possible future research direction is the study of monotone numerical schemes for elliptic equations with heterogeneous coefficients \cite{abdulle2003finite}.
The topic of monotone schemes for solving PDEs has a long history in numerical analysis. 
While our new ideas, inspired by the recent development of nonlocal modeling and meshfree methods, 
are presented for the linear elliptic equations, extending them to other types of PDEs is also possible for future research.


\section*{Acknowledgement}
This research is supported in part by NSF DMS-2111608 and NSF DMS-2240180.
UC San Diego Research IT Services and Academic Technology Services are acknowledged for providing the Research Cluster computing resource for the numerical simulations in this work.
The authors would like to thank Qiang Du, Xiaobing Feng, Nathaniel Trask and Zhen Zhang for their helpful discussions on the subject.
The authors would also like to thank the anonymous reviewers for their valuable comments and suggestions, which help to strengthen the manuscript.

\section*{Data availability \& Conflict of interest}
The datasets generated during and/or analysed during the current study are available from the corresponding author on reasonable request.
The authors declare that they have no conflict of interest.

\appendix
\renewcommand \theequation {\Alph{section}.\arabic{equation}}

\section{Proof of \texorpdfstring{\Cref{thm:nbh_criterion}}{Theorem 3.9}}
\label{sec:appendixA}

We show the proof of \Cref{thm:nbh_criterion} in \Cref{sec:appendixA}. We begin with some useful lemmas before proving the theorem.

\begin{lem}
\label{lem:appendix_1}
Let $r(\vb,\xb_i)$ denote the radius of the inscribed ball in $T_i (\cC_\del^{\vb}(\xb_i))$ and $h$ be the fill distance asscociated with $X = \{ \xb_i\}_{i=1}^M\subset \Om_\del$. 
If
\[
h < \min_{\vb\in \R^d, |\vb|=1} r(\vb,\xb_i), 
\]
then $S_{\del, h, 2}(\xb_i)$ and $\overline{S}_{\del, h, 2}(\xb_i)$ are not empty. 
\end{lem}
\begin{proof}
Notice that by the definition of the fill distance in \cref{eq:filldistance},  there are no holes with a radius larger than $h$. 
Suppose $S_{\del, h, 2}(\xb_i)$ or $\overline{S}_{\del, h, 2}(\xb_i)$ is empty, then by \Cref{coro:cone_condition}, there exists $\vb$ such that 
$T_i(\cC^{\vb}_\del(\xb_i))$ contains no point in $X\backslash \{ \xb_i\}$. Therefore the inscribed ball in $T_i(\cC^{\vb}_\del(\xb_i))$ is a hole 
with radius larger than $h$ by the assumption, which gives a contradiction. 
\end{proof}

From the lemma above, our goal is then to get a lower bound for $ \min_{\vb\in \R^d, |\vb|=1} r(\vb,\xb_i)$ for each $\xb_i\in\Om$.
We first present a result in 2d which will also be useful for the 3d estimates. 
In 2d, we assume that $\cC^{\vb}_\del(\xb_i)$ is a cone with total opening angle $2\phi$. In addition, without loss of generality, we fix $\xb_i\in\Om$ and assume that 
\beq
\label{eq:2d_matrix_standardform}
A(\xb_i) = 
\begin{pmatrix}
\varrho & 0 \\
0 & 1 \\
\end{pmatrix}. 
\eeq
From symmetry, we only need to consider $\vb(\theta) = (\cos(\theta), \sin(\theta))$ for $\theta\in [0,\frac{\pi}{2}]$. 

\begin{lem}
\label{lem:min_radius}
Consider $d=2$ and $A(\xb_i)$ given by \cref{eq:2d_matrix_standardform}. Assume that $\cC^{\vb}_\del(\xb_i)$ is a cone with total opening angle $2\phi$  for $\phi \in (0,\frac{\pi}{8}]$, and $r(\vb,\xb_i)$ denote the radius of the inscribed circle in $T_i (\cC_\del^{\vb}(\xb_i))$.
In addition, let $\vb(\theta) = (\cos(\theta), \sin(\theta))$ for $\theta\in [0,\frac{\pi}{2}]$. 
Then, there exists a constant $K =K (\phi)>0$ such that
\[
\min_{\theta \in [0, \frac{\pi}{2}]} r(\vb(\theta),\xb_i) \geq K \del \sqrt{\varrho}.
\]
\end{lem}
\begin{proof}
We try to fit a cone in $T_i(\cC_\del^{\vb(\theta)}(\xb_i))$ and then find the inscribed circle in the cone. 
First, notice that for a cone with a radius $R$ and total opening angle $\alpha\in (0,\pi)$, the radius of the inscribed circle is given by the formula
\[
\frac{\sin(\alpha/2)}{1+\sin(\alpha/2)} R \geq \frac{1}{2} \sin(\alpha/2) R. 
\]
Notice that $\frac{1}{2} \sin(\alpha/2) R$  increases with $\alpha \in (0,\pi)$ and $R$. 
Now for $\theta\in [0, \frac{\pi}{2}]$, we let $\Gamma(\theta)$ denote the opening angle of $T_i(\cC_\del^{\vb(\theta)}(\xb_i))$ and define
\[
R(\theta) := \min_{\varphi\in [\theta -\phi, \theta +\phi]}\del \sqrt{\varrho \cos^2(\varphi) + \sin^2(\varphi)}   = \min_{\varphi\in [\theta -\phi, \theta +\phi]}\del \sqrt{\varrho  + (1-\varrho) \sin^2(\varphi)},
\]
then it is easy to see that a cone with radius $R(\theta)$ and total opening angle $\Gamma(\theta)$ is contained in $T_i(\cC_\del^{\vb(\theta)}(\xb_i))$.  
Therefore we have
\[
\min_{\theta \in [0, \frac{\pi}{2}]} r(\vb(\theta),\xb_i) \geq  \min_{\theta \in [0, \frac{\pi}{2}]} \frac{1}{2} \sin(\Gamma(\theta)/2) R(\theta).  
\]
By calculation we have
\[
\Gamma(\theta) = 
\left\{ 
\begin{aligned}
\arctan(\sqrt{\varrho^{-1}} \tan(\theta + \phi)) -  \arctan(\sqrt{\varrho^{-1}} \tan(\theta - \phi)), \quad  &\theta \in [0, \pi/2-\phi), \\
\pi +   \arctan(\sqrt{\varrho^{-1}} \tan(\theta + \phi)) -  \arctan(\sqrt{\varrho^{-1}} \tan(\theta - \phi)), \quad  &\theta \in ( \pi/2-\phi, \pi/2], 
\end{aligned}
\right. 
\]
and
\[
R(\theta) =
\left\{ 
\begin{aligned}
&\del \sqrt{\varrho},   \quad &\theta \in [0, \phi], \\
&\del \sqrt{\varrho  + (1-\varrho) \sin^2(\theta - \phi)},   \quad &\theta \in [\phi, \pi/2]. 
\end{aligned}
\right. 
\]

For $\theta \in [0, \phi]  $, $\Gamma(\theta)$ decreases and $R(\theta) = \del \sqrt{\varrho}$, so 
\[
\min_{\theta \in [0,\phi]} \frac{1}{2} \sin(\Gamma(\theta)/2) R(\theta) = \frac{1}{2} \sin(\Gamma(\phi)/2) \del \sqrt{\varrho} \geq \frac{\sin(\phi)}{2}\del \sqrt{\varrho}
\]
where we have used  $2\phi\leq \Gamma(\phi)\leq \pi/2$. 

For $\theta \in [\pi/2 -2\phi, \pi/2] $,  $\Gamma(\theta)$ decreases and $R(\theta)$ increases, so 
\[
\begin{split}
&\quad \ \min_{\theta \in [\pi/2 -2\phi, \pi/2] } \frac{1}{2} \sin(\Gamma(\theta)/2) R(\theta)\\
&\geq  \frac{1}{2} \sin(\Gamma(\pi/2)/2) R(\pi/2 -2\phi)  
\geq \frac{\del}{2}\sin(\Gamma(\pi/2)/2) \sin(\pi/8)  \\
& =   \frac{\sin(\pi/8) \del }{2} \frac{\sqrt{\varrho}}{\sqrt{\varrho + \tan^2(\pi/2-\phi)}}  \geq \frac{\sin(\pi/8)}{2\sqrt{1 + \tan^2(3\pi/8)}}  \del\sqrt{\rho}. 
\end{split}
\]
where we have used $\phi\leq \pi/8$  and $\Gamma(\pi/2) = 2 \arctan(\sqrt{\varrho} \cot(\pi/2-\phi))$.   

Now for $\theta \in [\phi , \pi/2-2\phi]$, we use the formulas for $\Gamma(\theta)$ and $R(\theta)$ to compute $\frac{1}{2} \sin(\Gamma(\theta)/2) R(\theta) $. 
Denote $\alpha= \arctan(\sqrt{\varrho^{-1}} \tan(\theta + \phi))$ and $\beta = \arctan(\sqrt{\varrho^{-1}} \tan(\theta - \phi))$. Use the formula
\[
\sin\left(\frac{\alpha-\beta}{2}\right) =  \sqrt{\frac{1-\cos(\alpha-\beta)}{2}} = \sqrt{\frac{1-\cos(\alpha) \cos(\beta) - \sin(\alpha)\sin(\beta)}{2}} ,
\]
and the fact that 
\[
\begin{split}
&\sin(\alpha) = \frac{\sqrt{\varrho^{-1}} \tan(\theta + \phi)}{ \sqrt{1+\varrho^{-1}\tan^2(\theta+\phi)}},\; \cos(\alpha) = \frac{1}{ \sqrt{1+\varrho^{-1}\tan^2(\theta+\phi)}}, \\
& \sin(\beta) = \frac{\sqrt{\varrho^{-1}} \tan(\theta - \phi)}{ \sqrt{1+\varrho^{-1}\tan^2(\theta-\phi)}},\; \cos(\beta) = \frac{1}{ \sqrt{1+\varrho^{-1}\tan^2(\theta-\phi)}},
\end{split}
\]
we can obtain the formula for $\frac{1}{2} \sin(\Gamma(\theta)/2) R(\theta)$ where $\theta \in [\phi , \pi/2-2\phi]$. 
In particular, denote $g_1(\theta) = \tan^2(\theta - \phi)$,  $g_2(\theta) = \tan^2(\theta + \phi)$, and $g_3(\theta)= \sin^2(\theta-\phi)$, we fine
\[
\begin{split}
&\frac{1}{2} \sin(\Gamma(\theta)/2) R(\theta) \\
=&\frac{\del}{2\sqrt{2}}  \sqrt{\frac{\sqrt{(\varrho+g_1(\theta)) (\varrho+g_2(\theta)) } -\varrho- \sqrt{g_1(\theta)g_2(\theta)}  }{\sqrt{(\varrho+g_1(\theta)) (\varrho+g_2(\theta)) }} \left(\varrho  + (1-\varrho) g_3(\theta))\right)} \\
=& \frac{\del \sqrt{\rho}}{2}  \sqrt{\frac{ \left((g_1(\theta))^{1/2}-(g_2(\theta))^{1/2}\right)^2  \left(\varrho  + (1-\varrho) g_3(\theta))\right) }{\sqrt{(\varrho+g_1(\theta)) (\varrho+g_2(\theta)) } \left(\sqrt{(\varrho+g_1(\theta)) (\varrho+g_2(\theta)) } +\varrho + \sqrt{g_1(\theta)g_2(\theta)}\right) }} \\
=: & \frac{\del \sqrt{\rho}}{2} G(\varrho, \theta). 
\end{split}
\]
Notice that $ G(\varrho, \theta)$ defined in the above is a continuous function on $(\varrho, \theta) \in [0,1]\times [\phi, \pi/2-2\phi]$. 
Therefore it attains a minimum value at some $(\varrho^\ast, \theta^\ast) \in  [0,1]\times [\phi, \pi/2-2\phi]$. Next we show that we must have $ G(\varrho^\ast, \theta^\ast)>0$. 
Indeed, if $\varrho^\ast >0$, then it is easy to see that  $ G(\varrho^\ast, \theta^\ast)>0$. Now if $\varrho^\ast =0$, then
\[
\begin{split}
G(0, \theta) &= \sqrt{\frac{ \left((g_1(\theta))^{1/2}-(g_2(\theta))^{1/2}\right)^2  g_3(\theta) }{2g_1(\theta) g_2(\theta)}}  \\
&=  \left(\tan(\theta + \phi) -  \tan(\theta - \phi)\right)  \frac{\sin(\theta-\phi)}{2 \tan(\theta - \phi)  \tan(\theta + \phi)  } \\
& = \left(\tan(\theta + \phi) -  \tan(\theta - \phi)\right)  \frac{\cos(\theta-\phi)}{2  \tan(\theta + \phi)  }>0
\end{split}
\] 
for any $\theta\in[\phi, \pi/2-2\phi]$.  Therefore, we can take 
\[
K(\phi) = \min\left\{ \frac{\sin(\phi)}{2} ,\frac{\sin(\pi/8)}{2\sqrt{1 + \tan^2(3\pi/8)}}, G(\varrho^\ast, \theta^\ast)/2 \right\} >0 
\]
for the claim to be true.
\end{proof}

\begin{proof}[Proof of {\Cref{thm:nbh_criterion}}.]
Let $\lambda_j=\lambda_j(\xb_i)$ denotes the $j$-th smallest eigenvalue of $A(\xb_i)$. 
For $d=2$, we apply \Cref{lem:min_radius} with $\phi = \pi/8$ on a rescaled ellipse of $T_i(B_\del(\xb_i))$, we get 
\[
\min_{\theta \in [0, \frac{\pi}{2}]} r(\vb(\theta),\xb_i)  \geq C_1 \del \sqrt{\lambda_2}\sqrt{\frac{\lambda_1}{\lambda_2}} =C_1 \del \sqrt{\lambda_1},
\]
where $C_1 = K(\pi/8)$. 

Now consider $d=3$. 
First of all, we can assume without loss of generality that $\lambda_3=1$ by the method of rescaling. 
By the discussions at the beginning of \Cref{subsec:neighborhood_criteria},  $\cC^{\vb}_\del(\xb_i)$ is a 3d cone with total opening angle $33.7^\circ$ for a given unit vector $\vb\in \R^3$.
Let $P_{\vb} \subset \R^3$ be a 2d plane that contains the vector $\vb$, then we see that $P_{\vb} \cap B_\del(\xb_i)$ is a circular domain and $P_{\vb} \cap \cC^{\vb}_\del(\xb_i)$ is a 2d cone with total opening angle $33.7^\circ$. 
With the transform $T_i$, we see that $T_i (P_{\vb} \cap B_\del(\xb_i))$ is a 2d ellipse and $T_i( P_{\vb} \cap \cC^{\vb}_\del(\xb_i))$ is a section of the ellipse.
Therefore, the 2d calculations in \Cref{lem:min_radius} can be applied. 
Notice that for each $P_{\vb}$, there exists $\rho_1$ and $\rho_2$ with $\lambda_1 \leq \rho_1 \leq \rho_2\leq 1$ such that the lengths of the semi-axes of the ellipse $T_i (P_{\vb} \cap B_\del(\xb_i))$ are given by $\sqrt{\rho_1}$ and $\sqrt{\rho_2}$.
We can then rescale the the ellipse $T_i (P_{\vb} \cap B_\del(\xb_i))$ and use \Cref{lem:min_radius} with $\phi=33.7^\circ/2$ we see that the radius of the inscribed circle in $T_i( P_{\vb} \cap \cC^{\vb}_\del(\xb_i))$  has a lower bound
\[
\tilde{C} (\del \sqrt{\rho_2}) \sqrt{\frac{\rho_1}{\rho_2}} \geq \tilde{C} \del\sqrt{\lambda_1},
\]
where $\tilde{C}  =K(33.7^\circ/2)>0$. 
Notice that $P_{\vb}$ is an arbitrary plane that contains $\vb$, and in addition, the average length of the line segments that connect $T_i(\xb_i)$ and the edge of $T_i( P_{\vb} \cap \cC^{\vb}_\del(\xb_i))$ is at the same scale for different plane $P_{\vb}$. Therefore, it can be shown that there exists $C_2>0$ such that 
\[
\min_{\theta \in [0, \frac{\pi}{2}]} r(\vb(\theta),\xb_i) \geq C_2 \del\sqrt{\lambda_1}. 
\]

At last, by \Cref{lem:appendix_1} and the above discussions, there exists $C=C(d)>0$ such that whenever $h\leq C \del\sqrt{\lambda_1(\xb_i)} $, $S_{\del, h, 2}(\xb_i)$ and $\overline{S}_{\del, h, 2}(\xb_i)$ are not empty. Therefore taking $c(d)=1/C(d)$ the conclusion is true.  
By our assumption, $\varrho\leq \lambda_1(\xb_i)$ for all $\xb_i$, and therefore $\del\geq c h (\varrho)^{-1/2}$ implies the existence of positive stencils. 
\end{proof}

\section{Point cloud generation and adjustment}
\label{sec:appendixB}

In this appendix, we discuss the generation of proper point clouds that satisfy \Cref{def:pointcloud}.
We first initialize a random point cloud using the Quasi-Monte Carlo method \cite{niederreiter1978quasi} (see \cref{fig:generation_initialize}), then adjust this point cloud to make it proper. Adjustment contains three steps in each loop:\\
\indent {\it Step 1.} add points until $h$ satisfies condition $(i)$ (see \cref{fig:generation_add});  \\
\indent {\it Step 2.} map points until $\kappa$ satisfies condition $(iii)$ (see \cref{fig:generation_map}); \\
\indent {\it Step 3.} merge points until $\zeta$ satisfies condition $(ii)$ (see \cref{fig:generation_merge}). \\
Adjustment stops when the conditions in \Cref{def:pointcloud} are satisfied. In practice, it usually takes a few loops to make the point cloud proper.

\begin{figure}[htp]
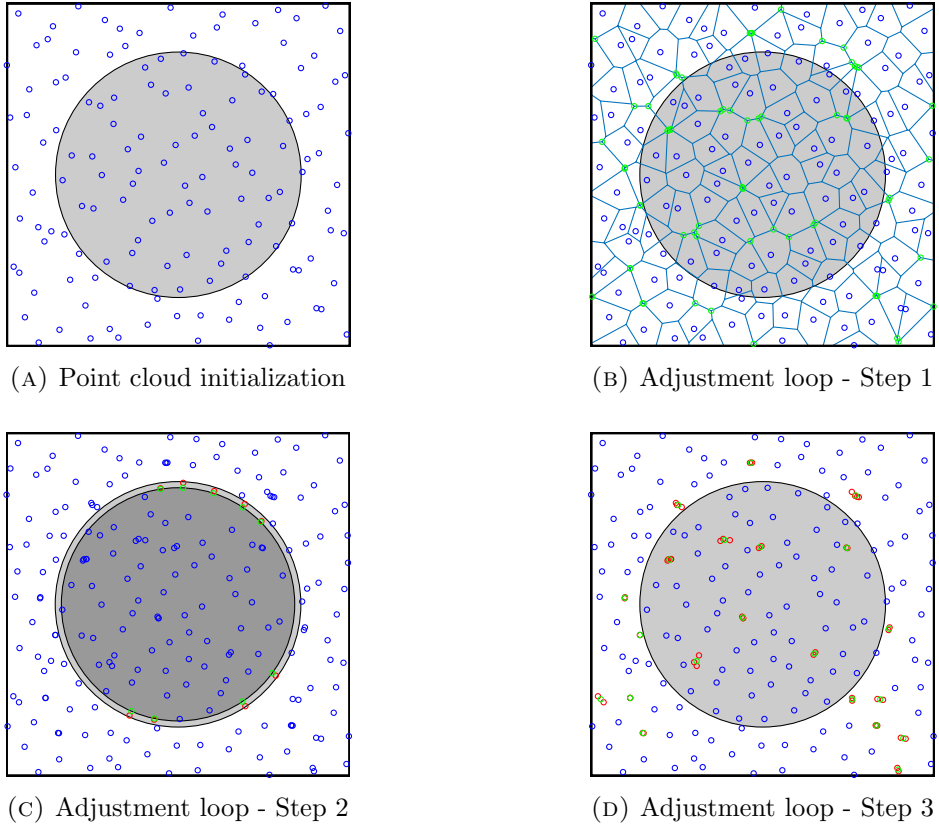

    \centering
    \begin{subfigure}[b]{0.490\textwidth}
        \centering

        \caption{Point cloud initialization}
        \label{fig:generation_initialize}
    \end{subfigure}
    \hfill
    \begin{subfigure}[b]{0.490\textwidth}
        \centering
        %
        \caption{Adjustment loop - Step 1}
        \label{fig:generation_add}
    \end{subfigure}
    \vskip\baselineskip
    \begin{subfigure}[b]{0.490\textwidth}
        \centering
        %
        \caption{Adjustment loop - Step 2}
        \label{fig:generation_map}
    \end{subfigure}
    \hfill
    \begin{subfigure}[b]{0.490\textwidth}
        \centering
        %
        \caption{Adjustment loop - Step 3}
        \label{fig:generation_merge}
    \end{subfigure}
    \caption{The process of proper point cloud generation. The grey circular domain is $\Om$. The square domain is a bounding box of $\Om_{\del_0}$ for some $\del_0>0$. 
    (A): initialize a point cloud by the Quasi-Monte Carlo method. (B):  use the Voronoi diagram \cite{voronoi1908nouvelles_b,voronoi1908nouvelles_a}
    for the calculation of the fill distance, and then add the green points to the point cloud so that $h$ satisfies condition $(i)$. (C): map points near the boundary of $\Om$ to the interior so that $\kappa$ satisfies condition $(iii)$. (D): merge points whose distances are less than $2 c_\zeta h$ so that $\zeta$ satisfies condition $(ii)$. Notice that after merging of points, the fill distance may increase, as a result, the adjustment loop may be needed again.}
    \label{fig:generation}
\end{figure}

\begin{remark}
Since the domain $\Om_\del$ may be irregular, in practice, we always generate point clouds on a larger bounding box of $\Om_{\del_0}$ for $\del\in (0,\del_0]$, as indicated by 
\cref{fig:generation_initialize,fig:generation_add,fig:generation_map,fig:generation_merge}. 
The formulas for the fill distance in \cref{eq:filldistance} and condition (i) \Cref{def:pointcloud} are then modified accordingly.
In practice, $\del_0$ is determined by the largest discretization parameter as well as the ratio $\varrho$ as indicated by neighborhood estimate in \Cref{subsec:neighborhood_criteria} and \Cref{subsec:searching}. 
\end{remark}

\begin{remark}
One can also use the Quasi-Monte Carlo method for generating the initial point cloud and perform only step 2 without the adding and merging steps.
Our adjustment algorithm provides explicit control over the fill distance and the separation distance and it leads to smaller fill distances for the same number of interior points compared with point clouds without adjustment. This is a trade-off situation, namely, one can save memory by using extra time adjusting the initialized point cloud, or vice versa.
\end{remark}


\section{Inscribed circle search algorithm}
\label{sec:appendixC}
We show the details of finding the radius of the inscribed circle in $T_{i}(\mathcal{C}_{1}^{\boldsymbol{v}}(\boldsymbol{x}_{i}))$ that contained in an ellipse given by $x^{2} / \varrho + y^{2} = 1$.
Let $(x_{0}, y_{0})$ be the center of the inscribed circle with radius $r$, then it can only sit on the angle bisector of $T_{i}(\mathcal{C}_{1}^{\boldsymbol{v}}(\boldsymbol{x}_{i}))$.
Let $(x_{1}, y_{1})$ and $(x_{2}, y_{2})$ represent the two corner points of $T_{i}(\mathcal{C}_{1}^{\boldsymbol{v}}(\boldsymbol{x}_{i}))$. 
Then by some elementary calculations, there exists some $t>0$ such that
\begin{align*}
(x_{0}, y_{0}) = t \left ( \sqrt{x_{2}^{2} + y_{2}^{2}} x_{1} + \sqrt{x_{1}^{2} + y_{1}^{2}} x_{2}, \sqrt{x_{2}^{2} + y_{2}^{2}} y_{1} + \sqrt{x_{1}^{2} + y_{1}^{2}} y_{2} \right )
\end{align*}
and
\begin{align*}
r = t |x_{1} y_{2} - y_{1} x_{2}|.
\end{align*}
To determine $t>0$, we find the closest point to the circle center on the ellipse and choose $t>0$ such that the point is also on the circle.  
The closest point to the circle center on the ellipse can be found by the minimization problem
\begin{align*}
\min_{\theta} \| (\sqrt{\varrho} \cos(\theta), \sin(\theta)) - (x_{0}, y_{0}) \|_{2}^{2},
\end{align*}
which can be solved by, e.g., Newton's method. 
At last, one may use a numerical method, e.g., the bisection method, to determine $t>0$ such that the point is also on the circle.


\bibliographystyle{abbrv} 

\bibliography{ref}

\end{document}